\documentclass[11pt,reqno]{amsart}

\usepackage{geometry}
\usepackage[utf8]{inputenc}
\usepackage{amstext,latexsym,amsbsy,amsmath,amssymb,amsthm,mathtools,relsize,geometry,enumerate}
\usepackage{hyperref}
\hypersetup{
	colorlinks   = true, 
	urlcolor     = blue, 
	linkcolor    = red, 
	citecolor   = red 
}
\usepackage{mathtools, enumerate,enumitem}

\usepackage{mathrsfs}

\usepackage{comment}
\usepackage{graphicx}
\usepackage{epstopdf}
\setkeys{Gin}{width=\linewidth,totalheight=\textheight,keepaspectratio}
\graphicspath{{./images/}}

\usepackage[normalem]{ulem}

\usepackage{multicol}
\usepackage{booktabs}
\usepackage{mathrsfs}
\usepackage{color}

\numberwithin{equation}{section}

\newtheorem{remark}{Remark}[section]
\newtheorem{theorem}{Theorem}[section]
\newtheorem{corollary}{Corollary}[section]
\newtheorem{lemma}{{Lemma}}[section]

\allowdisplaybreaks[4]

\newcommand{\cP}{\mathcal{P}}

\newcommand{\mP}{\mathbb{P}}

\newcommand{\mE}{\mathbb{E}}

\newcommand{\Ome}{\Omega}

\newcommand{\nab}{\nabla}

\newcommand{\vH}{{\bf H}}

\newcommand{\vv}{{\bf v}}

\newcommand{\e}{\pmb{\varepsilon}}

\newcommand{\pphi}{\pmb{\phi}}

\newcommand{\vvarphi}{\pmb{\varphi}}
\newcommand{\vsigma}{\pmb{\sigma}}
\newcommand{\vtheta}{\pmb{\theta}}

\begin{document}
	
	\title{Optimal strong convergence of a splitting Mixed finite element method for a Stochastic Keller--Segel System}

	\author{
		Liet Vo$^1$ 
	}

	\thanks{$^1$ School of Mathematical and Statistical Sciences, The University of Texas Rio Grande Valley, Edinburg, TX 78539, U.S.A.  ({\tt liet.vo@utrgv.edu}). This author was partially supported by the NSF grant DMS-2530211.}

\begin{abstract}
	In this work, we develop and analyze a splitting mixed finite element method, combined with a fully implicit Euler time discretization, for a stochastic Keller--Segel system with general polynomial logistic growth and multiplicative noise. The proposed splitting formulation allows the use of standard finite element spaces without requiring the LBB inf--sup condition associated with traditional mixed finite element methods and, at each time step, avoids the need for solving a coupled nonlinear system. We first establish the well-posedness of the variational solution and derive new high-moment, higher-order regularity, and exponential stability estimates. These estimates provide the key analytical foundation for the strong convergence analysis of the numerical method. A principal contribution of this work is the derivation of error estimates directly in full expectation, without the use of localization arguments. By combining the new stability and regularity estimates with a stochastic Gronwall framework, we overcome the difficulties arising from the interaction between the nonlinear chemotactic coupling and the multiplicative noise. As a result, we establish strong convergence in high-moment norms and obtain optimal convergence rates with respect to both the temporal and spatial discretizations. In particular, the proposed method achieves the optimal rates of order $k^{\frac{1}{2}}$ in time and $h$ in space in the relevant $L^2$- and $H^1$-based error norms. These results provide, to our knowledge, the first optimal strong convergence analysis in full expectation for a splitting mixed finite element approximation of this class of stochastic Keller--Segel systems.
\end{abstract}
	\maketitle

	
	\medskip
	

	\section{Introduction}\label{sec-1}
	
	We consider a stochastic Keller--Segel model incorporating logistic growth and multiplicative noise:
	\begin{subequations}\label{eq1.1}
		\begin{align}
			du &= \left[\Delta u-\chi\nabla\cdot(u\nabla v)+g(u)\right]dt+B(u)\,dW(t)
			&& \text{a.s. in }(0,T)\times D, \label{eq1.1a}\\
			\Delta v &= v-u
			&& \text{a.s. in }(0,T)\times D, \label{eq1.1b}\\
			u(0)&=u_0
			&& \text{a.s. in }D.
		\end{align}
	\end{subequations}
	Here, $D=(0,L)^2\subset\mathbb{R}^2$ is a periodic domain and $T>0$. The variables $u$ and $v$ represent the cell density and the concentration of the chemoattractant, respectively. We impose periodic boundary conditions on both $u$ and $v$ in the spatial variables. The parameter $\chi>0$ denotes the chemotactic sensitivity, and $\{W(t)\}_{t\geq0}$ is a real-valued Wiener process. Throughout this paper, we consider the polynomial logistic reaction $g(u)=u-u^q,$
	where $q\geq3$ is an odd integer.
	
	Chemotaxis describes the directed motion of cells or microorganisms along gradients of chemical substances. It arises in a wide variety of biological processes, including bacterial aggregation, wound healing, embryonic development, immune response, angiogenesis, and tumor invasion. A classical mathematical framework for chemotaxis was introduced by Keller and Segel \cite{keller1970initiation,keller1971model}. Since then, Keller--Segel systems have become fundamental models for studying the competition between diffusion and chemically induced aggregation, with particular attention devoted to global existence, boundedness, pattern formation, and finite-time blow-up; see, for example, \cite{sherratt1994chemotaxis,arumugam2021keller,herrero1996chemotactic} and the references therein.
	
	A widely studied extension of the classical Keller--Segel model incorporates a logistic reaction term to describe population growth and crowding effects. From a biological perspective, this mechanism reflects the suppression of population growth resulting from factors such as competition for resources and limitations in nutrients or space. From a mathematical perspective, the dissipative effect of the logistic nonlinearity counteracts the aggregation induced by chemotaxis and may therefore suppress the formation of singularities. Consequently, Keller--Segel models with logistic growth involve a delicate balance among diffusion, chemotactic attraction, population growth, and nonlinear damping. Considerable attention has been devoted to the qualitative analysis of such systems, including global existence, boundedness, and long-time behavior; see, for example, \cite{mimura1996aggregating,tello2007chemotaxis,lankeit2017generalized,cao2014boundedness} and the references therein.
	
	Stochastic effects provide another important component in the mathematical modeling of chemotaxis, as biological systems are naturally influenced by environmental variability, fluctuations in chemical signals, and randomness in individual cell motion. These effects have motivated various stochastic extensions of the Keller--Segel model, in which the uncertainty may be represented, for instance, by transport-type noise or multiplicative stochastic forcing. The mathematical analysis of stochastic chemotaxis systems has received increasing attention, with studies addressing questions such as existence and uniqueness, positivity and conservation properties, global solvability, and blow-up phenomena; see, for example, \cite{misiats2022global,huang2021microscopic,mukherjee2026uniqueness} and the references therein. Closely related to the model considered in this paper, Chen et al.~\cite{chen2025well} investigated a stochastic Keller--Segel system with logistic growth subject to multiplicative noise and proved the global well-posedness of its mild solution.

Compared with the substantial analytical literature, the rigorous numerical analysis of stochastic Keller--Segel systems remains relatively limited. A splitting mixed finite element method for a stochastic Keller--Segel equation driven by transport noise in the Stratonovich sense was analyzed in \cite{vo2025analysis}. More recently, a splitting mixed finite element approximation for the stochastic Keller--Segel system with logistic growth and multiplicative noise was studied in \cite{thieu2026splitting}. In that work, stability and error estimates were derived by means of a localization argument \cite{carelli2012rates}, leading to convergence in probability of the fully discrete approximation. However, optimal strong error estimates in full expectation were not obtained. One of the main difficulties arises from the interaction between the nonlinear and non-monotone chemotactic coupling and the stochastic forcing, which makes it challenging to control the nonlinear error terms globally in expectation. In particular, polynomial stability estimates alone are generally insufficient to close the strong error analysis without introducing suitable localization events. Overcoming this difficulty requires stronger stability, integrability, and regularity properties of both the exact and numerical solutions. This distinction is particularly important in stochastic numerical analysis, since convergence in probability does not by itself provide quantitative moment bounds for the numerical error or strong rates of convergence.

Therefore, the primary purpose of the present paper is to develop a rigorous strong convergence theory for a fully discrete splitting mixed finite element approximation of \eqref{eq1.1}. The analysis is complicated as mentioned above by the simultaneous presence of the nonlinear chemotactic flux $\nabla\cdot(u\nabla v),$ the polynomial reaction $u-u^q$, and the multiplicative stochastic forcing. To overcome the challenge, our approach is based on three closely connected ingredients: high-moment and exponetial estimates of the variational solution, a stable splitting mixed finite element formulation with high-moment and discrete exponential stability estimates, and a stochastic Gronwall framework to control the nonlinear random terms directly in expectation.

We first establish the well-posedness of the variational solution and develop the high-moment estimates required for the numerical analysis. For the spatial approximation, we introduce the auxiliary variable $\vsigma=\nab v$
and employ a splitting mixed formulation that treats the elliptic and parabolic components successively at each time level. The principal numerical result is then obtained by combining the resulting stability estimates with a stochastic Gronwall argument, leading to optimal strong convergence rates in full expectation and in high-moment norms.

More precisely, the main contributions of this paper are summarized as follows:
\begin{enumerate}
	\item We briefly establish the well-posedness of the variational solution of the stochastic Keller--Segel system \eqref{eq1.1} with polynomial logistic growth and multiplicative noise. We first sketching the construction a martingale solution using Galerkin approximations, uniform a priori estimates, and stochastic compactness arguments, and then establish pathwise uniqueness. In addition to well-posedness, we develop a collection of stability and regularity estimates that are essential for the subsequent numerical analysis. These include high-moment estimates, and exponential stability estimate. 
	
	\item We develop and analyze a fully discrete splitting mixed finite element method based on an implicit Euler discretization in time and a mixed finite element discretization in space. By introducing $\vsigma=\nab v$, the chemical gradient appearing in the chemotactic flux is approximated directly, and the elliptic and parabolic components of the system can be treated successively. Following the splitting strategy in \cite{zhang2016characteristic,vo2025analysis,thieu2026splitting}, the resulting formulation avoids the LBB inf--sup compatibility condition required by traditional mixed finite element methods. It therefore allows standard continuous piecewise linear finite element spaces to be used for all unknown variables. Moreover, the splitting structure avoids the solution of a fully coupled nonlinear system at each time step. These features simplify the implementation and are particularly advantageous for stochastic computations requiring a large number of sample paths.
	
\item Most importantly, we establish optimal strong error estimates directly in full expectation and in high-moment norms for the fully discrete approximation. In particular, we obtain strong convergence in $L^p(\Omega;L^\infty(0,T;L^2(D)))$ for $p\ge 1$, and in $L^2(\Omega;L^2(0,T;H^1(D)))$, with the optimal convergence rate $O(k^{1/2}+h)$. These theoretical rates are consistent with the numerical experiments conducted for the same splitting mixed finite element scheme in \cite[Section~4]{thieu2026splitting}, providing numerical evidence of their optimality.
\end{enumerate}

The above results provide a unified analytical and numerical framework for the stochastic Keller--Segel system considered in this paper.

	The remainder of the paper is organized as follows. In Section~\ref{sec2}, we introduce the notation and assumptions. In Section ~\ref{sec3}, we establish the well-posedness, in which its proof will be delayed until Appendix \ref{appendixB}. We also show stability, and temporal regularity properties of the variational solution. Furthermore, we derive the mixed variational formulation used for the numerical approximation. In Section~\ref{sec4}, we present the fully discrete splitting mixed finite element method and establish its stability and optimal strong error estimates. Finally, Section~\ref{sec5} contains the conclusion of the paper.
	
	\section{Prelimimaries}\label{sec2}
\subsection{Notations}

Standard notation for function spaces will be adopted throughout this paper. Since $D=(0,L)^2$ is a periodic domain, all function spaces over $D$, including $L^p(D)$, $W^{m,p}(D)$, $H^m(D)$, and their vector-valued counterparts, are understood to consist of functions satisfying periodic boundary conditions in the spatial variables whenever traces are defined. For simplicity, we suppress the subscript ``{\rm per}'' throughout the paper.

Let $(\cdot,\cdot)$ denote the standard $L^2(D)$ inner product with induced norm $\|\cdot\|_{L^2}$. Throughout the paper, $C$ denotes a generic positive constant independent of the mesh parameters $h$ and $k$, whose value may change from line to line.

Let $(\Omega,\mathcal F,\{\mathcal F_t\}_{t\geq0},\mathbb P)$ be a filtered probability space satisfying the usual conditions, where $\mathbb P$ is the probability measure, $\mathcal F$ is the $\sigma$-algebra, and $\{\mathcal F_t\}_{t\geq0}\subset\mathcal F$ is a filtration. For a random variable $v$ defined on $(\Omega,\mathcal F,\{\mathcal F_t\},\mathbb P)$, we denote its expectation by $\mathbb E[v]$.

For a Banach space $X$ with norm $\|\cdot\|_X$ and $1\leq p<\infty$, we define the Bochner space
\[
L^p(\Omega;X):=\left\{v:\Omega\to X:\mathbb E\bigl[\|v\|_X^p\bigr]<\infty\right\},
\]
equipped with the norm
\[
\|v\|_{L^p(\Omega;X)}:=\bigl(\mathbb E[\|v\|_X^p]\bigr)^{1/p}.
\]

We also make use of the Ladyzhenskaya inequality in two space dimensions,
\[
\|u\|_{L^4(D)}\leq C_L\|u\|_{L^2(D)}^{1/2}\|u\|_{H^1(D)}^{1/2},
\]
where $C_L>0$ depends only on the domain $D$.

For periodic vector fields $\vsigma\in\mathbf H^1(D)$, integration by parts and periodicity yield the identity
\begin{align}
	\|\nab\vsigma\|_{L^2}^2
	=
	\|\nab\cdot\vsigma\|_{L^2}^2
	+
	\|\operatorname{rot}\vsigma\|_{L^2}^2.
\end{align}
Consequently, we shall use the equivalent representation of the $H^1$-norm
\begin{align}
	\|\vsigma\|_{H^1}^2
	=
	\|\vsigma\|_{L^2}^2
	+
	\|\nab\cdot\vsigma\|_{L^2}^2
	+
	\|\operatorname{rot}\vsigma\|_{L^2}^2,
	\qquad
	\forall\,\vsigma\in\mathbf H^1(D).
\end{align}

%

By the elliptic regularity estimates for the periodic elliptic problem, for every $1<r<\infty$ and integer $m\geq0$, we have
\begin{align}
	\|w\|_{W^{2,r}(D)}&\leq C_r\|f\|_{L^r(D)}, \label{elliptic_estimate}\\
	\|w\|_{H^{m+2}(D)}&\leq C_m\|f\|_{H^m(D)}. \label{elliptic_estimate_higher}
\end{align}
Moreover, since $D\subset\mathbb R^2$, for $r>2$, the Sobolev embedding $W^{2,r}(D)\hookrightarrow C^{1,\alpha}(\overline D)$, where $\alpha=1-\frac{2}{r}$, yields
\begin{align}\label{sobolev_embedding}
	\|w\|_{L^\infty(D)}+\|\nab w\|_{L^\infty(D)}
	\leq C_r\|f\|_{L^r(D)}.
\end{align}
	\subsection{Assumptions}
	
	Concerning the stochastic forcing, in this paper, we will make the following assumptions on $B$.
	\begin{enumerate}[label=(\Alph*)]
		
		\item For $m = 0 ,1$, suppose that $B: H^{m}(D) \rightarrow H^m(D)$. Moreover, there exists a constant $C_{B}>0$ such that
		\begin{align}\label{Assump_Lipschitz}
			\|B(u) - B(v)\|_{H^m} \leq C_{B}\|u-v\|_{H^m}.
		\end{align}
		
		In addition, there exists a constant $c_B>0$ such that
		\begin{align}\label{Assump_LinearGrowth}
			\|B(u)\|_{H^m} \leq c_B (1 + \|u\|_{H^m}).
		\end{align}
		\item There exists a constant $L_0>0$ such that
		\begin{align}\label{Assump_bounded}
			\|B(u)\|_{L^2} \leq L_0.
		\end{align}
		
	\end{enumerate}

	\begin{remark}
		For simplicity of presentation, we restrict our attention to a real-valued Wiener process. The proposed splitting mixed finite element method and the corresponding error analysis can be extended to stochastic Keller--Segel systems driven by more general $Q$-Wiener or cylindrical Wiener processes through the Karhunen--Lo\`eve expansion; see, for example, \cite[Theorem~10.7]{lord2014introduction}. The extension requires only routine modifications based on the Hilbert-space formulation of stochastic integration, and the overall proof strategy remains unchanged.
	\end{remark}

	\section{PDE results}\label{sec3}
	
	We begin this subsection by presenting the variational formulation of \eqref{eq1.1} in the following theorem, in which its proof can be founded in Apprendix \ref{appendixB}.

\begin{theorem}\label{Thm_wellposedness}
	Assume that $B$ satisfies \eqref{Assump_Lipschitz} and \eqref{Assump_LinearGrowth} with $m=0$.  If $u_0\in L^2(\Omega;L^2(D))$, then there exists a unique variational solution $(u,v)$ of \eqref{eq_2.1} such that
	\begin{align}
		u&\in L^2\left(\Omega;C([0,T];L^2(D))\right)\cap L^2\left(\Omega;L^2(0,T;H^1(D))\right)\cap L^{q+1}\left(\Omega;L^{q+1}(0,T; L^{q+1}(D))\right),\\
		v&\in L^2\left(\Omega;C([0,T];H^2(D))\right).
	\end{align}
	Moreover, for every $t\in[0,T]$, the solution satisfies $\mP$-a.s.
	\begin{subequations}\label{eq_2.1}
		\begin{align}\label{eq_2.1a}
			(u(t),\phi)&=(u_0,\phi)-\int_0^t(\nab u(s),\nab\phi)\,ds+\chi\int_0^t(u(s)\nab v(s),\nab\phi)\,ds\\\nonumber
			&\qquad+\int_0^t(g(u(s)),\phi)\,ds+\int_0^t(B(u(s)),\phi)\,dW(s),\\
			&(\nab v(t),\nab\psi)+(v(t),\psi)=(u(t),\psi), \label{eq_2.1b}
		\end{align}
	\end{subequations}
	for all $\phi,\psi\in H^1(D)$.
\end{theorem}
\begin{lemma}\label{lemma_expo_moment_H}
	Suppose that $u_0\in L^2(\Omega;L^2(D))$, and that $B$ satisfies \eqref{Assump_Lipschitz}, \eqref{Assump_LinearGrowth}, and \eqref{Assump_bounded}. Let $\gamma_1>0$ be arbitrary and assume that
	\[
	\mE\left[\exp\left(2\gamma_1\|u_0\|_{L^2}^2\right)\right]<\infty.
	\]
	Then, for every $T>0$, there exists a constant $C_{\exp,1}>0$ such that
	\begin{align}\label{ineq:expo_moment_H}
		\sup_{t\in[0,T]}
		\mE\left[
		\exp\left(
		\gamma_1\|u(t)\|_{L^2}^2
		+2\gamma_1\int_0^t\|\nab u(s)\|_{L^2}^2\,ds
		+\frac{\gamma_1}{2}\int_0^t\|u(s)\|_{L^{q+1}}^{q+1}\,ds
		\right)
		\right]
		\leq C_{\exp,1},
	\end{align}
	where
	\begin{align}
		C_{\exp,1}
		:=
		\left(
		\mE\left[
		\exp\left(
		2\gamma_1\|u_0\|_{L^2}^2
		+2C(\chi,q,L_0,\gamma_1,D)T
		\right)
		\right]
		\right)^{1/2}.
	\end{align}
\end{lemma}
\begin{proof}
	To prove \eqref{ineq:expo_moment_H}, we follow the approach of \cite{nguyen2026fully,feng2024high} using the It\^o formula for $\Phi(u(t))=\|u(t)\|_{L^2}^2$, we obtain
	\begin{align}\label{eq2.4}
		\|u(t)\|_{L^2}^2+2\int_0^t\|\nab u(s)\|_{L^2}^2\,ds
		&=\|u_0\|_{L^2}^2
		+2\chi\int_0^t\bigl(u(s)\nab v(s),\nab u(s)\bigr)\,ds\\\nonumber
		&\qquad+2\int_0^t\bigl(g(u(s)),u(s)\bigr)\,ds
		+2\int_0^t(B(u(s)),u(s))\,dW(s)\\\nonumber
		&\qquad+\int_0^t\|B(u(s))\|_{L^2}^2\,ds\\\nonumber
		&\leq\|u_0\|_{L^2}^2
		+2\chi\int_0^t\bigl(u(s)\nab v(s),\nab u(s)\bigr)\,ds\\\nonumber
		&\qquad+2\int_0^t\|u(s)\|_{L^2}^2\,ds
		-2\int_0^t\|u(s)\|_{L^{q+1}}^{q+1}\,ds\\\nonumber
		&\qquad+2\int_0^t(B(u(s)),u(s))\,dW(s)+L_0^2T,
	\end{align}
	where we have used $g(u)=u-u^q$ and the boundedness assumption \eqref{Assump_bounded}.
	
	This implies that
	\begin{align}\label{eq2.5}
		&\|u(t)\|_{L^2}^2
		+2\int_0^t\|\nab u(s)\|_{L^2}^2\,ds
		+2\int_0^t\|u(s)\|_{L^{q+1}}^{q+1}\,ds\\\nonumber
		&\qquad\leq\|u_0\|_{L^2}^2
		+2\chi\int_0^t\bigl(u(s)\nab v(s),\nab u(s)\bigr)\,ds
		+2\int_0^t\|u(s)\|_{L^2}^2\,ds\\\nonumber
		&\qquad\qquad+2\int_0^t(B(u(s)),u(s))\,dW(s)+L_0^2T.
	\end{align}
	
	The second term on the right-hand side of \eqref{eq2.5} is estimated as follows. Using integration by parts, the periodic boundary conditions, the Cauchy--Schwarz inequality, and \eqref{elliptic_estimate}, we obtain
	\begin{align}\label{eq2.6}
		2\chi\int_0^t\bigl(u(s)\nab v(s),\nab u(s)\bigr)\,ds
		&=-\chi\int_0^t\bigl(u^2(s),\Delta v(s)\bigr)\,ds\\\nonumber
		&\leq\chi\int_0^t\|u(s)\|_{L^4}^2\|\Delta v(s)\|_{L^2}\,ds\\\nonumber
		&\leq C\chi\int_0^t\|u(s)\|_{L^4}^2\|u(s)\|_{L^2}\,ds\\\nonumber
		&\leq C\chi |D|^{\frac{q-2}{q+1}}
		\int_0^t\|u(s)\|_{L^{q+1}}^3\,ds\\\nonumber
		&\leq\int_0^t\|u(s)\|_{L^{q+1}}^{q+1}\,ds
		+C(\chi,q,D)T.
	\end{align}
	Here, we have used
	\[
	\|u\|_{L^4}\leq
	|D|^{\frac{q-3}{4(q+1)}}\|u\|_{L^{q+1}},
	\qquad
	\|u\|_{L^2}\leq
	|D|^{\frac{q-1}{2(q+1)}}\|u\|_{L^{q+1}},
	\]
	and Young's inequality.
	
	Combining \eqref{eq2.5} and \eqref{eq2.6}, we obtain
	\begin{align}\label{eq2.7}
		&\|u(t)\|_{L^2}^2
		+2\int_0^t\|\nab u(s)\|_{L^2}^2\,ds
		+\int_0^t\|u(s)\|_{L^{q+1}}^{q+1}\,ds\\\nonumber
		&\qquad\leq\|u_0\|_{L^2}^2
		+2\int_0^t(B(u(s)),u(s))\,dW(s)
		+2\int_0^t\|u(s)\|_{L^2}^2\,ds
		+C(\chi,q,D,L_0)T.
	\end{align}
	Using Young's inequality, we further have
	\begin{align}\label{eq2.8}
		2\int_0^t\|u(s)\|_{L^2}^2\,ds
		\leq\frac14\int_0^t\|u(s)\|_{L^{q+1}}^{q+1}\,ds
		+C(q,D)T.
	\end{align}
	Substituting \eqref{eq2.8} into \eqref{eq2.7} and multiplying the resulting inequality by $\gamma_1>0$, we obtain
	\begin{align}\label{eq2.9}
		&\gamma_1\left(
		\|u(t)\|_{L^2}^2
		+2\int_0^t\|\nab u(s)\|_{L^2}^2\,ds
		+\frac34\int_0^t\|u(s)\|_{L^{q+1}}^{q+1}\,ds
		\right)\\\nonumber
		&\qquad\leq
		\gamma_1\|u_0\|_{L^2}^2
		+Z_t-\langle Z\rangle_t+\langle Z\rangle_t
		+\gamma_1 C(\chi,q,L_0,D)T,
	\end{align}
	where
	\[
	Z_t:=2\gamma_1\int_0^t(B(u(s)),u(s))\,dW(s).
	\]
	The quadratic variation of $Z_t$ satisfies
	\begin{align}\label{eq2.10}
		\langle Z\rangle_t
		&=4\gamma_1^2\int_0^t(B(u(s)),u(s))^2\,ds\\\nonumber
		&\leq4\gamma_1^2L_0^2\int_0^t\|u(s)\|_{L^2}^2\,ds.
	\end{align}
	Since $q+1>2$, Young's inequality implies that
	\begin{align}
		4\gamma_1^2L_0^2\|u\|_{L^2}^2
		\leq\frac{\gamma_1}{4}\|u\|_{L^{q+1}}^{q+1}
		+C_{\gamma_1,q,L_0,D},
	\end{align}
	where $C_{\gamma_1,q,L_0,D}>0$ is independent of $u$. Consequently,
	\begin{align}\label{ineq:quadratic_variation}
		\langle Z\rangle_t
		\leq\frac{\gamma_1}{4}
		\int_0^t\|u(s)\|_{L^{q+1}}^{q+1}\,ds
		+C_{\gamma_1,q,L_0,D}T.
	\end{align}
	
	Substituting \eqref{ineq:quadratic_variation} into \eqref{eq2.9}, we obtain
	\begin{align}
		&\gamma_1\|u(t)\|_{L^2}^2
		+2\gamma_1\int_0^t\|\nab u(s)\|_{L^2}^2\,ds
		+\frac{\gamma_1}{2}\int_0^t\|u(s)\|_{L^{q+1}}^{q+1}\,ds\\\nonumber
		&\qquad\leq
		\gamma_1\|u_0\|_{L^2}^2
		+Z_t-\langle Z\rangle_t
		+\left[
		\gamma_1 C(\chi,q,L_0,D)
		+C_{\gamma_1,q,L_0,D}
		\right]T.
	\end{align}
	Taking the exponential function gives
	\begin{align}\label{eq2.11}
		&\exp\left(
		\gamma_1\|u(t)\|_{L^2}^2
		+2\gamma_1\int_0^t\|\nab u(s)\|_{L^2}^2\,ds
		+\frac{\gamma_1}{2}\int_0^t\|u(s)\|_{L^{q+1}}^{q+1}\,ds
		\right)\\\nonumber
		&\qquad\leq
		\exp\left(
		\gamma_1\|u_0\|_{L^2}^2
		+\left[
		\gamma_1 C(\chi,q,L_0,D)
		+C_{\gamma_1,q,L_0,D}
		\right]T
		\right)
		\exp\left(Z_t-\langle Z\rangle_t\right).
	\end{align}
	
	Taking the expectation and applying the Cauchy--Schwarz inequality to \eqref{eq2.11}, we obtain
	\begin{align}
		&\mE\left[
		\exp\left(
		\gamma_1\|u(t)\|_{L^2}^2
		+2\gamma_1\int_0^t\|\nab u(s)\|_{L^2}^2\,ds
		+\frac{\gamma_1}{2}\int_0^t\|u(s)\|_{L^{q+1}}^{q+1}\,ds
		\right)
		\right]\\\nonumber
		&\quad\leq
		\left(
		\mE\left[
		\exp\left(
		2\gamma_1\|u_0\|_{L^2}^2
		+2\left[
		\gamma_1 C(\chi,q,L_0,D)
		+C_{\gamma_1,q,L_0,D}
		\right]T
		\right)
		\right]
		\right)^{\frac12}\\\nonumber
		&\qquad\times
		\left(
		\mE\left[
		\exp\left(
		2Z_t-2\langle Z\rangle_t
		\right)
		\right]
		\right)^{\frac12}\\\nonumber
		&=
		\left(
		\mE\left[
		\exp\left(
		2\gamma_1\|u_0\|_{L^2}^2
		+2\left[
		\gamma_1 C(\chi,q,L_0,D)
		+C_{\gamma_1,q,L_0,D}
		\right]T
		\right)
		\right]
		\right)^{\frac12}\\\nonumber
		&\qquad\times
		\left(
		\mE\left[
		\exp\left(
		2Z_t-\frac12\langle2Z\rangle_t
		\right)
		\right]
		\right)^{\frac12}.
	\end{align}
	Since $\langle2Z\rangle_t=4\langle Z\rangle_t$, the process
	\[
	\exp\left(
	2Z_t-\frac12\langle2Z\rangle_t
	\right)
	\]
	is an exponential martingale. Hence,
	\[
	\mE\left[
	\exp\left(
	2Z_t-\frac12\langle2Z\rangle_t
	\right)
	\right]\leq1.
	\]
	Consequently,
	\begin{align}
		&\mE\left[
		\exp\left(
		\gamma_1\|u(t)\|_{L^2}^2
		+2\gamma_1\int_0^t\|\nab u(s)\|_{L^2}^2\,ds
		+\frac{\gamma_1}{2}\int_0^t\|u(s)\|_{L^{q+1}}^{q+1}\,ds
		\right)
		\right]\\\nonumber
		&\qquad\leq
		\left(
		\mE\left[
		\exp\left(
		2\gamma_1\|u_0\|_{L^2}^2
		+2\left[
		\gamma_1 C(\chi,q,L_0,D)
		+C_{\gamma_1,q,L_0,D}
		\right]T
		\right)
		\right]
		\right)^{\frac12}.
	\end{align}
	Therefore,
	\begin{align}
		\sup_{t\in[0,T]}
		\mE\left[
		\exp\left(
		\gamma_1\|u(t)\|_{L^2}^2
		+2\gamma_1\int_0^t\|\nab u(s)\|_{L^2}^2\,ds
		+\frac{\gamma_1}{2}\int_0^t
		\|u(s)\|_{L^{q+1}}^{q+1}\,ds
		\right)
		\right]
		\leq C_{\exp,1},
	\end{align}
	where
	\[
	C_{\exp,1}
	:=
	\left(
	\mE\left[
	\exp\left(
	2\gamma_1\|u_0\|_{L^2}^2
	+2\left[
	\gamma_1 C(\chi,q,L_0,D)
	+C_{\gamma_1,q,L_0,D}
	\right]T
	\right)
	\right]
	\right)^{\frac12}.
	\]
	This completes the proof.
\end{proof}

\begin{remark}
	The boundedness assumption \eqref{Assump_bounded} on the noise coefficient $B$ can be relaxed to the standard linear growth condition \eqref{Assump_LinearGrowth} in Lemma~\ref{lemma_expo_moment_H}. Indeed, the continuous exponential moment estimate \eqref{ineq:expo_moment_H} can still be established under the linear growth condition, although additional restrictions on the exponential parameter $\gamma_1$ may be required. In contrast, the discrete exponential stability estimate \eqref{ineq_exp_fem_stability} in Lemma~\ref{lemma_exp_fem_stability} requires stronger control of the stochastic increments.  Therefore, we impose \eqref{Assump_bounded} in the discrete analysis to establish the exponential stability needed for the strong convergence estimates.
\end{remark}

	{	\begin{lemma}\label{Stability_PDE} Let $(u,v)$ be the variational solution of \eqref{eq_2.1}. Then, let $p\ge 1$, the following estimates are satisfied:
			\begin{enumerate}
				\item[(a)] If 
				$u_0\in L^p(\Omega;L^2(D))$ and suppose that $B$ satisfies the assumptions \eqref{Assump_LinearGrowth} with $m =0$, then there exists a constant $C_{1,p}=C\bigl(p,T,\chi,D,c_B,\mathbb E[\|u_0\|_{L^2}^{2p}]\bigr)>0$
				such that the weak solution $u$ satisfies the following high-moment estimate:
				\begin{align}
					\mathbb E\Bigg[
					\sup_{0\le s\le T}\|u(s)\|_{L^2}^{2p}
					&+
					p\int_0^T
					\|u(s)\|_{L^2}^{2(p-1)}\|\nabla u(s)\|_{L^2}^{2}\,ds
					\Bigg]
					\le C_{1,p} .
				\end{align}
				\item[(b)] If 
				$u_0\in L^p(\Omega;L^2(D))$ and suppose that $B$ satisfies the assumptions \eqref{Assump_LinearGrowth} with $m =0$, then there exists a constant $C_{2,p}=C\bigl(p,T,\chi,D,c_B,\mathbb E[\|u_0\|_{L^2}^{2p}]\bigr)>0$
				such that the weak solution $u$ satisfies the following high-moment estimate:
				\begin{align}\label{eq2.20}
					\mathbb E\left[\left(\int_0^T
					\|\nabla u(s)\|_{L^2}^{2}\,ds
					\right)^p\right] + \mathbb E\left[\left(\int_0^T
					\|u(s)\|_{L^4}^{4}\,ds
					\right)^p\right]
					\le C_{2,p} .
				\end{align}

				\item[(c)] Let $u_0\in L^4(\Omega;L^4(D))$ and suppose that $B$ 
				satisfies \eqref{Assump_LinearGrowth} with $m = 0$ and $m =1$.
				Then there exists a constant $	C_{3,p}
				=
				C\left(
				T,\chi,D,c_B,
				\mathbb E\left[\|u_0\|_{L^4}^{4}\right]
				\right)>0$
				such that the weak solution $u$ satisfies the estimate
				\begin{align}\label{eq2.21}
					\mathbb E\Bigg[
					\sup_{0\le s\le T}\|u(s)\|_{L^4}^{4p}
					\Bigg]
					\le C_{3,p}.
				\end{align}
				\item[(d)] If $u_0\in L^2(\Ome;H^1(D))$, and suppose that $B$ satisfies the assumptions \eqref{Assump_LinearGrowth} with $m = 0, 1$ and $m =2$ . Then, there exists a constant $C_{4,p} = C\left(
				T,\chi,D,c_B,
				\mathbb E\left[\|u_0\|_{H^1}^{2p}\right]
				\right)>0$ such that
				\begin{align}\label{eq_2.22}
					\mE\left[\sup_{s\in [0,T]}\|\nab u(s)\|^{2p}_{L^2} + p\int_{0}^T \|\nab u(s)\|^{2p-2}_{L^2}\|\Delta u(s)\|^2_{L^2}\, ds\right] \leq C_{4,p}.
				\end{align}
				
				\item[(e)] If 
				$u_0\in L^p(\Omega;H^1(D))$ and suppose that $B$ satisfies the assumptions \eqref{Assump_LinearGrowth} with $m=0, 1$, then there exists a constant $C_{5,p}=C\bigl(p,T,\chi,D,c_B,\mathbb E[\|\nab u_0\|_{L^2}^{2p}]\bigr)>0$
				such that the weak solution $u$ satisfies the following high-moment estimate:
				\begin{align}\label{eq2.23}
					\mathbb E\left[\left(\int_0^T
					\|\Delta u(s)\|_{L^2}^{2}\,ds
					\right)^p\right]
					\le C_{5,p} .
				\end{align}
			\item[(f)] If $u_0\in L^{2p}(\Omega;H^2(D))$ and $B$ satisfies the assumption \eqref{Assump_LinearGrowth} with $m=0,1,2$, then there exists a constant
		$C_{6,p}=C\left(p,T,\chi,q,D,c_B,\mE\left[\|u_0\|_{H^2}^{2p}\right]\right)>0$
			such that the weak solution $u$ satisfies the following high-moment estimate:
			\begin{align}\label{eq23}
				\mE\left[\sup_{0\leq s\leq T}\|\Delta u(s)\|_{L^2}^{2p}+p\int_0^T\|\Delta u(s)\|_{L^2}^{2p-2}\|\nab\Delta u(s)\|_{L^2}^2\,ds\right]\leq C_{6,p}.
			\end{align}
				
			\end{enumerate}
			
		\end{lemma}
	}
	
	\begin{proof} First, the proof of (a) can be found in \cite[Lemma 2.1]{thieu2026splitting}. So, we only give the proof of (b), (c), (d) and (e) below. 
		
		(b) To prove \eqref{eq2.20}, we recall the estimate from \eqref{eq2.7}, which implies that
		\begin{align}\label{eq2.24}
			\int_{0}^t \|\nab u(s)\|^2_{L^2}\, ds  + \int_{0}^t \|u(s)\|^{q+1}_{L^{q+1}}\, ds 
			&\leq \|u_0\|^2_{L^2}+ 2\int_0^t (B(u(s)),u(s)) dW(s) \\\nonumber
			&\qquad+\left(2 + c_B^2\right)\int_{0}^t \|u(s)\|^2_{L^2}\, ds + C(\chi,q)|D|T.
		\end{align}
		
		Taking the $pth$-power for $p \geq 1$ and then the expectation to \eqref{eq2.24}, we obtain
		\begin{align}\label{eq2.25}
			&\mE\left[\left(	\int_{0}^T \|\nab u(s)\|^2_{L^2}\, ds \right)^p\right] + \mE\left[\left(	\int_{0}^T \| u(s)\|^{q+1}_{L^{q+1}}\, ds \right)^p\right]
			\\\nonumber
			&\leq C\mE\left[\|u_0\|^{2p}_{L^2}\right]+C\left(2 + c_B^2 \right)^p \mE\left[\left(\int_{0}^T \|u(s)\|^2_{L^2}\, ds\right)^p\right]\\\nonumber
			&\qquad+ C\mE\left[\sup_{t\in[0,T]} \left|\int_0^t (B(u(s)),u(s)) dW(s)\right|^p\right] + C(\chi,q,p,D,T).
		\end{align}
		
		Using the Burkholder-Davis-Gundy inequality and \eqref{Assump_LinearGrowth}, we estimate the last term of \eqref{eq2.25} as follows:
		\begin{align}\label{eq2.26}
			\mE\left[\sup_{t\in[0,T]} \left|\int_0^t (B(u(s)),u(s)) dW(s)\right|^p\right] &\leq C\mE\left[\left(\int_0^T \|B(u(s))\|^2_{L^2}\|u(s)\|^2_{L^2} \,ds\right)^{\frac{p}{2}}\right]\\\nonumber
			&\leq Cc_B^p\mE\left[\left(\int_0^T \|u(s)\|^4_{L^2} \,ds\right)^{\frac{p}{2}}\right]\\\nonumber
			&\leq Cc_B^pT^{p/2}C_{1,p},
		\end{align}
		where the last inequality of \eqref{eq2.26} is obtained by using Lemma \ref{Stability_PDE} (a).
		
		Additionally, using Lemma \ref{Stability_PDE} (a), we also control the second term of \eqref{eq2.25} as follows:
		\begin{align}\label{eq2.27}
			\left(2 + c_B^2\right)^p \mE\left[\left(\int_{0}^T \|u(s)\|^2_{L^2}\, ds\right)^p\right] &\leq \left(2 + c_B^2\right)^p T^p C_{1,p}.
		\end{align}
		
		Combining \eqref{eq2.25}, \eqref{eq2.26} and \eqref{eq2.27} gives us the estimate \eqref{eq2.20}. The proof is complete.
		
		\medskip
		
		(c) 	We jump to prove Part (c). 
		Applying It\^o's formula to $\Phi(u)=\|u\|_{L^4}^{4p}$, we obtain
		\begin{align}\label{eq2.288}
			&	\|u(t)\|_{L^4}^{4p} + 12p \int_{0}^t \|u(s)\|^{4p -4}_{L^4} \|u(s)\nab u(s)\|^2_{L^2}\, ds
			\\\nonumber	&=
			\|u_0\|_{L^4}^{4p}
			-
			4\chi p\int_0^t
			\|u(s)\|_{L^4}^{4p-4}
			\bigl(u^3(s),\nabla\cdot(u(s)\nabla v(s))\bigr)\,ds
			\\\nonumber
			&\quad
			+4p\int_0^t
			\|u(s)\|_{L^4}^{4p-4}
			\bigl(u^3(s), g(u(s))\bigr)\,ds
			\\\nonumber
			&\quad
			+
			4p\int_0^t
			\|u(s)\|_{L^4}^{4p-4}
			\bigl(u^3(s),B(u(s))\bigr)\,dW(s)
			\\\nonumber
			&\quad
			+
			6p\int_0^t
			\|u(s)\|_{L^4}^{4p-4}
			\bigl(u^2(s)B(u(s)),B(u(s))\bigr)\,ds
			\\\nonumber
			&\quad
			+
			8p(p-1)\int_0^t
			\|u(s)\|_{L^4}^{4p-8}
			\bigl|\bigl(u^3(s),B(u(s))\bigr)\bigr|^2\,ds\\\nonumber
			&:= \|u_0\|^{4p}_{L^4} + Z_{1,p} + ... + Z_{5,p}.
		\end{align}
		
		First, using the integration by parts, the Cauchy-Schwarz inequality and Young inequality and \eqref{sobolev_embedding}, we have
		\begin{align}
			Z_{1,p} &= 12\chi p \int_{0}^t \|u(s)\|^{4p-4}_{L^4} \bigl(u^3(s)\nab u(s),\nab v(s)\bigr)\, ds\\\nonumber
			&\leq 6p \int_{0}^t \|u(s)\|^{4p-4}_{L^4} \|u(s)\nab u(s)\|^2_{L^2}\, ds + 6p\chi^2 \int_{0}^t \|u(s)\|^{4p-4}_{L^4} \|u(s)\|^4_{L^4}\|\nab v(s)\|^2_{L^{\infty}}\, ds\\\nonumber
			&\leq 6p \int_{0}^t \|u(s)\|^{4p-4}_{L^4} \|u(s)\nab u(s)\|^2_{L^2}\, ds + C\int_{0}^t \|u(s)\|^{4p-2}_{L^4} \|u(s)\|^4_{L^4}\, ds\\\nonumber
			&\leq 6p \int_{0}^t \|u(s)\|^{4p-4}_{L^4} \|u(s)\nab u(s)\|^2_{L^2}\, ds + C\sup_{s\in[0,T]}\|u(s)\|^{4p-2}_{L^4} \int_{0}^t \|u(s)\|^4_{L^4}\, ds.
		\end{align}
		
		For $p\geq 1$, using the Young inequality we obtain
		\begin{align}
			Z_{1,p}  &\leq 6p \int_{0}^t \|u(s)\|^{4p-4}_{L^4} \|u(s)\nab u(s)\|^2_{L^2}\, ds  +\frac14\sup_{s\in[0,T]}\|u(s)\|^{4p}_{L^4}+ C \left(\int_{0}^t \|u(s)\|^4_{L^4}\, ds\right)^{2p}.
		\end{align}
		
		Using $g(u) = u - u^q$, for odd integers $q \geq 3$, we also have
		\begin{align}
			Z_{2,p} &= 4p\int_{0}^t \|u(s)\|^{4p}_{L^4} \, ds - 4p \int_0^t \|u(s)\|^{4p-4}_{L^4} \|u(s)\|^{q+3}_{L^{q+3}}\, ds.
		\end{align}
		
		Next, using the Burkholder-Davis-Gundy inequality, the Sobolev embedding inequality $\|\phi\|_{L^4} \leq C\|\phi\|_{H^1}$, the assumption \eqref{Assump_LinearGrowth}, and the Young inequality, we also get
		\begin{align}\label{eq2.32}
			\mE\left[Z_{3,p}\right] &\leq 4p \mE\left[\left(\int_{0}^T \|u(s)\|^{8p-8}_{L^4}\|u(s)\|^4_{L^4}\|u(s)B(u(s))\|^2_{L^2}\, ds\right)^{\frac12}\right] \\\nonumber
			&\leq 4p \mE\left[\left(\int_{0}^T \|u(s)\|^{8p-4}_{L^4}\|u(s)\|^2_{L^4}\|B(u(s))\|^2_{L^4}\, ds\right)^{\frac12}\right] \\\nonumber
			&\leq 4p \mE\left[\left(\int_{0}^T \|u(s)\|^{8p-2}_{L^4}\|B(u(s))\|^2_{H^1}\, ds\right)^{\frac12}\right] \\\nonumber
			&\leq 4pc_B \mE\left[\left(\int_{0}^T \|u(s)\|^{8p-2}_{L^4}\left(\|u(s)\|^2_{L^2} + \|\nab u(s)\|^2_{L^2}\right)\, ds\right)^{\frac12}\right] \\\nonumber
			&\leq 4pc_B \mE\left[\sup_{s\in[0,T]}\|u(s)\|^{4p-1}_{L^4}\left(\int_{0}^T \left(\|u(s)\|^2_{L^2} + \|\nab u(s)\|^2_{L^2}\right)\, ds\right)^{\frac12}\right] \\\nonumber
			&\leq \frac14 \mE\left[\sup_{s\in[0,T]}\|u(s)\|^{4p}_{L^4}\right] + C\mE\left[\left(\int_{0}^T \left(\|u(s)\|^2_{L^2} + \|\nab u(s)\|^2_{L^2}\right)\, ds\right)^{2p}\right] \\\nonumber
			&\leq \frac14 \mE\left[\sup_{s\in[0,T]}\|u(s)\|^{4p}_{L^4}\right] + C(C_{1,2p} + C_{2,2p}),
		\end{align}
		where the last inequality of \eqref{eq2.32} is obtained by using Lemma \ref{Stability_PDE} (a) and (b).
		
		Next, using the Young inequality, the Sobolev embedding inequality $\|\phi\|_{L^4} \leq C\|\phi\|_{H^1}$, the assumption \eqref{Assump_LinearGrowth}, we have
		\begin{align}
			Z_{4,p} + Z_{5,p} &\leq 6p \int_0^t \|u(s)\|^{4p-4}_{L^4}\|u(s)\|^2_{L^4}\|B(u(s))\|^2_{L^4}\, ds \\\nonumber
			&\qquad+ 8p(p-1) \int_0^t \|u(s)\|^{4p-8}_{L^4}\|u(s)\|^4_{L^4}\|u(s)B(u(s))\|^2_{L^2}\, ds\\\nonumber
			&\leq 6pCc_B \int_0^t \|u(s)\|^{4p-2}_{L^4}\|u(s)\|^2_{H^1}\, ds \\\nonumber
			&\qquad+ 8p(p-1) \int_0^t \|u(s)\|^{4p-4}_{L^4}\|u(s)\|^2_{L^4}\|B(u(s))\|^2_{L^4}\, ds\\\nonumber
			&\leq 6pCc_B \int_0^t \|u(s)\|^{4p-2}_{L^4}\|u(s)\|^2_{H^1}\, ds \\\nonumber
			&\qquad+ 8p(p-1)Cc_B \int_0^t \|u(s)\|^{4p-2}_{L^4}\|u(s)\|^2_{H^1}\, ds\\\nonumber
			&\leq 14pCc_B \sup_{s\in[0,T]}\|u(s)\|^{4p-2}_{L^4}\int_0^t \left(\|u(s)\|^2_{L^2} + \|\nab u(s)\|^2_{L^2}\right)\, ds \\\nonumber
			&\leq \frac14\sup_{s\in[0,T]}\|u(s)\|^{4p-2}_{L^4} + C\left(\int_0^t \left(\|u(s)\|^2_{L^2} + \|\nab u(s)\|^2_{L^2}\right)\, ds\right)^{2p}.
		\end{align}
		
		Now, collecting and sustituting all the estimates from $Z_{1,p},..., Z_{5,p}$ into the right-hand side of \eqref{eq2.288} with the expectation, we arrive at
		\begin{align}\label{eq2.355}
			&\frac14\mE\left[\sup_{s\in[0,T]}\|u(s)\|_{L^4}^{4p}\right] + 6p \int_{0}^T\mE\left[ \|u(s)\|^{4p -4}_{L^4} \|u(s)\nab u(s)\|^2_{L^2}\right]\, ds
			\\\nonumber	&\leq 
			\mE\left[\|u_0\|_{L^4}^{4p}\right] + C\mE\left[\left(\int_{0}^T \|u(s)\|^4_{L^4}\, ds\right)^{2p}\right] + 4p\int_0^T \mE\left[\sup_{r\in[0,s]}\|u(r)\|^{4p}_{L^4}\right]\, ds \\\nonumber
			&\qquad+ C(C_{1,2p} + C_{2,2p}) + C\mE\left[\left(\int_0^T \left(\|u(s)\|^2_{L^2} + \|\nab u(s)\|^2_{L^2}\right)\, ds\right)^{2p}\right]\\\nonumber
			&\leq 
			\mE\left[\|u_0\|_{L^4}^{4p}\right] + CC_{2,2p}+ 4p\int_0^T \mE\left[\sup_{r\in[0,s]}\|u(r)\|^{4p}_{L^4}\right]\, ds + 2C(C_{1,2p} + C_{2,2p}).
		\end{align}
		
		Lastly, using the Gronwall inequality and \eqref{eq2.355} we establish the estimate \eqref{eq2.21} for any $p\geq 1$. The proof is complete.

		\medskip
		
		(d)	Applying the It\^o formula to $\Phi(u(t)) = \|\nab u(t)\|^{2p}_{L^2}$, 
		we obtain
		\begin{align}\label{eq2.22}
			&	\|\nabla u(t)\|_{L^2}^{2p} + 2p \int_{0}^t \|\nab u(s)\|^{2p-2}_{L^2}\|\Delta u(s)\|^2_{L^2}\, ds
			\\\nonumber	&=
			\|\nabla u_0\|_{L^2}^{2p}
			+2\chi p\int_0^t
			\|\nabla u(s)\|_{L^2}^{2p-2}
			\bigl(\Delta u(s),\nabla\cdot(u(s)\nabla v(s))\bigr)\,ds \\\nonumber
			&\qquad-2p\int_0^t
			\|\nabla u(s)\|_{L^2}^{2p-2}
			\bigl(\Delta u(s), g(u(s))\bigr)\,ds
			\\\nonumber
			&\qquad
			+p\int_0^t
			\|\nabla u(s)\|_{L^2}^{2p-2}
			\|\nabla B(u(s))\|_{L^2}^{2}\,ds
			\\\nonumber
			&\qquad
			+2p(p-1)\int_0^t
			\|\nabla u(s)\|_{L^2}^{2p-4}
			\bigl|(\nabla u(s),\nabla B(u(s)))\bigr|^2\,ds
			\\\nonumber
			&\qquad
			+2p\int_0^t
			\|\nabla u(s)\|_{L^2}^{2p-2}
			\bigl(\nabla u(s),\nabla B(u(s))\bigr)\,dW(s)\\\nonumber
			&:= \|\nab u_0\|^{2p}_{L^2} + L_1 + L_2 + L_3 + L_4 + L_5.
		\end{align}
		
		Now, we estimate the right-hand side \eqref{eq2.22}. First, by using the integration by parts, we have
		\begin{align}\label{eq2.16}
			L_2 &= 2p \int_{0}^t\|\nab u(s)\|^{2p}_{L^2}\, ds -2pq\int_{0}^t \|\nab u(s)\|^{2p-2}_{L^2} \|u^{\frac{q-1}{2}}(s)\nab u(s)\|^2_{L^2}\, ds.
		\end{align}
		
		Next, using the integration by parts and then \eqref{elliptic_estimate} and \eqref{sobolev_embedding}, we also obtain
		\begin{align}\label{eq2.17}
			L_1 &=  2p\chi \int_{0}^t \|\nab u(s)\|^{2p-2}_{L^2} (\Delta u(s),\nab\cdot(u(s)\nab v(s)))\, ds \\\nonumber
			&=2p\chi\int_{0}^t \|\nab u(s)\|^{2p-2}_{L^2}\left[(\Delta u(s),\nab u\nab v) + (\Delta u(s), u(s) \Delta v(s))\right]\, ds\\\nonumber
			&\leq 2p \chi\int_{0}^t  \|\nab u(s)\|^{2p-2}_{L^2} \left[\|\Delta u(s)\|_{L^2}\|\nab u\|_{L^2}\|\nab v\|_{L^{\infty}} + \|\Delta u(s)\|_{L^2}\|u(s)\|_{L^4}\|\Delta v(s)\|_{L^4}\right]\, ds \\\nonumber
			&\leq 2p\chi C\int_{0}^t  \|\nab u(s)\|^{2p-2}_{L^2} \left[\|\Delta u(s)\|_{L^2}\|\nab u\|_{L^2}\|u(s)\|_{L^{4}} + \|\Delta u(s)\|_{L^2}\|u(s)\|_{L^4}\|u(s)\|_{L^4}\right]\, ds\\\nonumber
			&\leq p\int_{0}^t \|\nab u(s)\|^{2p-2}_{L^2}\|\Delta u(s)\|^2_{L^2} \, ds   +  C\int_{0}^t \|\nab u(s)\|^{2p-2}_{L^2} \|\nab u(s)\|^2_{L^2}\|u(s)\|^2_{L^4}\, ds \\\nonumber
			&\qquad+ C\int_{0}^t \|\nab u(s)\|^{2p-2}_{L^2} \|u(s)\|^4_{L^4}\, ds.
		\end{align} 
		
		To estimate $L_3$ and $L_4$, we use the Cauchy-Schwarz inequality and the assumption \eqref{Assump_LinearGrowth} with $m=1$ as follows:
		
		\begin{align}
			L_3 + L_4 &\leq pc_B \int_{0}^{t} \|\nab u(s)\|^{2p-2}_{L^2}\|\nab u(s)\|^{2}_{L^2}\, ds + 2p(p-1)c_B \int_{0}^{t} \|\nab u(s)\|^{2p-4}_{L^2}\|\nab u(s)\|^{4}_{L^2}\, ds \\\nonumber
			&= (2p - 1)pc_B \int_{0}^{t} \|\nab u(s)\|^{2p}_{L^2}\, ds.
		\end{align}
		
		Substituting all the estimates from $L_1, L_2, L_3, L_4$ and $L_5$ into the right-hand side of \eqref{eq2.22} we arrive at
		\begin{align}
			&	\|\nabla u(t)\|_{L^2}^{2p} + p \int_{0}^t \|\nab u(s)\|^{2p-2}_{L^2}\|\Delta u(s)\|^2_{L^2}\, ds\\\nonumber
			&\leq \|\nab u_0\|^{2p}_{L^2}+ C\int_{0}^t \|\nab u(s)\|^{2p-2}_{L^2} \|\nab u(s)\|^2_{L^2}\|u(s)\|^2_{L^4}\, ds + C\int_{0}^t \|\nab u(s)\|^{2p-2}_{L^2} \|u(s)\|^4_{L^4}\, ds\\\nonumber
			&\qquad+ [2p + (2p-1)pc_B] \int_{0}^t\|\nab u(s)\|^{2p}_{L^2}\, ds +2p\int_0^t
			\|\nabla u(s)\|_{L^2}^{2p-2}
			\bigl(\nabla u(s),\nabla B(u(s))\bigr)\,dW(s),
		\end{align}
		which implies that
		\begin{align}\label{eq2.18}
			&	\sup_{t\in[0,T]}\|\nabla u(t)\|_{L^2}^{2p} + p \int_{0}^T \|\nab u(s)\|^{2p-2}_{L^2}\|\Delta u(s)\|^2_{L^2}\, ds\\\nonumber
			&\leq \|\nab u_0\|^{2p}_{L^2}+ C\int_{0}^T \|\nab u(s)\|^{2p-2}_{L^2} \|\nab u(s)\|^2_{L^2}\|u(s)\|^2_{L^4}\, ds + C\int_{0}^T \|\nab u(s)\|^{2p-2}_{L^2} \|u(s)\|^4_{L^4}\, ds\\\nonumber
			&\qquad+ [2p + (2p-1)pc_B] \int_{0}^T\|\nab u(s)\|^{2p}_{L^2}\, ds \\\nonumber
			&\qquad+2p\sup_{t\in[0,T]}\left|\int_0^t
			\|\nabla u(s)\|_{L^2}^{2p-2}
			\bigl(\nabla u(s),\nabla B(u(s))\bigr)\,dW(s)\right|\\\nonumber
			&:= \|\nab u_0\|^{2p}_{L^2} + I_{1,p} + I_{2,p} + I_{3,p} + I_{4,p}.
		\end{align}
		
		Now, we consider the case when $p =1$. Then, using the Young inequality, we obtain
		\begin{align}\label{eq2.28}
			I_{1,1} &= C\int_{0}^T \|\nab u(s)\|^2_{L^2}\|u(s)\|^2_{L^4}\, ds\\\nonumber
			&\leq C\sup_{s\in[0,T]} \|u(s)\|^2_{L^4}\int_{0}^T  \|\nab u(s)\|^2_{L^2}\, ds\\\nonumber
			&\leq \frac14 \sup_{s\in[0,T]} \|u(s)\|^{4}_{L^4} + C\left(\int_{0}^T  \|\nab u(s)\|^2_{L^2}\, ds\right)^{2}.
		\end{align}
		
		By using the Burkholder-Davis-Gundy inequality and then \eqref{Assump_LinearGrowth} and the Young inequality, we obtain
		\begin{align}\label{eq2.29}
			\mE[I_{4,1}] &= 2\mE\left[\sup_{t\in[0,T]}\left|\int_0^t
			\bigl(\nabla u(s),\nabla B(u(s))\bigr)\,dW(s)\right|\right]\\\nonumber
			&\leq C\mE\left[\left(\int_{0}^{T} \|\nab u(s)\|^2_{L^2}\|\nab B(u(s))\|^2_{L^2}\, ds\right)^{\frac12}\right]\\\nonumber
			&\leq Cc_B\mE\left[\left(\int_{0}^{T} \|\nab u(s)\|^2_{L^2}\|\nab u(s)\|^2_{L^2}\, ds\right)^{\frac12}\right]\\\nonumber
			&\leq Cc_B\mE\left[\sup_{s\in[0,T]}\|\nab u(s)\|_{L^2}\left(\int_{0}^{T} \|\nab u(s)\|^2_{L^2}\, ds\right)^{\frac12}\right]\\\nonumber
			&\leq \frac14\mE\left[\sup_{s\in[0,T]}\|\nab u(s)\|^2_{L^2}\right] + C\mE\left[\int_{0}^{T} \|\nab u(s)\|^2_{L^2}\, ds\right].
		\end{align}
		
		Substituting \eqref{eq2.28} and \eqref{eq2.29} into the right-hand side of \eqref{eq2.18} with $p=1$ and the expectation $\mE[\cdot]$, we arrive at
		\begin{align}\label{eq2.30}
			&	\frac34\mE\left[\sup_{t\in[0,T]}\|\nabla u(t)\|_{L^2}^{2}\right] +  \mE\left[\int_{0}^T \|\Delta u(s)\|^2_{L^2}\, ds\right]\\\nonumber
			&\leq \mE\left[\|\nab u_0\|^{2}_{L^2}\right]+\frac14 \mE\left[ \sup_{s\in[0,T]} \|u(s)\|^{4}_{L^4} \right] + C\mE\left[\left(\int_{0}^T  \|\nab u(s)\|^2_{L^2}\, ds\right)^{2}\right] \\\nonumber
			&\qquad+ C\mE\left[\int_{0}^T\|u(s)\|^4_{L^4}\, ds\right]+ (2 + c_B + C) \int_{0}^T\mE\left[\sup_{s\in [0,t]}\|\nab u(s)\|^{2}_{L^2}\right]\, dt\\\nonumber
			&\leq \mE\left[\|\nab u_0\|^{2}_{L^2}\right]+\frac14 C_{4,1}+ C_{2,2} + C_{1,1}  + (2 + c_B + C) \int_{0}^T\mE\left[\sup_{s\in [0,t]}\|\nab u(s)\|^{2}_{L^2}\right]\, dt,
		\end{align}
		where the last step of \eqref{eq2.30} is obtained by using Lemma \ref{Stability_PDE} (a), (b), (c). Next, using the Gronwall inequality, we obtain the estimate \eqref{eq_2.22} with $p =1$.
		
		Now, we consider the cases when $p > 1$. Then, the term $I_{1,p}$ on the right-hand side of \eqref{eq2.18} is estimated as follows: Using the Young inequality, we have
		\begin{align}
			I_{1,p}&= C\int_{0}^T \|\nab u(s)\|^{2p-2}_{L^2} \|\nab u(s)\|^2_{L^2}\|u(s)\|^2_{L^4}\, ds\\\nonumber
			&\leq C\sup_{s\in[0,T]} \|\nab u(s)\|^{2p-2}_{L^2}\|u(s)\|^2_{L^4}\int_{0}^T  \|\nab u(s)\|^2_{L^2}\, ds\\\nonumber
			&\leq \frac14 \sup_{s\in[0,T]} \|\nab u(s)\|^{2p}_{L^2} + C\sup_{s\in[0,T]} \|u(s)\|^{4p}_{L^4} + C\left(\int_{0}^T  \|\nab u(s)\|^2_{L^2}\, ds\right)^{2p}.
		\end{align}
		Similarly, we also have
		\begin{align}
			I_{2,p} &=C\int_{0}^T \|\nab u(s)\|^{2p-2}_{L^2} \|u(s)\|^4_{L^4}\, ds\\\nonumber
			&\leq C\sup_{s\in[0,T]}\|\nab u(s)\|^{2p-2}_{L^2} \int_{0}^T \|u(s)\|^4_{L^4}\, ds\\\nonumber
			&\leq \frac14 \sup_{s\in[0,T]}\|\nab u(s)\|^{2p}_{L^2} + C\left(\int_{0}^T \|u(s)\|^4_{L^4}\, ds\right)^p\\\nonumber
			&\leq \frac14 \sup_{s\in[0,T]}\|\nab u(s)\|^{2p}_{L^2} + C\left(\int_{0}^T \|u(s)\|^4_{L^4}\, ds\right)^p\\\nonumber
			&\leq \frac14 \sup_{s\in[0,T]}\|\nab u(s)\|^{2p}_{L^2} + CT\sup_{s\in[0,T]} \|u(s)\|^{4p}_{L^4}. 
		\end{align}
		
		Next, we also need to use the Burkholder-Davis-Gundy inequality to control $I_4$ for $p>1$. 
		\begin{align}\label{eq2.35}
			\mE[I_{4,p}]  &=2p\mE\left[\sup_{t\in[0,T]}\left|\int_0^t
			\|\nabla u(s)\|_{L^2}^{2p-2}
			\bigl(\nabla u(s),\nabla B(u(s))\bigr)\,dW(s)\right|\right]\\\nonumber
			&\leq C \mE\left[\left(\int_{0}^T \|\nab u(s)\|^{4p-4}_{L^2}\|\nab u(s)\|^2_{L^2}\|\nab B(u(s))\|^2_{L^2} \, ds\right)^{\frac12}\right]\\\nonumber
			&\leq Cc_B \mE\left[\left(\int_{0}^T \|\nab u(s)\|^{4p-4}_{L^2}\|\nab u(s)\|^2_{L^2}\|\nab u(s)\|^2_{L^2} \, ds\right)^{\frac12}\right]\\\nonumber
			&\leq Cc_B \mE\left[\sup_{s\in[0,T]}\|\nab u(s)\|^{2p-1}_{L^2}\left(\int_{0}^T \|\nab u(s)\|^2_{L^2} \, ds\right)^{\frac12}\right]\\\nonumber
			&\leq \frac14 \mE\left[\sup_{s\in[0,T]}\|\nab u(s)\|^{2p}_{L^2}\right] + C\mE\left[\left(\int_{0}^T \|\nab u(s)\|^{2}_{L^2} \, ds\right)^p\right],
		\end{align}
		where the last inequality of \eqref{eq2.35} is obtained by using the Young inequality.
		
		Substituting all the estimates from $I_{1,p},..., I_{4,p}$ into the right-hand of \eqref{eq2.18} with the expectation, we obtain
		\begin{align}\label{eq2.36}
			&\frac14	\mE\left[\sup_{t\in[0,T]}\|\nabla u(t)\|_{L^2}^{2p} \right]+ p \int_{0}^T\mE\left[ \|\nab u(s)\|^{2p-2}_{L^2}\|\Delta u(s)\|^2_{L^2}\right]\, ds\\\nonumber
			&\leq \mE\left[\|\nab u_0\|^{2p}_{L^2}\right] + C\mE\left[\sup_{s\in[0,T]} \|u(s)\|^{4p}_{L^4} + \left(\int_{0}^T  \|\nab u(s)\|^2_{L^2}\, ds\right)^{2p}\right]\\\nonumber
			&\qquad + CT\mE\left[\sup_{s\in[0,T]} \|u(s)\|^{4p}_{L^4}\right] + [2p + (2p-1)pc_B]\mE\left[ \int_{0}^T\|\nab u(s)\|^{2p}_{L^2}\, ds \right] \\\nonumber
			&\qquad+ C\mE\left[\left(\int_{0}^T \|\nab u(s)\|^{2}_{L^2} \, ds\right)^p\right]\\\nonumber
			&\leq \mE\left[\|\nab u_0\|^{2p}_{L^2}\right] + C_{3,p} + C_{2,2p} + TC_{3,p} + C_{2,p} + [2p + (2p-1)pc_B]\mE\left[ \int_{0}^T\|\nab u(s)\|^{2p}_{L^2}\, ds \right],
		\end{align}
		where the last inequality of \eqref{eq2.36} is obtained by using Lemma \ref{Stability_PDE} (b), and (c).
		
		Lastly, using the Gronwall inequality and \eqref{eq2.36} we arrive the estimate \eqref{eq_2.22} for any $p \geq 1$.

		\medskip
		
		(e) The proof of \eqref{eq2.23} is similar to the one in \eqref{eq2.20}. So, we skip this proof to save space.

		\medskip
		
		(f) Applying the It\^o formula to $\Phi(u(t))=\|\Delta u(t)\|_{L^2}^{2p}$, we obtain
		\begin{align}\label{eq_H2_Ito}
			&\|\Delta u(t)\|_{L^2}^{2p}+2p\int_0^t\|\Delta u(s)\|_{L^2}^{2p-2}\|\nab\Delta u(s)\|_{L^2}^2\,ds\\\nonumber
			&=\|\Delta u_0\|_{L^2}^{2p}-2p\chi\int_0^t\|\Delta u(s)\|_{L^2}^{2p-2}\bigl(\Delta u(s),\Delta\nab\cdot(u(s)\nab v(s))\bigr)\,ds\\\nonumber
			&\qquad+2p\int_0^t\|\Delta u(s)\|_{L^2}^{2p-2}\bigl(\Delta u(s),\Delta g(u(s))\bigr)\,ds\\\nonumber
			&\qquad+p\int_0^t\|\Delta u(s)\|_{L^2}^{2p-2}\|\Delta B(u(s))\|_{L^2}^2\,ds\\\nonumber
			&\qquad+2p(p-1)\int_0^t\|\Delta u(s)\|_{L^2}^{2p-4}\bigl|(\Delta u(s),\Delta B(u(s)))\bigr|^2\,ds\\\nonumber
			&\qquad+2p\int_0^t\|\Delta u(s)\|_{L^2}^{2p-2}\bigl(\Delta u(s),\Delta B(u(s))\bigr)\,dW(s)\\\nonumber
			&:= \|\Delta u_0\|_{L^2}^{2p} + Z_1 + Z_2 + Z_3 + Z_4 + Z_5.
		\end{align}
		
		Now, we estimate the right-hand side of \eqref{eq_H2_Ito} as follows. First, using the fact that $g(u) = u - u^q$, we have
		\begin{align}
			Z_2 &= 2p\int_0^t\|\Delta u(s)\|_{L^2}^{2p-2}\bigl(\Delta u(s),\Delta g(u(s))\bigr)\,ds\\\nonumber
			&=2p\int_0^t\|\Delta u(s)\|_{L^2}^{2p}\,ds - 2p\int_0^t\|\Delta u(s)\|_{L^2}^{2p-2}\bigl(\Delta u(s),\Delta(u(s)^q)\bigr)\,ds.
		\end{align}
	With this and using the identity $\Delta(u^q)
	=
	q u^{q-1}\Delta u
	+
	q(q-1)u^{q-2}|\nab u|^2,$ we obtain
		\begin{align}\label{eq2.56}
		Z_2 
		&=2p\int_0^t\|\Delta u(s)\|_{L^2}^{2p}\,ds - 2pq\int_0^t\|\Delta u(s)\|_{L^2}^{2p-2} \|u^{\frac{q-1}{2}} \Delta u(s)\|^2_{L^2}\,ds \\\nonumber
		&\qquad- 2pq(q-1)\int_0^t\|\Delta u(s)\|_{L^2}^{2p-2}\bigl(u^{{q-2}}\Delta u(s),|\nab u(s)|^2\bigr)\,ds.
	\end{align}
The last term of \eqref{eq2.56} can be estimated as follows:
\begin{align}\label{eq2.57}
	&2pq(q-1)\int_0^t\|\Delta u(s)\|_{L^2}^{2p-2}\left|\bigl(u^{q-2}\Delta u(s),|\nab u(s)|^2\bigr)\right|\,ds\\\nonumber
	&\quad\leq 2pq(q-1)\int_0^t\|\Delta u(s)\|_{L^2}^{2p-2}\|u^{\frac{q-1}{2}}\Delta u(s)\|_{L^2}\|u^{\frac{q-3}{2}}|\nab u(s)|^2\|_{L^2}\,ds\\\nonumber
	&\quad\leq pq\int_0^t\|\Delta u(s)\|_{L^2}^{2p-2}\|u^{\frac{q-1}{2}}\Delta u(s)\|_{L^2}^2\,ds+pq(q-1)^2\int_0^t\|\Delta u(s)\|_{L^2}^{2p-2}\int_D|u(s)|^{q-3}|\nab u(s)|^4\,dx\,ds.
\end{align}
Moreover, by the Cauchy--Schwarz inequality, the Sobolev embedding $H^1(D)\hookrightarrow L^{2(q-3)}(D)$, and the Gagliardo--Nirenberg inequality, we have
\begin{align}\label{eq2.58}
	\int_D|u(s)|^{q-3}|\nab u(s)|^4\,dx&\leq \|u(s)\|_{L^{2(q-3)}}^{q-3}\|\nab u(s)\|_{L^8}^4\\\nonumber
	&\leq C\|u(s)\|_{H^1}^{q-3}\|\nab u(s)\|_{L^2}\|\Delta u(s)\|_{L^2}^3\\\nonumber
	&\leq C\|u(s)\|_{H^1}^{q-2}\|\Delta u(s)\|_{L^2}^3.
\end{align}
Consequently,
\begin{align}\label{eq2.59}
	&\int_0^t\|\Delta u(s)\|_{L^2}^{2p-2}\int_D|u(s)|^{q-3}|\nab u(s)|^4\,dx\,ds\\\nonumber
	&\quad\leq C\sup_{0\leq s\leq t}\|\Delta u(s)\|_{L^2}^{2p-1}\int_0^t\|u(s)\|_{H^1}^{q-2}\|\Delta u(s)\|_{L^2}^2\,ds\\\nonumber
	&\quad\leq \frac{1}{16pq(q-1)^2}\sup_{0\leq s\leq t}\|\Delta u(s)\|_{L^2}^{2p}
	+\frac{1}{2p}\left(8C(2p-1)q(q-1)^2\right)^{2p-1}
	\left(\int_0^t\|u(s)\|_{H^1}^{q-2}\|\Delta u(s)\|_{L^2}^2\,ds\right)^{2p}\\\nonumber
	&\quad\leq \frac{1}{16pq(q-1)^2}\sup_{0\leq s\leq t}\|\Delta u(s)\|_{L^2}^{2p}
\\\nonumber
&\qquad	+\frac{1}{2p}\left(8C(2p-1)q(q-1)^2\right)^{2p-1}
	\sup_{0\leq s\leq t}\|u(s)\|_{H^1}^{2p(q-2)}
	\left(\int_0^t\|\Delta u(s)\|_{L^2}^2\,ds\right)^{2p}.
\end{align}
Substituting \eqref{eq2.57} and \eqref{eq2.59} into \eqref{eq2.56}, we obtain
\begin{align}\label{eq2.60}
	Z_2
	&\leq 2p\int_0^t\|\Delta u(s)\|_{L^2}^{2p}\,ds-pq\int_0^t\|\Delta u(s)\|_{L^2}^{2p-2}\|u^{\frac{q-1}{2}}\Delta u(s)\|_{L^2}^2\,ds\\\nonumber
	&\qquad+\frac{1}{16}\sup_{0\leq s\leq t}\|\Delta u(s)\|_{L^2}^{2p}\\\nonumber
	&\qquad+\frac{q(q-1)^2}{2}\left(8C(2p-1)q(q-1)^2\right)^{2p-1}
	\sup_{0\leq s\leq t}\|u(s)\|_{H^1}^{2p(q-2)}
	\left(\int_0^t\|\Delta u(s)\|_{L^2}^2\,ds\right)^{2p}.
\end{align}

Now, we move on to estimate $Z_1$. Using integration by parts and the periodic boundary conditions, we have
\begin{align}\label{eq2.62}
	Z_1&=-2p\chi\int_0^t\|\Delta u(s)\|_{L^2}^{2p-2}\bigl(\Delta u(s),\Delta\nab\cdot(u(s)\nab v(s))\bigr)\,ds\\\nonumber
	&=2p\chi\int_0^t\|\Delta u(s)\|_{L^2}^{2p-2}\bigl(\nab\Delta u(s),\nab\nab\cdot(u(s)\nab v(s))\bigr)\,ds\\\nonumber
	&\leq \frac{p}{4}\int_0^t\|\Delta u(s)\|_{L^2}^{2p-2}\|\nab\Delta u(s)\|_{L^2}^2\,ds+4p\chi^2\int_0^t\|\Delta u(s)\|_{L^2}^{2p-2}\|\nab\nab\cdot(u(s)\nab v(s))\|_{L^2}^2\,ds.
\end{align}
The last term on the right-hand side of $Z_1$ can be estimated as follows: First, using the identity
\begin{align*}
	\nab\nab\cdot(u\nab v)=D^2u\,\nab v+D^2v\,\nab u+(\Delta v)\nab u+u\nab\Delta v,
\end{align*}
we have
\begin{align}
	\|\nab\nab\cdot(u\nab v)\|_{L^2}^2
	&\leq C\|D^2u\|_{L^2}^2\|\nab v\|_{L^\infty}^2+C\|D^2v\|_{L^4}^2\|\nab u\|_{L^4}^2+C\|\Delta v\|_{L^4}^2\|\nab u\|_{L^4}^2+C\|u\|_{L^\infty}^2\|\nab\Delta v\|_{L^2}^2.
\end{align}
By \eqref{elliptic_estimate} with $r=4$, \eqref{sobolev_embedding}, and \eqref{elliptic_estimate_higher} with $m=1$, we have
\begin{align*}
	\|\nab v\|_{L^\infty}+\|D^2v\|_{L^4}+\|\Delta v\|_{L^4}&\leq C\|u\|_{L^4},\\
	\|\nab\Delta v\|_{L^2}&\leq C\|v\|_{H^3}\leq C\|u\|_{H^1}.
\end{align*}
Moreover, using integration by parts and the periodic boundary conditions, we have $\|D^2u\|_{L^2}=\|\Delta u\|_{L^2}$. Consequently, using the Gagliardo--Nirenberg inequalities
\[
\|\nab u\|_{L^4}^2\leq C\|\nab u\|_{L^2}\|\Delta u\|_{L^2},\qquad \|u\|_{L^\infty}^2\leq C\|u\|_{H^1}\|u\|_{H^2},
\]
and the estimate $\|u\|_{H^2}\leq C(\|u\|_{H^1}+\|\Delta u\|_{L^2})$, we obtain
\begin{align}\label{eq2.64}
	\|\nab\nab\cdot(u\nab v)\|_{L^2}^2
	&\leq C\|u\|_{L^4}^2\|\Delta u\|_{L^2}^2+C\|u\|_{L^4}^2\|\nab u\|_{L^4}^2+C\|u\|_{L^\infty}^2\|u\|_{H^1}^2\\\nonumber
	&\leq C\|u\|_{H^1}^2\|\Delta u\|_{L^2}^2+C\|u\|_{H^1}^3\|\Delta u\|_{L^2}+C\|u\|_{H^1}^4\\\nonumber
	&\leq C\left(1+\|u\|_{H^1}^2\right)\|\Delta u\|_{L^2}^2+C\left(1+\|u\|_{H^1}^6\right).
\end{align}
Substituting \eqref{eq2.64} into \eqref{eq2.62}, and using Young's inequality, we obtain
\begin{align}
	Z_1
	&\leq \frac{p}{4}\int_0^t\|\Delta u(s)\|_{L^2}^{2p-2}\|\nab\Delta u(s)\|_{L^2}^2\,ds+C p\chi^2\int_0^t\left(1+\|u(s)\|_{H^1}^2\right)\|\Delta u(s)\|_{L^2}^{2p}\,ds\\\nonumber
	&\qquad+C p\chi^2\int_0^t\left(1+\|u(s)\|_{H^1}^6\right)\|\Delta u(s)\|_{L^2}^{2p-2}\,ds\\\nonumber
	&\leq \frac{p}{4}\int_0^t\|\Delta u(s)\|_{L^2}^{2p-2}\|\nab\Delta u(s)\|_{L^2}^2\,ds+C p\chi^2\sup_{0\leq s\leq t}\|\Delta u(s)\|_{L^2}^{2p-2}\int_0^t\left(1+\|u(s)\|_{H^1}^2\right)\|\Delta u(s)\|_{L^2}^2\,ds\\\nonumber
	&\qquad+C p\chi^2\int_0^t\left(1+\|u(s)\|_{H^1}^6\right)\|\Delta u(s)\|_{L^2}^{2p-2}\,ds\\\nonumber
	&\leq \frac{p}{4}\int_0^t\|\Delta u(s)\|_{L^2}^{2p-2}\|\nab\Delta u(s)\|_{L^2}^2\,ds+\frac{1}{16}\sup_{0\leq s\leq t}\|\Delta u(s)\|_{L^2}^{2p}\\\nonumber
	&\qquad+\frac{1}{p}\left(\frac{16(p-1)}{p}\right)^{p-1}(Cp\chi^2)^p\left(\int_0^t\left(1+\|u(s)\|_{H^1}^2\right)\|\Delta u(s)\|_{L^2}^2\,ds\right)^p\\\nonumber
	&\qquad+\frac{p-1}{p}\int_0^t\|\Delta u(s)\|_{L^2}^{2p}\,ds+\frac{(Cp\chi^2)^p}{p}\int_0^t\left(1+\|u(s)\|_{H^1}^6\right)^p\,ds\\\nonumber
	&\leq \frac{p}{4}\int_0^t\|\Delta u(s)\|_{L^2}^{2p-2}\|\nab\Delta u(s)\|_{L^2}^2\,ds+\frac{1}{16}\sup_{0\leq s\leq t}\|\Delta u(s)\|_{L^2}^{2p}\\\nonumber
	&\qquad+\frac{1}{p}\left(\frac{16(p-1)}{p}\right)^{p-1}(Cp\chi^2)^p\sup_{0\leq s\leq T}\left(1+\|u(s)\|_{H^1}^2\right)^p\left(\int_0^t\|\Delta u(s)\|_{L^2}^2\,ds\right)^p\\\nonumber
	&\qquad+\frac{p-1}{p}\int_0^t\|\Delta u(s)\|_{L^2}^{2p}\,ds+\frac{(Cp\chi^2)^p}{p}\int_0^t\left(1+\|u(s)\|_{H^1}^6\right)^p\,ds.
\end{align}

For $Z_3$ and $Z_4$, using the assumption \eqref{Assump_LinearGrowth} with $m=2$ and the Cauchy--Schwarz inequality, we obtain
\begin{align}
	Z_3+Z_4
	&\leq p\int_0^t\|\Delta u(s)\|_{L^2}^{2p-2}\|\Delta B(u(s))\|_{L^2}^2\,ds\\\nonumber
	&\qquad+2p(p-1)\int_0^t\|\Delta u(s)\|_{L^2}^{2p-2}\|\Delta B(u(s))\|_{L^2}^2\,ds\\\nonumber
	&=p(2p-1)\int_0^t\|\Delta u(s)\|_{L^2}^{2p-2}\|\Delta B(u(s))\|_{L^2}^2\,ds\\\nonumber
	&\leq p(2p-1)c_B\int_0^t\|\Delta u(s)\|_{L^2}^{2p-2}\left(1+\|\Delta u(s)\|_{L^2}^2\right)\,ds\\\nonumber
	&\leq (2p-1)(p-1)c_B\int_0^t\|\Delta u(s)\|_{L^2}^{2p}\,ds+(2p-1)c_B\int_0^t1\,ds\\\nonumber
	&\qquad+p(2p-1)c_B\int_0^t\|\Delta u(s)\|_{L^2}^{2p}\,ds\\\nonumber
	&\leq (2p-1)^2c_B\int_0^t\|\Delta u(s)\|_{L^2}^{2p}\,ds+(2p-1)c_BT.
\end{align}

To estimate $Z_5$, we take the supremum over $t$ and then the expectation, and use the Burkholder--Davis--Gundy inequality as follows:
\begin{align}
	\mE\left[\sup_{0\leq s\leq t}|Z_5|\right]
	&\leq C p\,\mE\left[\left(\int_0^t\|\Delta u(s)\|_{L^2}^{4p-4}\bigl|(\Delta u(s),\Delta B(u(s)))\bigr|^2\,ds\right)^{\frac12}\right]\\\nonumber
	&\leq C p\,\mE\left[\sup_{0\leq s\leq t}\|\Delta u(s)\|_{L^2}^{p}\left(\int_0^t\|\Delta u(s)\|_{L^2}^{2p-2}\|\Delta B(u(s))\|_{L^2}^2\,ds\right)^{\frac12}\right]\\\nonumber
	&\leq \frac{1}{16}\mE\left[\sup_{0\leq s\leq t}\|\Delta u(s)\|_{L^2}^{2p}\right]+4C^2p^2\mE\left[\int_0^t\|\Delta u(s)\|_{L^2}^{2p-2}\|\Delta B(u(s))\|_{L^2}^2\,ds\right]\\\nonumber
	&\leq \frac{1}{16}\mE\left[\sup_{0\leq s\leq t}\|\Delta u(s)\|_{L^2}^{2p}\right]+4C^2p^2c_B\mE\left[\int_0^t\|\Delta u(s)\|_{L^2}^{2p-2}\left(1+\|\Delta u(s)\|_{L^2}^2\right)\,ds\right]\\\nonumber
	&\leq \frac{1}{16}\mE\left[\sup_{0\leq s\leq t}\|\Delta u(s)\|_{L^2}^{2p}\right]+4C^2p^2c_B(2p-1)\int_0^t\mE\left[\|\Delta u(s)\|_{L^2}^{2p}\right]\,ds+4C^2p^2c_B\frac{T}{p}.
\end{align}

Substituting the estimates for $Z_1,\ldots,Z_5$ into \eqref{eq_H2_Ito}, taking the supremum over $0\leq s\leq t$ for any $t\in[0,T]$, and then taking the expectation, we obtain
\begin{align}\label{eq_H2_combined}
	&\frac{13}{16}\mE\left[\sup_{0\leq s\leq t}\|\Delta u(s)\|_{L^2}^{2p}\right]
	+\frac{7p}{4}\mE\left[\int_0^t\|\Delta u(s)\|_{L^2}^{2p-2}\|\nab\Delta u(s)\|_{L^2}^2\,ds\right]\\\nonumber
	&\qquad+pq\,\mE\left[\int_0^t\|\Delta u(s)\|_{L^2}^{2p-2}\|u(s)^{\frac{q-1}{2}}\Delta u(s)\|_{L^2}^2\,ds\right]\\\nonumber
	&\leq \mE\left[\|\Delta u_0\|_{L^2}^{2p}\right]
	+\left(2p+\frac{p-1}{p}+(2p-1)^2c_B+4C^2p^2c_B(2p-1)\right)
	\int_0^t\mE\left[\sup_{0\leq r\leq s}\|\Delta u(r)\|_{L^2}^{2p}\right]\,ds\\\nonumber
	&\qquad+\frac{q(q-1)^2}{2}\left(8C(2p-1)q(q-1)^2\right)^{2p-1}
	\mE\left[\sup_{0\leq s\leq T}\|u(s)\|_{H^1}^{2p(q-2)}
	\left(\int_0^T\|\Delta u(s)\|_{L^2}^2\,ds\right)^{2p}\right]\\\nonumber
	&\qquad+\frac{1}{p}\left(\frac{16(p-1)}{p}\right)^{p-1}(Cp\chi^2)^p
	\mE\left[\sup_{0\leq s\leq T}\left(1+\|u(s)\|_{H^1}^2\right)^p
	\left(\int_0^T\|\Delta u(s)\|_{L^2}^2\,ds\right)^p\right]\\\nonumber
	&\qquad+\frac{(Cp\chi^2)^p}{p}\mE\left[\int_0^T\left(1+\|u(s)\|_{H^1}^6\right)^p\,ds\right]
	+(2p-1)c_BT+4C^2p^2c_B\frac{T}{p}.
\end{align}

Using Lemma~\ref{Stability_PDE}(d) and (e), together with the Cauchy--Schwarz inequality, we bound the remaining terms on the right-hand side of \eqref{eq_H2_combined}. First, we have
\begin{align}
	&\mE\left[\sup_{0\leq s\leq T}\|u(s)\|_{H^1}^{2p(q-2)}
	\left(\int_0^T\|\Delta u(s)\|_{L^2}^2\,ds\right)^{2p}\right]\\\nonumber
	&\quad\leq
	\left\{\mE\left[\sup_{0\leq s\leq T}\|u(s)\|_{H^1}^{4p(q-2)}\right]\right\}^{\frac12}
	\left\{\mE\left[\left(\int_0^T\|\Delta u(s)\|_{L^2}^2\,ds\right)^{4p}\right]\right\}^{\frac12}
	\leq C_{4,2p(q-2)}^{\frac12}C_{5,4p}^{\frac12}.
\end{align}
Similarly,
\begin{align}
	&\mE\left[\sup_{0\leq s\leq T}\left(1+\|u(s)\|_{H^1}^2\right)^p
	\left(\int_0^T\|\Delta u(s)\|_{L^2}^2\,ds\right)^p\right]\\\nonumber
	&\quad\leq
	\left\{\mE\left[\sup_{0\leq s\leq T}\left(1+\|u(s)\|_{H^1}^2\right)^{2p}\right]\right\}^{\frac12}
	\left\{\mE\left[\left(\int_0^T\|\Delta u(s)\|_{L^2}^2\,ds\right)^{2p}\right]\right\}^{\frac12}\\\nonumber
	&\quad\leq C\left(1+C_{4,2p}^{\frac12}\right)C_{5,2p}^{\frac12},
\end{align}
and
\begin{align}
	\mE\left[\int_0^T\left(1+\|u(s)\|_{H^1}^6\right)^p\,ds\right]
	&\leq CT\left(1+\mE\left[\sup_{0\leq s\leq T}\|u(s)\|_{H^1}^{6p}\right]\right)
	\leq CT\left(1+C_{4,3p}\right).
\end{align}
Therefore, \eqref{eq_H2_combined} implies
\begin{align}
	&\frac{13}{16}\mE\left[\sup_{0\leq s\leq t}\|\Delta u(s)\|_{L^2}^{2p}\right]
	+\frac{7p}{4}\mE\left[\int_0^t\|\Delta u(s)\|_{L^2}^{2p-2}\|\nab\Delta u(s)\|_{L^2}^2\,ds\right]\\\nonumber
	&\qquad+pq\,\mE\left[\int_0^t\|\Delta u(s)\|_{L^2}^{2p-2}\|u(s)^{\frac{q-1}{2}}\Delta u(s)\|_{L^2}^2\,ds\right]\\\nonumber
	&\leq \mE\left[\|\Delta u_0\|_{L^2}^{2p}\right]
	+\frac{q(q-1)^2}{2}\left(8C(2p-1)q(q-1)^2\right)^{2p-1}C_{4,2p(q-2)}^{\frac12}C_{5,4p}^{\frac12}\\\nonumber
	&\qquad+\frac{1}{p}\left(\frac{16(p-1)}{p}\right)^{p-1}(Cp\chi^2)^pC\left(1+C_{4,2p}^{\frac12}\right)C_{5,2p}^{\frac12}
	+\frac{(Cp\chi^2)^p}{p}CT\left(1+C_{4,3p}\right)\\\nonumber
	&\qquad+(2p-1)c_BT+4C^2p^2c_B\frac{T}{p}\\\nonumber
	&\qquad+\left(2p+\frac{p-1}{p}+(2p-1)^2c_B+4C^2p^2c_B(2p-1)\right)
	\int_0^t\mE\left[\sup_{0\leq r\leq s}\|\Delta u(r)\|_{L^2}^{2p}\right]\,ds.
\end{align}

Hence, an application of Gronwall's inequality yields
\begin{align}
	\mE\left[\sup_{0\leq s\leq T}\|\Delta u(s)\|_{L^2}^{2p}\right] + \mE\left[\int_0^T\|\Delta u(s)\|_{L^2}^{2p-2}\|\nab\Delta u(s)\|_{L^2}^2\,ds\right]\leq C_{6,p},
\end{align}
where $C_{6,p}=C\left(p,T,\chi,q,D,c_B,\mE[\|u_0\|_{H^2}^{2p}]\right)>0$. This completes the proof.

	\end{proof}

\begin{remark}
	The periodic boundary conditions in \eqref{eq1.1} are imposed primarily to establish the higher-order regularity estimate in Lemma~\ref{Stability_PDE}(f). The remaining estimates in Lemma~\ref{Stability_PDE} can also be established under homogeneous Neumann boundary conditions. Consequently, the temporal H\"older continuity estimate in Lemma~\ref{Lemma_Holder}(b), which relies on the higher-order regularity estimate, also requires periodic boundary conditions.
\end{remark}
	
	\subsection{Variational mixed formulation}
To control the second-order nonlinear term in \eqref{eq_2.1a}, we introduce the new variable $\vsigma=\nab v$. Using \eqref{eq_2.1b}, we obtain
\begin{align}
	\bigl(\vsigma,\vvarphi\bigr)+\bigl(v,\nab\cdot\vvarphi\bigr)&=0, \label{eq2.1}\\
	\bigl(\nab\cdot\vsigma,\psi\bigr)+(u,\psi)&=\bigl(v,\psi\bigr), \label{eq2.2}
\end{align}
for all $\vvarphi\in\vH^1(D)$ and $\psi\in L^2(D)$, where integration by parts in \eqref{eq2.1} produces no boundary term due to the periodic boundary conditions. Next, choosing $\psi=\nab\cdot\vvarphi$ in \eqref{eq2.2} and substituting the resulting identity into \eqref{eq2.1}, and adding the zero term $\bigl(\operatorname{rot}\vsigma,\operatorname{rot}\vvarphi\bigr)$, since $\operatorname{rot}\vsigma=0$, we obtain the following mixed variational formulation: Find
$(u,\vsigma,v)\in \left[L^2(\Ome;C([0,T];L^2(D)))\cap L^2(\Ome;L^2(0,T;H^1(D)))\right]\newline\times L^2(\Ome;C([0,T];\vH^1(D)))\times L^2(\Ome;C([0,T];L^2(D)))$
such that $\mP$-a.s.
\begin{subequations}\label{Weak_formulation}
	\begin{align}\label{eq_sigma}
		\bigl(\vsigma,\vvarphi\bigr)+\bigl(\nab\cdot\vsigma,\nab\cdot\vvarphi\bigr)+\bigl(\operatorname{rot}\vsigma,\operatorname{rot}\vvarphi\bigr)&=-\bigl(u,\nab\cdot\vvarphi\bigr), \\
	 \label{eq_u}	\bigl(u(t),\varphi\bigr)+\int_0^t\bigl(\nab u,\nab\varphi\bigr)\,ds&=\bigl(u_0,\varphi\bigr)+\chi\int_0^t\bigl(u\vsigma,\nab\varphi\bigr)\,ds+\int_0^t\bigl(g(u),\varphi\bigr)\,ds\\\nonumber
		&\qquad+\left(\int_0^tB(u)\,dW(s),\varphi\right),\\
		\bigl(v,\psi\bigr)&=\bigl(\nab\cdot\vsigma,\psi\bigr)+(u,\psi), \label{eq_c}
	\end{align}
\end{subequations}
for all $(\varphi,\vvarphi,\psi)\in H^1(D)\times\vH^1(D)\times L^2(D)$.
	
	Next, we state the following H\"older continuity estimates for the variational solutions. Their proofs can be found similar to the proof \cite[Lemma 2.1]{thieu2026splitting} with a little modification on the nonlinear logistic growth. To save space, we skip the proof.

\begin{lemma}\label{Lemma_Holder}
	Let $(u,\vsigma,v)$ be the solution of
	\eqref{Weak_formulation}.
	Then the following estimates hold.
	
	\begin{enumerate}
		\item[(a)] Let $p\geq1$. Assume that $B$ satisfies
		\eqref{Assump_LinearGrowth} with $m=0$ and
		\[
		u\in L^{2p}(\Omega;C([0,T];H^1(D)))
		\cap L^{2pq}(\Omega;L^{2pq}(0,T;H^1(D))).
		\]
		Then, for all $0\leq s\leq t\leq T$,
		\begin{align}
			\mE\left[
			\|u(t)-u(s)\|_{L^2}^{2p}
			\right]
			&\leq \widetilde C_{1,p}|t-s|^p,
			\label{eq_holder_L2}\\
			\mE\left[
			\|\vsigma(t)-\vsigma(s)\|_{H^1}^{2p}
			\right]
			&\leq \widetilde C_{1,p}|t-s|^p,
			\label{eq_holder_sigma}
		\end{align}
		where
		\begin{align*}
			\widetilde C_{1,p}
			:=C\left(
			1+\sup_{0\leq r\leq T}
			\mE\left[\|u(r)\|_{H^1}^{2p}\right]+\int_0^T\mE\left[
			\|u(r)\|_{L^4}^{2p}
			\|\nab u(r)\|_{L^2}^{2p}
			+\|u(r)\|_{L^4}^{4p}
			+\|u(r)\|_{L^{2q}}^{2pq}
			\right]\,dr\right).
		\end{align*}
		
		\item[(b)] Let $p\geq1$. Assume that $B$ satisfies
		\eqref{Assump_LinearGrowth} with $m=1$ and
		\[
		u\in L^{2p}(\Omega;L^\infty(0,T;H^2(D)))
		\cap L^{2pq}(\Omega;L^\infty(0,T;H^1(D))).
		\]
		Then, for all $0\leq s\leq t\leq T$,
		\begin{align}
			\mE\left[
			\|\nab(u(t)-u(s))\|_{L^2}^{2p}
			\right]
			&\leq \widetilde C_{2,p}|t-s|^p,
			\label{eq_holder_H1_2}
		\end{align}
		where
		\begin{align*}
			\widetilde C_{2,p}
			:=C\left(1+\mE\Bigl[
			\sup_{0\leq r\leq T}
			\|u(r)\|_{H^2}^{2p}
			\Bigr]+\mE\Bigl[
			\sup_{0\leq r\leq T}
			\|u(r)\|_{H^1}^{2pq}
			\Bigr]\right).
		\end{align*}
	\end{enumerate}
	Here, $C>0$ may depend on $p,T,\chi,q,D$
	and $c_B$,
	but is independent of $s$ and $t$.
\end{lemma}

	\section{Fully discrete splitting mixed finite element method}\label{sec4}
	
\subsection{Formulation of the finite element method}

In this part, we consider the fully discrete method for \eqref{Weak_formulation}, which uses the implicit Euler scheme for the time discretization and the splitting mixed finite element method for the spatial discretization.

First, let $k=\frac{T}{M}$ be the uniform time step of a partition $\{t_m\}_{m=0}^M$ over the time interval $[0,T]$. Let $\{\mathcal{T}_h\}_{h>0}$ be a family of quasi-uniform periodic triangular meshes of $D=(0,L)^2$, with the mesh size $0<h\ll1$. We consider the finite element spaces $X_u^h\times X_{\vsigma}^h\times X_v^h$, which are generated by polynomials of degree at most $r_1,r_2,r_3\geq1$, respectively. In addition, the splitting mixed finite element method does not require the LBB inf-sup condition for its finite element spaces. Therefore, we freely choose linear polynomials for our finite element spaces. Although we consider only linear finite element spaces in this paper, all the error estimates for the finite element approximation can be extended to higher-order polynomial finite element spaces.

Next, we introduce the following projections:
\begin{align}
	&\cP_u:L^2(D)\rightarrow X_u^h,\qquad \cP_{\vsigma}:\vH^1(D)\rightarrow X_{\vsigma}^h,\qquad \cP_c:L^2(D)\rightarrow X_v^h,
\end{align}
where the operators satisfy
\begin{align}
	\label{Pu}\bigl(v-\cP_uv,\varphi_h\bigr)&=0\qquad\forall\,\varphi_h\in X_u^h,\\
	\label{Ps}\bigl(\vv-\cP_{\vsigma}\vv,\vvarphi_h\bigr)+\bigl(\nab\cdot(\vv-\cP_{\vsigma}\vv),\nab\cdot\vvarphi_h\bigr)+\bigl(rot(\vv-\cP_{\vsigma}\vv),rot\,\vvarphi_h\bigr)&=0\qquad\forall\,\vvarphi_h\in X_{\vsigma}^h,\\
	\label{Pc}\bigl(v-\cP_cv,\psi_h\bigr)&=0\qquad\forall\,\psi_h\in X_v^h.
\end{align}

Then, we have the following interpolation inequalities:
\begin{align}\label{ineq_interpolation}
	\|v-\cP_uv\|_{L^2}+h\|v-\cP_uv\|_{H^1}&\leq Ch^2\|v\|_{H^2}\qquad\forall\,v\in H^2(D),\\\nonumber
	\|\pphi-\cP_{\vsigma}\pphi\|_{L^2}+h\|\pphi-\cP_{\vsigma}\pphi\|_{\vH^1}&\leq Ch^2\|\pphi\|_{\vH^2}\qquad\forall\,\pphi\in\vH^2(D),\\\nonumber
	\|v-\cP_cv\|_{L^2}+h\|v-\cP_cv\|_{H^1}&\leq Ch^2\|v\|_{H^2}\qquad\forall\,v\in H^2(D).
\end{align}

Then we can state the fully discrete implicit Euler finite element method for \eqref{Weak_formulation} in the following algorithm.

\bigskip
{\bf Main Algorithm:} Let $u_h^0=\mathcal{P}_u u_0$ and $\vsigma_h^0=\cP_{\vsigma}\vsigma_0$. For any $0\leq m\leq M-1$, we do the following iterations: Find $\bigl(u_h^{m+1},\vsigma_h^{m+1},v_h^{m+1}\bigr)\in X_u^h\times X_{\vsigma}^h\times X_v^h$ such that $\mP$-a.s.
\begin{align}
	\label{eq_discrete_sigma}
	&\bigl(\vsigma_h^{m+1},\vvarphi_h\bigr)+\bigl(\nab\cdot\vsigma_h^{m+1},\nab\cdot\vvarphi_h\bigr)+\bigl(rot\,\vsigma_h^{m+1},rot\,\vvarphi_h\bigr)=-\bigl(u_h^m,\nab\cdot\vvarphi_h\bigr),\\
	\medskip
	\label{eq_discrete_u}
	&\bigl(u_h^{m+1}-u_h^m,\varphi_h\bigr)+k\bigl(\nab u_h^{m+1},\nab\varphi_h\bigr)=k\chi\bigl(u_h^{m+1}\vsigma_h^{m+1},\nab\varphi_h\bigr)\\\nonumber
	&\qquad\qquad\qquad\qquad\qquad\qquad\qquad+k\bigl(g(u_h^{m+1}),\varphi_h\bigr)+\bigl(B(u_h^m)\Delta W_m,\varphi_h\bigr),\\
	\medskip
	\label{eq_discrete_c}
	&\bigl(v_h^{m+1},\psi_h\bigr)=\bigl(\nab\cdot\vsigma_h^{m+1},\psi_h\bigr)+\bigl(u_h^{m+1},\psi_h\bigr),
\end{align}
for all $\bigl(\varphi_h,\vvarphi_h,\psi_h\bigr)\in X_u^h\times X_{\vsigma}^h\times X_v^h$, where $\Delta W_m=W(t_{m+1})-W(t_m)\sim\mathcal{N}(0,k)$.

	Since we choose $u^0_h = \mathcal{P}_u u_0$ and $\vsigma_h^0 = \cP_{\vsigma} \vsigma_0$ and \eqref{ineq_interpolation}, without loss of generality, we assume that $u^0_h = u_0$ and $\vsigma_h^0  = \vsigma_0$ in all the proofs of the stability and error estimates of the Main Algorithm.

	Now, we present the stability estimates of the Main Algorithm. 
	
	\begin{lemma}\label{lemma_fem_stability} Let $(u^m_h,\vsigma_h^m,v^m_h)$ be the solution of the Main Algorithm. Assume that $u_0 \in L^2(\Ome; L^2(D))$. Then, there exist a constant $\bar{C}>0$ and $k_0>0$ such that for all $k\in(0,k_0)$
		\begin{enumerate}
			\item[(a)] $\displaystyle \mE\left[\max_{1\leq m \leq M}\|u_h^m\|^2_{L^2} +  k\sum_{m=1}^{M}\|\nab u_h^{m}\|^2_{L^2} +  k\sum_{m=1}^{M}\| u_h^{m}\|^{q+1}_{L^{q+1}}\right] \leq \bar{C}$,
			\item[(b)] $\displaystyle \mE\left[\max_{1\leq m \leq M}\|\vsigma_h^m\|^2_{H^1}\right] \leq \bar{C}$,
			\item[(c)] $\displaystyle \mE\left[\max_{1\leq m \leq M} \|v_h^m\|^2_{L^2}\right]  \leq \bar{C}$.
		\end{enumerate}
	\end{lemma}
	\begin{proof}
		We proceed to prove (a) as follows. Taking $\varphi_h= u_h^{m+1}$ in \eqref{eq_discrete_u} and using the identity $2a(a-b) = a^2 - b^2 + (a-b)^2$, we obtain
		\begin{align}\label{eq_3.6}
			&\frac{1}{2}\bigl[\|u_h^{m+1}\|^2_{L^2} - \|u_h^{m}\|^2_{L^2} + \|u_h^{m+1} - u_h^m\|^2_{L^2}\bigr] + k\|\nab u_h^{m+1}\|^2_{L^2}\\\nonumber
			& = \chi k\bigl(u_h^{m+1}\vsigma_h^{m+1}, \nab u_h^{m+1}\bigr)+ k\bigl(g(u^{m+1}_h), u_h^{m+1}\bigr) +\bigl(B(u_h^m)\Delta W_m,  u_h^{m+1} - u_h^m\bigr)\\\nonumber
			&\qquad +\bigl(B(u_h^m)\Delta W_m,  u_h^m\bigr)\\\nonumber
			&:= I + II + III + IV.
		\end{align}
		
		First, we notice that 
		\begin{align*}
			II =  k\|u_h^{m+1}\|^2_{L^2} - k\|u_h^{m+1}\|^{q+1}_{L^{q+1}}.
		\end{align*}
		
		Next, using integration by parts and the  Cauchy-Schwarz inequality, we obtain
		\begin{align}\label{eq_3.7}
			I&=-\chi k \bigl((u_h^{m+1})^2, \nab\cdot\vsigma_h^{m+1}\bigr)\\\nonumber
			&\leq \chi k \|u_h^{m+1}\|^2_{L^4}\|\nab\cdot \vsigma_h^{m+1}\|_{L^2}.
		\end{align}
		
		Next, choosing $\vvarphi_h = \vsigma_h^{m+1}$ in \eqref{eq_discrete_sigma}, we obtain
		\begin{align}\label{eq_3.12}
			\|\vsigma_h^{m+1}\|^2_{H^1} \leq \|u_h^{m}\|^2_{L^2},
		\end{align}
		which together with \eqref{eq_3.7} give us that 
		\begin{align}\label{eq_3.8}
			I &\leq   \chi k \|u_h^{m+1}\|^2_{L^4}\|u_h^{m}\|_{L^2}\\\nonumber
			&\leq \chi k |D|^{\frac{q-3}{2(q+1)}}\|u^{m+1}_h\|^2_{L^{q+1}}\|u_h^m\|_{L^{2}}\\\nonumber
			&\leq
			\frac{k}{4}\|u_h^{m+1}\|_{L^{q+1}}^{q+1}
			+k\|u_h^m\|^2_{L^2} + C(\chi,q,D)k,
		\end{align}
		where the second inequality of \eqref{eq_3.8} is obtained by using the H\"older inequality, while the third inequality of \eqref{eq_3.8} is arrived by using the Young inequality.

		Next, we estimate $III$ and $IV$ with the expectation. First, using the independence property of the increments $\Delta W_m$, we immediately have that $\mE\left[IV\right] = 0$. Then, using the Cauchy-Schwarz inequality and the assumption of $B$ as well as the fact that $\mE\left[|\Delta W_m|^2\right] = k$, we obtain
		\begin{align*}
			\mE\left[III\right] &\leq \frac14\mE\left[\|u_h^{m+1} - u_h^m\|^2_{L^2}\right] + c_Bk \mE\left[\|u_h^m\|^2_{L^2}\right].
		\end{align*}
		
		Substituting all the estimates from $I, ..., IV$ into \eqref{eq_3.6} in the expectation, we arrive at
		\begin{align}\label{eq_3.9}
			&	\frac12\mE\left[\|u_h^{m+1}\|^2_{L^2} - \|u_h^m\|^2_{L^2} + \frac12\|u_h^{m+1} - u_h^m\|^2_{L^2}\right]  + \frac{3k}{4}\mE\left[ \|u_h^{m+1}\|^{q+1}_{L^{q+1}}\right] + k\mE\left[\|\nab u_h^{m+1}\|^2_{L^2}\right]\\\nonumber
			&\leq Ck\mE\left[\|u_h^{m+1}\|^2_{L^2} + \|u_h^m\|^2_{L^2}\right] + C(\chi,q,D)k.
		\end{align}
		
		Next, applying the summation $\sum_{m=0}^{\ell}$ to \eqref{eq_3.9} for any $1\leq \ell \leq M-1$, we obtain
		\begin{align}\label{eq_3.13}
			&	\frac12\mE\left[\|u_h^{\ell+1}\|^2_{L^2}\right] + \mE\left[\sum_{m=0}^{\ell}\left(\frac14\|u_h^{m+1} - u_h^m\|^2_{L^2} + \frac{ 3k }{4}\|u_h^{m+1}\|^{q+1}_{L^{q+1}} +  k\|\nab u_h^{m+1}\|^2_{L^2}\right)\right]\\\nonumber
			&\leq \frac{1}{2}\mE\left[\|u_h^0\|^2_{L^2}\right] + C(\chi,q,D)T+ Ck\sum_{m=0}^{\ell}\mE\left[\|u_h^{m+1}\|^2_{L^2} + \|u_h^m\|^2_{L^2}\right].
		\end{align}
		
		With this and using the discrete Gronwall inequality, for sufficiently small $k\in (0,1)$, we obtain the estimate in Part (a) after taking the maximum over $0\leq \ell\leq M-1$.
		
		Lastly, it is clear that Part (b) is obtained by taking $\vvarphi_h= \vsigma_h^{m+1}$ and using Part (a). Similarly, Part (c) is a consequence of Parts (a) and (b). The proof is complete.
	\end{proof}

	\medskip
	
	To derive strong error estimates for the Main Algorithm, we would need to establish the following higher moment stability estimates for $u_h^m$. 
	
	\begin{lemma}\label{lemma_fem_highmoment_stability} Let $(u^m_h,\vsigma_h^m,v^m_h)$ be the solution of the Main Algorithm. Assume that $u_0 \in L^{2^p}(\Ome; L^2(D))$ for all $p\in \mathbb{N}$. Let $B$ satisfy \eqref{Assump_LinearGrowth} with $m =0$. Then, there exist a constant $\bar{C}_{1,p}>0$ and $k_0>0$ such that for all $k\in(0,k_0)$
	\begin{align}\label{higher_moment_stability}
		\mE\left[\max_{1\leq m \leq M}\|u_h^m\|^{2^p}_{L^2} \right] \leq \bar{C}_{1,p}
	\end{align}
	\end{lemma}
	\begin{proof} 
	First, the case when $p = 1$ was established in Lemma \ref{lemma_fem_stability} (a). Next, we give the proof for the case $p = 2$. 
	
	Recall that the inequality \eqref{eq_3.9} without the expectation is
	\begin{align}\label{eq3.16}
			&	\frac12\left[\|u_h^{m+1}\|^2_{L^2} - \|u_h^m\|^2_{L^2}\right] + \frac14\|u_h^{m+1} - u_h^m\|^2_{L^2} + \frac{3k}{4}\|u_h^{m+1}\|^{q+1}_{L^{q+1}} + k\|\nab u_h^{m+1}\|^2_{L^2}\\\nonumber
		&\leq k\left[\|u_h^{m+1}\|^2_{L^2} + \|u_h^m\|^2_{L^2}\right] + \|B(u_h^m)\|^2_{L^2}|\Delta W_m|^2 + \bigl(B(u_h^m)\Delta W_m,  u_h^m\bigr)+ C(\chi,q,D)k.
	\end{align}
	
	Multiplying \eqref{eq3.16} by $\|u_h^{m+1}\|^2_{L^2}$ and using the identity $2a(a-b) = a^2 - b^2 + (a-b)^2$, we obtain
	\begin{align}\label{eq3.12}
		&\frac{1}{4}\left[\|u_h^{m+1}\|^4_{L^2} - \|u_h^m\|^4_{L^2}\right] + \frac14 \left(\|u_h^{m+1}\|^2_{L^2} - \|u_h^m\|^2_{L^2}\right)^2+ \frac14\|u_h^{m+1} - u_h^m\|^2_{L^2}\|u_h^{m+1}\|^2_{L^2}  \\\nonumber
		&\qquad+  \frac{3k}{4}\| u_h^{m+1}\|^{q+1}_{L^{q+1}}\|u_h^{m+1}\|^2_{L^2} + k \|\nab u_h^{m+1}\|^2_{L^2}\|u_h^{m+1}\|^2_{L^2}\\\nonumber
		&\leq  k\left[\|u_h^{m+1}\|^2_{L^2} + \|u_h^m\|^2_{L^2}\right]\|u^{m+1}_h\|^2_{L^2} + c_B^2\|u_h^m\|^2_{L^2}|\Delta W_m|^2\|u_h^{m+1}\|^2_{L^2} \\\nonumber
		&\qquad+ \bigl(B(u_h^m)\Delta W_m,  u_h^m\bigr)\|u_h^{m+1}\|^2_{L^2}+ C(\chi,q,D)k\|u_h^{m+1}\|^2_{L^2}\\\nonumber
		&\leq  k\|u_h^{m+1}\|^4_{L^2}+ \frac{k}{4} \|u_h^m\|^4_{L^2} + k\|u^{m+1}_h\|^4_{L^2} + \frac{c_B^4}{4k}\|u_h^m\|^4_{L^2}|\Delta W_m|^4 + k\|u_h^{m+1}\|^4_{L^2} \\\nonumber
		&\qquad+ \bigl(B(u_h^m)\Delta W_m,  u_h^m\bigr)\|u_h^{m}\|^2_{L^2}+ \frac18\left(\|u_h^{m+1}\|^2_{L^2} - \|u_h^m\|^2_{L^2}\right) +2\|B(u_h^m)\|^2_{L^2}\|u_h^m\|^2_{L^2}|\Delta W_m|^2\\\nonumber
		&\qquad+  C(\chi,q,D)k\|u_h^{m+1}\|^2_{L^2}\\\nonumber
		&\leq  k\|u_h^{m+1}\|^4_{L^2}+ \frac{k}{4} \|u_h^m\|^4_{L^2} + k\|u^{m+1}_h\|^4_{L^2} + \frac{c_B^4}{4k}\|u_h^m\|^4_{L^2}|\Delta W_m|^4 + k\|u_h^{m+1}\|^4_{L^2} \\\nonumber
		&\qquad+ \bigl(B(u_h^m)\Delta W_m,  u_h^m\bigr)\|u_h^{m}\|^2_{L^2}+ \frac18\left(\|u_h^{m+1}\|^2_{L^2} - \|u_h^m\|^2_{L^2}\right) +2c_B^2\|u_h^m\|^4_{L^2}|\Delta W_m|^2\\\nonumber
		&\qquad+  \frac{k}{4}C^2(\chi,q,D) + k\|u_h^{m+1}\|^4_{L^2},
	\end{align}
	where the last inequality above is obtained by using the Young inequality and \eqref{Assump_LinearGrowth} with $m=0$. Next, applying the summation $\sum_{m = 0}^{\ell}$ to \eqref{eq3.12} and then taking the expectation, we get
	\begin{align}\label{eq_3.18}
		&\frac{1}{4}\mE\left[\|u_h^{\ell+1}\|^4_{L^2}\right] + \frac18 \sum_{m = 0}^{\ell}\mE\left[\left(\|u_h^{m+1}\|^2_{L^2} - \|u_h^m\|^2_{L^2}\right)^2 \right]\\\nonumber
		&\leq \frac14 \mE\left[\|u_h^0\|^4_{L^2}\right]  + 4k\sum_{m = 0}^{\ell}\mE\left[\|u_h^{m+1}\|^4_{L^2}\right] + k\sum_{m = 0}^{\ell}\left(\frac14 + \frac{3c_B^4}{4} + 2c_B^2\right)\mE\left[\|u_h^m\|^4_{L^2}\right] +  \frac{T}{4}C^2(\chi,q,D),
	\end{align}
	where the last inequality of \eqref{eq_3.18} is obtained by using the facts that $\mE\left[\bigl(B(u_h^m)\Delta W_m,  u_h^m\bigr)\|u_h^{m}\|^2_{L^2}\right]=0, \mE[|\Delta W_m|^4] = 3k^2$ and $\mE[|\Delta W_m|^2] = k$.
	
	Then, taking applying the discrete Gronwall inequality to \eqref{eq_3.18} for sufficiently small $k$, we obtain 
	\begin{align}\label{eq3.200}
		&\mE\left[\|u_h^{\ell+1}\|^4_{L^2} \right]+  \sum_{m = 0}^{\ell}\mE\left[\left(\|u_h^{m+1}\|^2_{L^2} - \|u_h^m\|^2_{L^2}\right)^2\right] \leq C\mE\left[\|u_h^0\|^4_{L^2}\right] e^{CT}.
	\end{align}
	
Lastly, to establish \eqref{higher_moment_stability} for $p>2$, we use an inductive argument as it was done in the proof of \cite[Lemma 3.1]{brzezniak2013finite}, starting with multiplying the error inequality \eqref{eq3.16} by $\|u_h^{m+1}\|^{2^j}_{L^2}$ (assuming \eqref{higher_moment_stability} holds for $p=j$),
and utilize the discrete Gronwall's inequality as shown for the case $p=2$. We skip the derivation to save space and leave it to the interested reader to complete.

\end{proof}

	\medskip
	
	To obtain the strong error estimates, we need to derive the following discrete exponential stability for $u_h^m$.

\begin{lemma}\label{lemma_exp_fem_stability}
	Let $\{u_h^m\}_{m=0}^M$ be the solution of the Main Algorithm. Assume that $B$ satisfies \eqref{Assump_bounded}, and let $\gamma_2>0$ be arbitrary. Assume that the initial approximation $u_h^0$ satisfies
	\[
	\mE\left[
	\exp\left(
	4\gamma_2\|u_h^0\|_{L^2}^2
	+8\gamma_2^2L_0^2k\|u_h^0\|_{L^2}^2
	\right)
	\right]<\infty.
	\]
	Then, for any $T>0$, there exist constants $k_0>0$ and $C_{exp,2}>0$, independent of $h$, $k$, and $M$, such that, for every $0<k\leq k_0$,
	\begin{align}\label{ineq_exp_fem_stability}
		\max_{1\leq\ell\leq M}
		\mE\left[
		\exp\left(
		\gamma_2\|u_h^\ell\|_{L^2}^2
		+\gamma_2 k\sum_{m=1}^{\ell}\|\nab u_h^m\|_{L^2}^2
		+\gamma_2 k\sum_{m=1}^{\ell}\|u_h^m\|_{L^{q+1}}^{q+1}
		\right)
		\right]
		\leq C_{exp,2}.
	\end{align}
	Here, $k_0$ can be chosen sufficiently small such that
$k_0<\frac{1}{32L_0^2\gamma_2},$
	and
	\[
	C_{exp,2}
	=
	\left(
	\mE\left[
	\exp\left(
	4\gamma_2\|u_h^0\|_{L^2}^2
	+C(\chi,q,L_0,\gamma_2,D)T
	\right)
	\right]
	\right)^{1/2}
	\exp\left(4L_0^2\gamma_2T\right).
	\]
\end{lemma}
	\begin{proof}
	This proof follows the same lines the proof of \cite[Lemma 4.3]{nguyen2026finite}. However, for completeness, we also give the detailed proof here.	From \eqref{eq_3.9}, (without the expectation), and the assumption \eqref{Assump_bounded}, we have
		\begin{align}
		&	\|u_h^{m}\|^2_{L^2} - \|u_h^{m-1}\|^2_{L^2} + k \|\nab u_h^{m}\|^2_{L^2} +  \frac{3k}{2} \| u_h^{m}\|^{q+1}_{L^{q+1}} 
		\\\nonumber	&\qquad\leq 2\|B(u_h^{m-1})\|^2_{L^2}|\Delta W_{m-1}|^2 + 2(B(u_h^{m-1})\Delta W_{m-1}, u_h^{m-1}) + C(\chi,q,D)k
		\\\nonumber	&\qquad\leq 2L_0^2|\Delta W_{m-1}|^2 + 2(B(u_h^{m-1})\Delta W_{m-1}, u_h^{m-1}) + C(\chi,q,D)k.
		\end{align}
		Applying the summation $ \sum_{m=1}^{\ell}$ for any $1 \leq \ell \leq M$ produces
		\begin{align}\label{equu3.18}
			&\|u_h^{\ell}\|^2_{L^2}  + k \sum_{m = 1}^{\ell}\|\nab u_h^{m}\|^2_{L^2} + \frac{3k}{2} \sum_{m = 1}^{\ell}\| u_h^{m}\|^{q+1}_{L^{q+1}}\\\nonumber 
			&\qquad\leq \|u_h^0\|^2_{L^2} +C(\chi,q,D)T + 2L_0^2\sum_{m = 1}^{\ell}|\Delta W_{m-1}|^2 + 2\sum_{m = 1}^{\ell}(B(u_h^{m-1})\Delta W_{m-1}, u^{m-1}_h).
		\end{align}
		
		For any $\gamma > 0$, denote $Z_{\ell} =2\gamma\sum_{m=1}^{\ell}\bigl(B(u_h^{m-1})\Delta W_{m-1}, u^{m-1}\bigr)$, for $1\leq \ell\leq M$. We introduce a martingale by using a piece-wise constant function that is defined as follows: for $s \in [t_{m-1},t_{m})$ with $m =1,2,\cdots, M$, and the convention $t_0=0$, set $\bar{u}(s)= u_h^{m-1}$ and
		\begin{align}
			\tilde{Z}_{t} = 2\gamma\int_{0}^{t}\bigl(B(\bar{u}(s)), \bar{u}(s)\bigr)\,d W(s),   \qquad\forall t\in [0,T].
		\end{align}
		Observe that $\displaystyle Z_{\ell} = \tilde{Z}_{t_{\ell}}$ and that $\tilde{Z}_t$ is a $\mathcal{F}_t$-martingale for $t \in [0,T]$ whose quadratic variation satisfies
		\begin{align}\label{equu3.21}
			\bigl<\tilde{Z}\bigr>_{t_{\ell}}  &=4\gamma^2 \int_{t_0}^t |(B(\bar{u}(s)), \bar{u}(s))|^2\, ds\\\nonumber
			&\leq 4\gamma^2\int_{t_0}^{t_{\ell}} \|B(\bar{u}(s))\|^2_{L^2}\|\bar{u}(s)\|^2_{L^2}\, ds\\\nonumber
			&\leq 4\gamma^2L_0^2\, k\sum_{m=1}^{\ell}\|u_h^{m-1}\|^2_{L^2}\\\nonumber
			&\leq \frac{\gamma}{4}\, k\sum_{m=1}^{\ell}\|u_h^{m-1}\|^{q+1}_{L^{q+1}} + C(\gamma,L_0,D) T + 4\gamma^2 L_0^2k\|u_h^0\|^2_{L^2}.
		\end{align}
		We multiply \eqref{equu3.18} by $\gamma >0$ and then take the exponential on both sides to get
		\begin{align}\label{equu323}
			&\exp\Bigl(\gamma\|u_h^{\ell}\|^2_{L^2}  +\gamma k \sum_{m = 1}^{\ell}\|\nab u_h^{m}\|^2_{L^2} + \gamma k \sum_{m = 1}^{\ell}\| u_h^{m}\|^{q+1}_{L^{q+1}}\Bigr) \\\nonumber
			&\leq \exp\bigl(2\gamma\|u_h^0\|^2_{L^2} + C(\chi,q,D)T\bigr)\times \exp\Bigl(2L_0^2\gamma\sum_{m=1}^{\ell}|\Delta W_{m-1}|^2\Bigr)\\\nonumber
			&\qquad\times\exp\Bigl(\bigl[Z_{t_{\ell}} -2\bigl<\tilde{Z}\bigr>_{t_{\ell}}\bigr]\Bigr)\times \exp\biggl(2\bigl<\tilde{Z}\bigr>_{t_{\ell}} -  \frac{\gamma}{2}k\sum_{n=1}^{\ell}\|u_h^n\|^{q+1}_{L^{q+1}}\biggr).
		\end{align} 
		Moreover, from \eqref{equu3.21} we have
		\begin{align}\label{equu3.24}
			2\bigl<\tilde{Z}\bigr>_{t_{\ell}} -  \frac{\gamma}{2}k\sum_{n=1}^{\ell}\|u_h^n\|^{q+1}_{L^{q+1}}
			&\leq C(\gamma,L_0,D)T + 4\gamma^2 L_0^2k\|u_h^0\|^2_{L^2}.
		\end{align}	
		
		Next, taking the expectation on \eqref{equu323} and using \eqref{equu3.21} we obtain
		\begin{align}\label{eq4.25}
		&	\mE\left[\exp\left(\gamma\|u_h^{\ell}\|^2_{L^2}  +\gamma k \sum_{m = 1}^{\ell}\|\nab u_h^{m}\|^2_{L^2} + \gamma k \sum_{m = 1}^{\ell}\| u_h^{m}\|^{q+1}_{L^{q+1}}\right)\right]\\\nonumber
			&\leq \mE\left[\exp\left(2\gamma\|u_h^0\|^2_{L^2} + 4\gamma^2 L_0^2k\|u_h^0\|^2_{L^2} + C(\chi,q,L_0,\gamma,D)T\right)\right.\\\nonumber
			&\qquad \left.\times \exp\Bigl(2L_0^2\gamma\sum_{m=1}^{\ell}|\Delta W_{m-1}|^2\Bigr) \times\exp\Bigl(\bigl[Z_{t_{\ell}} -2\bigl<\tilde{Z}\bigr>_{t_{\ell}}\bigr]\Bigr)\right].
		\end{align}
		Using the H\"older inequality to separate the product on the right-hand side of \eqref{eq4.25} yields
		\begin{align}\label{equu3.28}
		&	\mE\left[\exp\left(\gamma\|u_h^{\ell}\|^2_{L^2}  +\gamma k \sum_{m = 1}^{\ell}\|\nab u_h^{m}\|^2_{L^2} + \gamma k \sum_{m = 1}^{\ell}\| u_h^{m}\|^{q+1}_{L^{q+1}}\right)\right]\\\nonumber
			&\leq \Bigl(\mE\left[\exp\left(4\gamma\|u_h^0\|^2_{L^2} + 8\gamma^2 L_0^2k\|u_h^0\|^2_{L^2}+  C(\chi,q,L_0,\gamma,D)T\right)\right]\Bigr)^{1/2}\\\nonumber
			&\qquad\times \biggl(\mE\Bigl[\exp\Bigl(8L_0^2\gamma\sum_{m=1}^M|\Delta W_m|^2\Bigr)\Bigr]\biggr)^{1/4} \times \Bigl(\mE\Bigl[\exp\Bigl(\bigl[4\tilde{Z}_{t_{\ell}} -\frac{1}{2}\bigl<4\tilde{Z}\bigr>_{t_{\ell}}\bigr]\Bigr)\Bigr]\Bigr)^{1/4}.
		\end{align}
		Since $\exp\Bigl(4\tilde{Z}_{t} -\frac{1}{2}\bigl<4\tilde{Z}\bigr>_{t}\Bigr)$ is an exponential martingale on $[0,T]$, it suffices to show that the expectation $\mE\Bigl[\exp\Bigl(8L_0^2\gamma\sum_{m=1}^M|\Delta W_m|^2\Bigr)\Bigr]$ is bounded.  To this end, by using the independence of the increments of $W(t)$, we can rewrite
		\begin{align}\label{equation3.23}
			\mE\Bigl[\exp\Bigl(8L_0^2\gamma\sum_{m=1}^M|\Delta W_m|^2\Bigr)\Bigr] &= 
			\prod_{m=1}^M \mE\bigl[e^{8L_0^2\gamma (\Delta W_m)^2}\bigr].
		\end{align}
		Moreover, since $\Delta W_m \sim \mathcal{N}(0,k)$, a simple computation gives
		\begin{align*}
			\mE\bigl[e^{8L_0^2\gamma (\Delta W_m)^2}\bigr] = \frac{1}{\sqrt{1- 16L_0^2\gamma k}},
		\end{align*}
		provided that $1- 16L_0^2\gamma k > 0$. This condition can be fulfilled since the time step $k$ can be chosen as small as possible. For example, we may choose $k < \frac{1}{32 L_0^2\gamma}$. So, $1- 16L_0^2\gamma k \geq \frac12 >0$. Therefore, \eqref{equation3.23} is reduced to
		\begin{align}
			\mE\Bigl[\exp\Bigl(8L_0^2\gamma\sum_{m=1}^M|\Delta W_m|^2\Bigr)\Bigr] &=  (1-16 L_0^2\gamma k)^{\frac{-T}{2k}} = \exp\left(-\frac{T}{2k}\log(1-16L_0^2\gamma k)\right) \leq e^{16 L_0^2\gamma},
		\end{align}
		provide that $k < \frac{1}{32 L_0^2\gamma}$.
		Finally, putting this one into \eqref{equu3.28}, the proof is complete.
		
	\end{proof}
	
	Next, we state and prove the error estimates of the Main Algorithm. 
	
	\subsection{Error estimates} In this part, we analyze and derive the error estimates for the Main Algorithm.
	
\begin{theorem}\label{Thm_global_error}
	Let $(u,\vsigma,v)$ and $(u_h^m,\vsigma_h^m,v_h^m)$ be the variational solution of \eqref{Weak_formulation} and the approximate solution generated by the Main Algorithm, respectively. Assume that $B$ satisfies \eqref{Assump_Lipschitz}, \eqref{Assump_LinearGrowth}, and \eqref{Assump_bounded}. Assume that $u_0\in L^4(\Omega;H^2(D))$ and that
	\[
	\mE\left[
	\exp\left(
	\gamma\|u_0\|_{L^2}^2
	\right)
	\right]<\infty
	\qquad\text{for some }\gamma>0.
	\]
	Let $0<r<1$. Then there exist constants $k_0>0$ and $\widehat{C}_1>0$, where $k_0$ is chosen according to Lemma~\ref{lemma_exp_fem_stability} and $\widehat{C}_1$ is independent of $h$ and $k$, such that, for every $0<k\leq k_0$,
	\begin{align}
		\label{estimate_u_fem}
		\left(
		\mE\left[
		\max_{1\leq m\leq M}
		\|u(t_m)-u_h^m\|_{L^2}^{2r}
		\right]
		\right)^{\frac{1}{2r}}
		&+
		\left(
		\mE\left[
		\left(
		k\sum_{m=1}^{M}
		\|\nab(u(t_m)-u_h^m)\|_{L^2}^2
		\right)^r
		\right]
		\right)^{\frac{1}{2r}}
		\\\nonumber
		&\leq \widehat{C}_1\left(k^{\frac12}+h\right),\\
		\label{estimate_sigma_fem}
		\left(
		\mE\left[
		\max_{1\leq m\leq M}
		\|\vsigma(t_m)-\vsigma_h^m\|_{H^1}^{2r}
		\right]
		\right)^{\frac{1}{2r}}
		&\leq \widehat{C}_1\left(k^{\frac12}+h\right),\\
		\label{estimate_c_fem}
		\left(
		\mE\left[
		\max_{1\leq m\leq M}
		\|v(t_m)-v_h^m\|_{L^2}^{2r}
		\right]
		\right)^{\frac{1}{2r}}
		&\leq \widehat{C}_1\left(k^{\frac12}+h\right).
	\end{align}
\end{theorem}

	\begin{proof} First, we focus on giving the proof of \eqref{estimate_u_fem}. Denote
		\begin{align*}
			e_{u}^m &= u(t_m) - u^m_h  = u(t_m) - \cP_{u} u(t_m) + \cP_{u} u(t_m) - u_h^m:= \theta_{u}^m + \varepsilon_{u}^m  \\
			e_{\vsigma}^m &= \vsigma(t_m) - \vsigma_h^m = \vsigma(t_m) - \cP_{\vsigma}\vsigma(t_m) + \cP_{\vsigma} \vsigma(t_m) - \vsigma_h^m:= \vtheta_{\vsigma}^m + \e_{\vsigma}^m. 
		\end{align*}
		
		Subtract \eqref{eq_u} from \eqref{eq_discrete_u}, we obtain the following error equation.
		\begin{align}\label{eq_3.27}
			\bigl(e_u^{m+1} - e^m_u,\varphi_h\bigr) + k\bigl(\nab e_u^{m+1}, \nab \varphi_h\bigr)
			&= \int_{t_m}^{t_{m+1}} \bigl(\nab u(t_{m+1}) - \nab u(s), \nab \varphi_h\bigr)\, ds \\\nonumber
			&\qquad+ \chi\int_{t_m}^{t_{m+1}} \bigl(u(s)\vsigma(s) - u^{m+1}_h\vsigma^{m+1}_h, \nab \varphi_h\bigr)\, ds \\\nonumber
			&\qquad+\int_{t_m}^{t_{m+1}} \bigl(g(u(s)) - g(u_h^{m+1}), \varphi_h\bigr)\, ds \\\nonumber
			&\qquad+ \Bigl(\int_{t_m}^{t_{m+1}} B(u(s)) dW(s) - B(u_h^m)\Delta W_m, \varphi_h\Bigr).
		\end{align}
		
		Taking $v_h = \varepsilon_u^{m+1}$ in \eqref{eq_3.27} and using the orthogonal property of $\cP_u$ \eqref{Pu}, we obtain the following error equation:
		\begin{align}\label{eq_3.28}
			&\frac12\bigl[\|\varepsilon^{m+1}_u\|^2_{L^2} - \|\varepsilon_u^m\|^2_{L^2}\bigr] +\frac12\|\varepsilon_u^{m+1} - \varepsilon_u^m\|^2_{L^2}+  k\|\nab\varepsilon_{u}^{m+1}\|^2_{L^2}\\\nonumber
			&= - k\bigl(\nab\theta_u^{m+1},\nab\varepsilon_u^{m+1}\bigr) +  \int_{t_m}^{t_{m+1}} \bigl(\nab u(t_{m+1}) - \nab u(s), \nab \varepsilon_u^{m+1}\bigr)\, ds \\\nonumber
			&\qquad+ \chi\int_{t_m}^{t_{m+1}} \bigl(u(s)\vsigma(s) - u^{m+1}_h\vsigma^{m+1}_h, \nab \varepsilon_u^{m+1}\bigr)\, ds \\\nonumber
			&\qquad+\int_{t_m}^{t_{m+1}} \bigl(g(u(s)) - g(u_h^{m+1}), \varepsilon_{u}^{m+1}\bigr)\, ds \\\nonumber
			&\qquad+ \Bigl(\int_{t_m}^{t_{m+1}} B(u(s)) dW(s) - B(u_h^m)\Delta W_m,\varepsilon_{u}^{m+1}\Bigr)\\\nonumber
			&:= I + II + III + IV + V.
		\end{align}
		
		Now, we begin to estimate the right-hand side of \eqref{eq_3.28}. First, using Cauchy-Schwarz's inequality and \eqref{ineq_interpolation}, we get
		\begin{align}
			I &\leq \frac{k}{16} \|\nab \varepsilon_u^{m+1}\|^2_{L^2} + Ck h^2 \|u(t_{m+1})\|^2_{H^2}.
		\end{align}
		
		Using the Cauchy-Schwarz inequality, we obtain
		\begin{align*}
			II &\leq C\int_{t_m}^{t_{m+1}}\|\nab(u(t_{m+1}) - u(s))\|^2_{L^2}\, ds + \frac{ k}{16}\|\nab \varepsilon_u^{m+1}\|^2_{L^2}.
		\end{align*}
		
		Now, we turn to estimate $III$ by adding and subtracting the term $u(t_{m+1})\vsigma(t_{m+1})$:
		\begin{align}
			III &= \chi \int_{t_m}^{t_{m+1}} \bigl(u(s)\vsigma(s) - u(t_{m+1})\vsigma(t_{m+1}),\nab \varepsilon_u^{m+1}\bigr)\, ds \\\nonumber
			&\qquad+ \chi k \bigl(u(t_{m+1})\vsigma(t_{m+1}) - u^{m+1}_h\vsigma^{m+1}_h,\nab \varepsilon_u^{m+1}\bigr)\\\nonumber
			&:= III_1 + III_2,
		\end{align}
		where the terms $III_1$ and $III_2$ are estimated as follows. By using the Cauchy-Schwarz inequality and then the Young inequality, we have
		\begin{align*}
			III_1 &=\chi \int_{t_m}^{t_{m+1}}[\bigl(u(s)[\vsigma(s) - \vsigma(t_{m+1})] + [u(s) - u(t_{m+1})]\vsigma(t_{m+1}), \nab \varepsilon_{u}^{m+1}\bigr)]\, ds \\\nonumber
			&\leq C\int_{t_m}^{t_{m+1}}[\|u(s)[\vsigma(s) - \vsigma(t_{m+1})]\|^2_{L^2} + \|[u(s) - u(t_{m+1})]\vsigma(t_{m+1})\|^2_{L^2}]\, ds \\\nonumber
			&\qquad\qquad\qquad+ \frac{k}{16}\|\nab \varepsilon_u^{m+1}\|^2_{L^2}\\\nonumber
			&\leq C\int_{t_m}^{t_{m+1}}[\|u(s)\|^2_{L^4}\|[\vsigma(s) - \vsigma(t_{m+1})]\|^2_{L^4} + \|[u(s) - u(t_{m+1})]\|^2_{L^4}\|\vsigma(t_{m+1})\|^2_{L^4}]\, ds \\\nonumber
			&\qquad\qquad\qquad+ \frac{k}{16}\|\nab \varepsilon_u^{m+1}\|^2_{L^2}.
		\end{align*}

		Now, we turn to control $III_2$ by adding and subtracting $u(t_{m+1})\vsigma_h^{m+1}$ to get
		\begin{align*}
			III_2 &= \chi k \bigl(u(t_{m+1})e_{\vsigma}^{m+1},\nab \varepsilon_u^{m+1}\bigr) + \chi k \bigl(e_{u}^{m+1}\vsigma_h^{m+1},\nab \varepsilon_u^{m+1}\bigr)\\\nonumber
			&= \chi k \bigl(u(t_{m+1})\varepsilon_{\vsigma}^{m+1},\nab \varepsilon_u^{m+1}\bigr) + \chi k \bigl(\varepsilon_{u}^{m+1}\vsigma_h^{m+1},\nab \varepsilon_u^{m+1}\bigr)\\\nonumber
			&\qquad+ \chi k \bigl(u(t_{m+1})\theta_{\vsigma}^{m+1},\nab \varepsilon_u^{m+1}\bigr) + \chi k \bigl(\theta_{u}^{m+1}\vsigma_h^{m+1},\nab \varepsilon_u^{m+1}\bigr)\\\nonumber
			&:= III_{2,1} + III_{2,2} + III_{2,3} + III_{2,4}.
		\end{align*}	
		
		To estimate $III_{2,1}$, we have to control $\varepsilon_{\vsigma}^{m+1}$ as follows.	Subtracting \eqref{eq_sigma} to \eqref{eq_discrete_sigma}, we obtain
		\begin{align}\label{eq_3.31}
			&	\bigl(e_{\vsigma}^{m+1}, \vvarphi_h\bigr) + \bigl(\nab\cdot e_{\vsigma}^{m+1},\nab\cdot\vvarphi_h\bigr) + \bigl(rot\, e_{\vsigma}^{m+1},\, rot\, \vvarphi_h\bigr) \\\nonumber
			&= -\bigl(u(t_{m+1}) - u(t_m),\nab\cdot\vvarphi_h\bigr) - \bigl(e_{u}^m,\nab\cdot\vvarphi_h\bigr).
		\end{align}
		Taking $\vvarphi_h = \varepsilon_{\vsigma}^{m+\frac12}$ in \eqref{eq_3.31} and using the orthogonality of $\cP_{\vsigma}$ from \eqref{Ps}, we arrive at
		\begin{align}\label{eq3.32}
			\|\varepsilon_{\vsigma}^{m+1}\|^2_{H^1} & \leq C\bigl[\|u(t_{m+1}) - u(t_m)\|^2_{L^2} \bigr] + \frac12\|e_u^{m}\|^2_{L^2}\\\nonumber
			& \leq C\bigl[\|u(t_{m+1}) - u(t_m)\|^2_{L^2} \bigr] + \|\varepsilon_u^{m}\|^2_{L^2} + Ch^4\|u(t_{m})\|^2_{H^2}.
		\end{align}
		
		Using  \eqref{eq3.32} and the Ladyzhenskaya inequality to bound $III_{2,1}$, we obtain 
		\begin{align}\label{eq3.35}
			III_{2,1} &\leq \chi k \|u(t_{m+1})\|_{L^4}\|\varepsilon_{\vsigma}^{m+1}\|_{L^4}\|\nab \varepsilon_{u}^{m+1}\|_{L^2}\\\nonumber
			&\leq \frac{ k}{16}\|\nab \varepsilon_{u}^{m+1}\|^2_{L^2} + 4\chi^2 k\|u(t_{m+1})\|^2_{L^4} \|\varepsilon_{\vsigma}^{m+1}\|^2_{H^1}\\\nonumber
			&\leq  \frac{k}{16}\|\nab \varepsilon_{u}^{m+1}\|^2_{L^2} + Ck\|u(t_{m+1})\|^2_{L^4}\|u(t_{m+1}) - u(t_m)\|^2_{L^2}  \\\nonumber
			&\qquad+ 4\chi^2k\|u(t_{m+1})\|^2_{L^4}\|\varepsilon_u^{m}\|^2_{L^2} + Ckh^4\|u(t_{m+1})\|^2_{L^4}\|u(t_{m})\|^2_{H^2}.
		\end{align}
		
		Now, we can further analyze the third-term on the right-hand side of \eqref{eq3.35} as follows: Adding the terms $\pm\int_{t_{m}}^{t_{m+1}}u(s)\, ds$ and then using the H\"older inequality, we have
		\begin{align}
			4\chi^2k\|u(t_{m+1})\|^2_{L^4}\|\varepsilon_u^{m}\|^2_{L^2} &\leq \frac{8\chi^2}{k} \left\|\int_{t_{m}}^{t_{m+1}}(u(t_{m+1}) - u(s))\, ds\right\|^2_{L^4}\|\varepsilon_{u}^m\|^2_{L^2}\\\nonumber
			&\qquad +  \frac{8\chi^2}{k} \left\|\int_{t_{m}}^{t_{m+1}} u(s)\, ds\right\|^2_{L^4}\|\varepsilon_{u}^m\|^2_{L^2}\\\nonumber
			&\leq 8\chi^2k^{\frac12}\left(\int_{t_{m}}^{t_{m+1}}\left\|(u(t_{m+1}) - u(s))\right\|^4_{L^4}\, ds\right)^{\frac12}\|\varepsilon_{u}^m\|^2_{L^2}\\\nonumber
			&\qquad +  8\chi^2 k^{\frac12}\left(\int_{t_{m}}^{t_{m+1}} \|u(s)\|^4_{L^4}\, ds\right)^{\frac12}\|\varepsilon_{u}^m\|^2_{L^2}\\\nonumber
			&\leq 8\chi^2k^{\frac12}\left(\int_{t_{m}}^{t_{m+1}}\left\|(u(t_{m+1}) - u(s))\right\|^4_{L^4}\, ds\right)^{\frac12}\|\varepsilon_{u}^m\|^2_{L^2}\\\nonumber
			&\qquad +  4\chi^2 \left(k +\int_{t_{m}}^{t_{m+1}} \|u(s)\|^4_{L^4}\, ds\right)\|\varepsilon_{u}^m\|^2_{L^2},
		\end{align}
		which implies that
		\begin{align}\label{eq3.37}
			III_{2,1} 	&\leq  \frac{k}{16}\|\nab \varepsilon_{u}^{m+1}\|^2_{L^2} + Ck\|u(t_{m+1})\|^2_{L^4}\|u(t_{m+1}) - u(t_m)\|^2_{L^2}  \\\nonumber
			&\qquad +8\chi^2k^{\frac12}\left(\int_{t_{m}}^{t_{m+1}}\left\|(u(t_{m+1}) - u(s))\right\|^4_{L^4}\, ds\right)^{\frac12}\|\varepsilon_{u}^m\|^2_{L^2}\\\nonumber
			&\qquad +  4\chi^2 \left(k +\int_{t_{m}}^{t_{m+1}} \|u(s)\|^4_{L^4}\, ds\right)\|\varepsilon_{u}^m\|^2_{L^2}+  Ckh^4\|u(t_{m+1})\|^2_{L^4}\|u(t_{m})\|^2_{H^2}\\\nonumber
				&\leq \frac{k}{16}\|\nab \varepsilon_{u}^{m+1}\|^2_{L^2}
				+ Ck\|u(t_{m+1})\|^2_{L^4}
				\|u(t_{m+1})-u(t_m)\|^2_{L^2}
				\\\nonumber
				&\qquad
				+8\chi^2k^{\frac12}
				\left(
				\int_{t_m}^{t_{m+1}}
				\|u(t_{m+1})-u(s)\|^4_{L^4}\,ds
				\right)^{\frac12}
				\|\varepsilon_u^m\|^2_{L^2}
				\\\nonumber
				&\qquad
				+\frac{16\chi^2}{q+1}
				\left(
				\int_{t_m}^{t_{m+1}}
				\|u(s)\|_{L^{q+1}}^{q+1}\,ds
				\right)
				\|\varepsilon_u^m\|^2_{L^2}
				\\\nonumber
				&\qquad
				+4\chi^2
				\left(
				1+\frac{q-3}{q+1}|D|
				\right)k
				\|\varepsilon_u^m\|^2_{L^2}
				+Ckh^4\|u(t_{m+1})\|^2_{L^4}
				\|u(t_m)\|^2_{H^2}\\\nonumber
				&\leq \frac{k}{16}\|\nab \varepsilon_{u}^{m+1}\|^2_{L^2}
				+ Ck\|u(t_{m+1})\|^2_{L^4}
				\|u(t_{m+1})-u(t_m)\|^2_{L^2}
				\\\nonumber
				&\qquad
				+8\chi^2k^{\frac12}
				\left(
				\int_{t_m}^{t_{m+1}}
				\|u(t_{m+1})-u(s)\|^4_{L^4}\,ds
				\right)^{\frac12}
				\|\varepsilon_u^m\|^2_{L^2}
				\\\nonumber
				&\qquad
				+4\chi^2
				\left(
				\int_{t_m}^{t_{m+1}}
				\|u(s)\|_{L^{q+1}}^{q+1}\,ds
				+(1+|D|)k
				\right)
				\|\varepsilon_u^m\|^2_{L^2}
				+Ckh^4\|u(t_{m+1})\|^2_{L^4}
				\|u(t_m)\|^2_{H^2},
		\end{align}
		where the second inequality of \eqref{eq3.37} is obtained by using the Young inequality.

		Next, using the Cauchy-Schwarz inequality, the Ladyzhenskaya inequality and then \eqref{eq_3.12}, we have
		\begin{align}
			III_{2,2} &=-\chi k ((\varepsilon_{u}^{m+1})^2,\nab\cdot\vsigma_h^{m+1}) \\\nonumber
			&\leq \chi k \|\varepsilon_{u}^{m+1}\|^2_{L^4}\|\vsigma_h^{m+1}\|_{H^1}\\\nonumber
			&\leq C_L^2\chi k \|\varepsilon_{u}^{m+1}\|_{L^2}\|\varepsilon_{u}^{m+1}\|_{H^1}\|u_h^m\|_{L^2}\\\nonumber
			&\leq \frac{ k}{16} \left(\|\varepsilon_{u}^{m+1}\|^2_{L^2} + \|\nab \varepsilon_{u}^{m+1}\|^2_{L^2}\right) + 4C_L^4\chi^2k\|u_h^m\|^2_{L^2}\|\varepsilon_{u}^{m+1}\|^2_{L^2}.
		\end{align}
		
		To estimate the terms $III_{2,3}$ and $III_{2,4}$, we use the Cauchy-Schwarz inequality, the inequality \eqref{ineq_interpolation} as follows:
		\begin{align}
			III_{2,3} + III_{2,4} &\leq \frac{k}{16}\|\nab \varepsilon_{u}^{m+1}\|^2_{L^2} + Ck \|u(t_{m+1})\|^2_{L^4}\|\theta_{\vsigma}^{m+1}\|^2_{L^4} + Ck \|\theta_{u}^{m+1}\|^2_{L^4}\|\vsigma_h^{m+1}\|^2_{L^4}\\\nonumber
			&\leq \frac{k}{16}\|\nab \varepsilon_{u}^{m+1}\|^2_{L^2} + Ckh^3\|u(t_{m+1})\|^2_{L^4}\|\vsigma(t_{m+1})\|^2_{H^2} + Ck h^3\|u_h^{m}\|^2_{L^2}\|u(t_{m+1})\|^2_{H^2}.
		\end{align}

		Next, we proceed to control $IV$ as follows. Using the fact that $(g(a) - g(b))(a-b) \leq (a-b)^2$, we have
		\begin{align}
			IV &= \int_{t_{m}}^{t_{m+1}} \bigl(g(u(s)) - g(u(t_{m+1})), \varepsilon_{u}^{m+1}\bigr)\, ds + k \bigl(g(u(t_{m+1})) - g(\cP_{u}u(t_{m+1})), \varepsilon_{u}^{m+1}\bigr) \\\nonumber
			&\qquad+ k \bigl(g(\cP_{u}u(t_{m+1})) - g(u_h^{m+1}), \varepsilon_{u}^{m+1}\bigr) \\\nonumber
			&\leq C\int_{t_{m}}^{t_{m+1}}\|g(u(s)) - g(u(t_{m+1}))\|^2_{L^2}\, ds + Ck\|g(u(t_{m+1})) - g(\cP_{u}u(t_{m+1}))\|^2_{L^2} + k \|\varepsilon_{u}^{m+1}\|^2_{L^2}.
		\end{align}
		
		Finally, we estimate the noise term $V$ as follows. 
		\begin{align*}
			V &=  \left(\int_{t_m}^{t_{m+1}} (B(u(s)) - B(u_h^m))\, dW(s), \varepsilon_u^{m+1} - \varepsilon_u^m\right) +  \left(\int_{t_m}^{t_{m+1}} (B(u(s)) - B(u_h^m))\, dW(s),  \varepsilon_u^m\right)\\\nonumber
			&=  \left(\int_{t_m}^{t_{m+1}} (B(u(s)) - B(u(t_{m})))\, dW(s), \varepsilon_u^{m+1} - \varepsilon_u^m\right) +\bigl((B(u(t_m)) - B(u_h^m))\Delta W_m, \varepsilon_u^{m+1} - \varepsilon_u^m\bigr) \\\nonumber
			&\qquad+  \left(\int_{t_m}^{t_{m+1}} (B(u(s)) - B(u^m_h))\, dW(s),  \varepsilon_u^m\right)\\\nonumber
			&\leq  2\left\|\int_{t_m}^{t_{m+1}} (B(u(s)) - B(u(t_{m})))\, dW(s)\right\|^2_{L^2} +2\left\|(B(u(t_m)) - B(u_h^m))\Delta W_m\right\|^2_{L^2} + \frac14\|\varepsilon_u^{m+1} - \varepsilon_u^{m}\|^2_{L^2} \\\nonumber
			&\qquad+  \left(\int_{t_m}^{t_{m+1}} (B(u(s)) - B(u^m_h))\, dW(s),  \varepsilon_u^m\right).
		\end{align*}
		
		To apply the stochastic Gronwall inequality in Lemma \ref{Stochastic_Gronwall}, we need to add the terms $\pm2k\|B(u(t_m)) - B(u_h^m)\|^2_{L^2}$ on the right-hand side of $V$, and then using the condition \eqref{Assump_Lipschitz} as follows:
		\begin{align*}
			V &\leq  2\left\|\int_{t_m}^{t_{m+1}} (B(u(s)) - B(u(t_{m})))\, dW(s)\right\|^2_{L^2}  + \frac14\|\varepsilon_u^{m+1} - \varepsilon_u^{m}\|^2_{L^2} \\\nonumber
			&\qquad+2\left\|(B(u(t_m)) - B(u_h^m))\Delta W_m\right\|^2_{L^2} - 2k\|B(u(t_m)) - B(u_h^m)\|^2_{L^2} + 2k\|B(u(t_m)) - B(u_h^m)\|^2_{L^2}\\\nonumber
			&\qquad+  \left(\int_{t_m}^{t_{m+1}} (B(u(s)) - B(u^m_h))\, dW(s),  \varepsilon_u^m\right)\\\nonumber
			&\leq  2\left\|\int_{t_m}^{t_{m+1}} (B(u(s)) - B(u(t_{m})))\, dW(s)\right\|^2_{L^2}  + \frac14\|\varepsilon_u^{m+1} - \varepsilon_u^{m}\|^2_{L^2} \\\nonumber
			&\qquad+2\left\|(B(u(t_m)) - B(u_h^m))\Delta W_m\right\|^2_{L^2} - 2k\|B(u(t_m)) - B(u_h^m)\|^2_{L^2} + 2kc_B^2\|e_u^{m}\|^2_{L^2}\\\nonumber
			&\qquad+  \left(\int_{t_m}^{t_{m+1}} (B(u(s)) - B(u_h^m))\, dW(s),  \varepsilon_u^m\right)\\\nonumber
			&\leq  2\left\|\int_{t_m}^{t_{m+1}} (B(u(s)) - B(u(t_{m})))\, dW(s)\right\|^2_{L^2}  + \frac14\|\varepsilon_u^{m+1} - \varepsilon_u^{m}\|^2_{L^2} \\\nonumber
			&\qquad+2\left\|(B(u(t_m)) - B(u_h^m))\Delta W_m\right\|^2_{L^2} - 2k\|B(u(t_m)) - B(u_h^m)\|^2_{L^2} \\\nonumber
			&\qquad+ 2kc_B^2\|\varepsilon_u^{m}\|^2_{L^2} + 2kc_B^2C h^{4}\|u(t_m)\|^2_{H^{2}}\\\nonumber
			&\qquad+  \left(\int_{t_m}^{t_{m+1}} (B(u(s)) - B(u^m_h))\, dW(s),  \varepsilon_u^m\right).
		\end{align*}
		
		Now, collecting and substituting all of the estimates on $I, II, ..., V$ into \eqref{eq_3.28} and absorbing the like terms to the left-hand side of \eqref{eq_3.28}, we obtain
		\begin{align}\label{eq_3.14}
			&\frac12\bigl[\|\varepsilon^{m+1}_u\|^2_{L^2} - \|\varepsilon_u^m\|^2_{L^2}\bigr] +\frac14\|\varepsilon_u^{m+1} - \varepsilon_u^m\|^2_{L^2}+ \frac{10 k}{16}\|\nab\varepsilon_{u}^{m+1}\|^2_{L^2}\\\nonumber
			&\leq Ck h^2 \|u(t_{m+1})\|^2_{H^2} + Ckh^4\|u(t_{m+1})\|^2_{L^4}
			\|u(t_m)\|^2_{H^2}  +  Ck\|g(u(t_{m+1})) - g(\cP_{u}u(t_{m+1}))\|^2_{L^2} \\\nonumber
			&\qquad+ Ckh^3\|u(t_{m+1})\|^2_{L^4}\|\vsigma(t_{m+1})\|^2_{H^2} + Ck h^3\|u_h^{m}\|^2_{L^2}\|u(t_{m+1})\|^2_{H^2} + 2kCc_B^2h^4\|u(t_m)\|^2_{H^2} \\\nonumber
			&\qquad + C\int_{t_m}^{t_{m+1}}\|\nab(u(t_{m+1}) - u(s))\|^2_{L^2}\, ds \\\nonumber
			&\qquad+ C\int_{t_m}^{t_{m+1}}[\|u(s)\|^2_{L^4}\|\vsigma(s) - \vsigma(t_{m+1})\|^2_{L^4} + \|u(s) - u(t_{m+1})\|^2_{L^4}\|\vsigma(t_{m+1})\|^2_{L^4}]\, ds \\\nonumber
			&\qquad + Ck\|u(t_{m+1})\|^2_{L^4}
			\|u(t_{m+1})-u(t_m)\|^2_{L^2}
			\\\nonumber
			&\qquad
			+8\chi^2k^{\frac12}
			\left(
			\int_{t_m}^{t_{m+1}}
			\|u(t_{m+1})-u(s)\|^4_{L^4}\,ds
			\right)^{\frac12}
			\|\varepsilon_u^m\|^2_{L^2} \\\nonumber
			&\qquad + C\int_{t_{m}}^{t_{m+1}}\|g(u(s)) - g(u(t_{m+1}))\|^2_{L^2}\, ds + 2\left\|\int_{t_m}^{t_{m+1}} (B(u(s)) - B(u(t_{m})))\, dW(s)\right\|^2_{L^2}\\\nonumber
			&\qquad + 4\chi^2
			\left(
			\int_{t_m}^{t_{m+1}}
			\|u(s)\|_{L^{q+1}}^{q+1}\,ds
			+(1+|D|)k
			\right)
			\|\varepsilon_u^m\|^2_{L^2} \\\nonumber
			&\qquad+  \frac{ k}{16} \|\varepsilon_{u}^{m+1}\|^2_{L^2}+ 4C_L^4\chi^2k\|u_h^m\|^{2}_{L^{2}}\|\varepsilon_{u}^{m+1}\|^2_{L^2}\\\nonumber
			&\qquad + k \|\varepsilon_{u}^{m+1}\|^2_{L^2} + 2kc_B^2 \|\varepsilon_{u}^{m+1}\|^2_{L^2}\\\nonumber
			&\qquad+ 2\left\|(B(u(t_m)) - B(u_h^m))\Delta W_m\right\|^2_{L^2} - 2k\|B(u(t_m)) - B(u_h^m)\|^2_{L^2} \\\nonumber
			&\qquad+ \left(\int_{t_m}^{t_{m+1}} (B(u(s)) - B(u^m_h))\, dW(s),  \varepsilon_u^m\right).
		\end{align}
		
		Applying the summation $\sum_{m=0}^{\ell}$ for any $0\leq \ell <M$ to \eqref{eq_3.14}, we obtain
		\begin{align}\label{eq3.15}
			&	\|\varepsilon_u^{\ell+1}\|^2_{L^2} + \frac12\sum_{m=0}^{\ell}\|\varepsilon_u^{m+1} - \varepsilon_u^m\|^2_{L^2} +  \frac{5k}{4}\sum_{m=0}^{\ell} \| \nab\varepsilon_u^{m+1}\|^2_{L^2} \\\nonumber
			&\leq F_{\ell} + M_{\ell} + \sum_{m=0}^{\ell} G_m\|\varepsilon_u^{m}\|^2_{L^2},
		\end{align}
		where 
		\begin{align*}
			F_{\ell} &:= 2\sum_{m=0}^{\ell}\Biggl[Ck h^2 \|u(t_{m+1})\|^2_{H^2} + Ckh^4\|u(t_{m+1})\|^2_{L^4}
			\|u(t_m)\|^2_{H^2}  +  Ck\|g(u(t_{m+1})) - g(\cP_{u}u(t_{m+1}))\|^2_{L^2} \\\nonumber
			&\qquad+ Ckh^3\|u(t_{m+1})\|^2_{L^4}\|\vsigma(t_{m+1})\|^2_{H^2} + Ck h^3\|u_h^{m}\|^2_{L^2}\|u(t_{m+1})\|^2_{H^2} + 2kCc_B^2h^4\|u(t_m)\|^2_{H^2} \\\nonumber
			&\qquad + C\int_{t_m}^{t_{m+1}}\|\nab(u(t_{m+1}) - u(s))\|^2_{L^2}\, ds + Ck\|u(t_{m+1})\|^2_{L^4}
			\|u(t_{m+1})-u(t_m)\|^2_{L^2}\\\nonumber
			&\qquad+ C\int_{t_m}^{t_{m+1}}[\|u(s)\|^2_{L^4}\|\vsigma(s) - \vsigma(t_{m+1})\|^2_{L^4} + \|u(s) - u(t_{m+1})\|^2_{L^4}\|\vsigma(t_{m+1})\|^2_{L^4}]\, ds 
			\\\nonumber
			&\qquad
			+8\chi^2k^{\frac12}
			\left(
			\int_{t_m}^{t_{m+1}}
			\|u(t_{m+1})-u(s)\|^4_{L^4}\,ds
			\right)^{\frac12}
			\|\varepsilon_u^m\|^2_{L^2} \\\nonumber
			&\qquad + C\int_{t_{m}}^{t_{m+1}}\|g(u(s)) - g(u(t_{m+1}))\|^2_{L^2}\, ds + 2\left\|\int_{t_m}^{t_{m+1}} (B(u(s)) - B(u(t_{m})))\, dW(s)\right\|^2_{L^2}\Biggr] \\\nonumber
			&\qquad+ k \left(\frac{17}{16}+ 4C_L^4\chi^2\|u_h^{\ell}\|^{2}_{L^{2}}  \right)\|\varepsilon_u^{\ell+1}\|^2_{L^2} ,\\\nonumber
			M_{\ell} &:= \sum_{m=0}^{\ell} Z_{m},\\\nonumber
			Z_m&:= 4\left\|(B(u(t_m)) - B(u_h^m))\Delta W_m\right\|^2_{L^2} - 4k\|B(u(t_m)) - B(u_h^m)\|^2_{L^2} \\\nonumber
			&\qquad+ 4 \left(\int_{t_m}^{t_{m+1}} (B(u(s)) - B(u^m_h))\, dW(s),  \varepsilon_u^m\right),\\\nonumber
			G_m&:=  2k\left(\frac{17}{16} + 2c_B^2 + 4\chi^2(1+|D|) + 4C_L^4\chi^2 \|u^m_h\|^{2}_{L^{2}} \right)+ 8\chi^2\int_{t_m}^{t_{m+1}}\|u(s)\|^{q+1}_{L^{q+1}}\, ds.
		\end{align*}
		
		Now, we move to apply the stochastic Gronwall inequality to \eqref{eq3.15}. To do that, we suppose that $\{M_{\ell};\, \ell \ge 0\}$ is a martingale (this will be verified at the end of the proof). Applying the stochastic Gronwall inequality \eqref{ineq2.4} with 
		$\alpha = \frac{2}{1+r}$, $\beta = \frac{2}{1-r}$, and $0 < r < 1$ 
		to \eqref{eq3.15}, we obtain
		
	\begin{align}\label{eq3.13}
		&\left(\mE\left[\sup_{1\leq \ell\leq M}\|\varepsilon_u^{\ell}\|_{L^2}^{2r}\right]\right)^{\frac{1}{2r}}+\left(\mE\left[\left(k\sum_{m=0}^{\ell}\|\nab\varepsilon_u^{m+1}\|_{L^2}^2\right)^r\right]\right)^{\frac{1}{2r}}
		\\\nonumber
		&\leq \left(1+\frac{1}{1-\alpha r}\right)^{\frac{1}{2\alpha r}}\left(\mE\left[\exp\left(\beta r\sum_{m=0}^{M-1}G_m\right)\right]\right)^{\frac{1}{2\beta r}}\left(\mE\left[\sup_{0\leq\ell<M}F_\ell\right]\right)^{\frac12}.
	\end{align}
	
	We next estimate the exponential factor on the right-hand side of \eqref{eq3.13}. Recalling that $\beta=2/(1-r)$, we have
	\begin{align*}
		\mE\left[\exp\left(\frac{2r}{1-r}\sum_{m=0}^{M-1}G_m\right)\right]
		&=\exp\left(\frac{2r}{1-r}\sum_{m=0}^{M-1}2k\left(\frac{17}{16}+2c_B^2+4\chi^2(1+|D|)\right)\right)\\
		&\quad\times\mE\left[\exp\left(\frac{16r}{1-r}C_L^4\chi^2k\sum_{m=0}^{M-1}\|u_h^m\|_{L^{2}}^{2}\right)\exp\left(\frac{16\chi^2r}{1-r}\int_0^T\|u(s)\|_{L^{q+1}}^{q+1}\,ds\right)\right].
	\end{align*}
	
	Let $C_G:=\frac{17}{16}+2c_B^2+4\chi^2(1+|D|)$. Then, since $Mk=T$, we have
	\begin{align}\label{eq:exp_G}
		&\mE\left[\exp\left(\frac{2r}{1-r}\sum_{m=0}^{M-1}G_m\right)\right]\\\nonumber
		&\quad=\exp\left(\frac{4r}{1-r}C_GT\right)\mE\left[\exp\left(\frac{16r}{1-r}C_L^4\chi^2k\sum_{m=0}^{M-1}\|u_h^m\|_{L^{2}}^{2}\right)\exp\left(\frac{16\chi^2r}{1-r}\int_0^T\|u(s)\|_{L^{q+1}}^{q+1}\,ds\right)\right].
	\end{align}
	
	By the Cauchy--Schwarz inequality, we have
	\begin{align}\label{eq:exp_G_CS}
		&\mE\left[\exp\left(\frac{2r}{1-r}\sum_{m=0}^{M-1}G_m\right)\right]\\\nonumber
		&\quad\leq\exp\left(\frac{4r}{1-r}C_GT\right)\left\{\mE\left[\exp\left(\frac{32r}{1-r}C_L^4\chi^2k\sum_{m=0}^{M-1}\|u_h^m\|_{L^{2}}^{2}\right)\right]\right\}^{\frac12}\left\{\mE\left[\exp\left(\frac{32\chi^2r}{1-r}\int_0^T\|u(s)\|_{L^{q+1}}^{q+1}\,ds\right)\right]\right\}^{\frac12}.
	\end{align}

	By the discrete Jensen inequality and Lemma \ref{lemma_exp_fem_stability}, we obtain
	\begin{align}\label{eq:exp_G_discrete}
		\mE\left[
		\exp\left(
		\frac{32r}{1-r}C_L^4\chi^2k
		\sum_{m=1}^{M-1}\|u_h^m\|_{L^2}^2
		\right)
		\right] &\leq
		\frac{1}{M-1}\sum_{m=1}^{M-1}
		\mE\left[
		\exp\left(
		\frac{32r}{1-r}C_L^4\chi^2T
		\|u_h^m\|_{L^2}^2
		\right)
		\right]\\\nonumber
		&\leq
		\max_{1\leq m\leq M-1}
		\mE\left[
		\exp\left(
		\frac{32r}{1-r}C_L^4\chi^2T
		\|u_h^m\|_{L^2}^2
		\right)
		\right] \leq C_{exp,2},
	\end{align}
	while Lemma~\ref{lemma_expo_moment_H} gives
	\[
	\mE\left[\exp\left(\frac{32\chi^2r}{1-r}\int_0^T\|u(s)\|_{L^{q+1}}^{q+1}\,ds\right)\right]\leq C_{\exp,1}.
	\]
	
	 Combining \eqref{eq:exp_G_CS} and \eqref{eq:exp_G_discrete}, we obtain
	\begin{align}\label{eq:exp_G_bound}
		\mE\left[\exp\left(\frac{2r}{1-r}\sum_{m=0}^{M-1}G_m\right)\right]
		&\leq\exp\left(\frac{4r}{1-r}C_GT\right)C_{exp,2}^{\frac12}C_{\exp,1}^{\frac12}\leq C,
	\end{align}
	where $C>0$ is independent of $h$ and $k$.

	Next, we estimate $\Bigl(\mE\Bigl[\sup_{0\leq\ell<M}F_\ell\Bigr]\Bigr)^{\frac12}$. Using the definition of $F_\ell$, we have
	\begin{align*}
		\mE\Bigl[\sup_{0\leq\ell<M}F_\ell\Bigr]
		&\leq Ck\sum_{m=0}^{M-1}\mE\Biggl[h^2\|u(t_{m+1})\|_{H^2}^2+h^4\|u(t_{m+1})\|_{L^4}^2\|u(t_m)\|_{H^2}^2\\\nonumber
		&\qquad+\|g(u(t_{m+1}))-g(\cP_u u(t_{m+1}))\|_{L^2}^2\\\nonumber
		&\qquad+h^3\|u(t_{m+1})\|_{L^4}^2\|\vsigma(t_{m+1})\|_{H^2}^2+h^3\|u_h^m\|_{L^2}^2\|u(t_{m+1})\|_{H^2}^2+2c_B^2h^4\|u(t_m)\|_{H^2}^2\Biggr]\\\nonumber
		&\quad+C\sum_{m=0}^{M-1}\mE\Biggl[\int_{t_m}^{t_{m+1}}\|\nab(u(t_{m+1})-u(s))\|_{L^2}^2\,ds+Ck\|u(t_{m+1})\|_{L^4}^2\|u(t_{m+1})-u(t_m)\|_{L^2}^2\\\nonumber
		&\qquad+\int_{t_m}^{t_{m+1}}\Bigl[\|u(s)\|_{L^4}^2\|\vsigma(s)-\vsigma(t_{m+1})\|_{L^4}^2+\|u(s)-u(t_{m+1})\|_{L^4}^2\|\vsigma(t_{m+1})\|_{L^4}^2\Bigr]\,ds\\\nonumber
		&\qquad+k^{\frac12}\left(\int_{t_m}^{t_{m+1}}\|u(t_{m+1})-u(s)\|_{L^4}^4\,ds\right)^{\frac12}\|\varepsilon_u^m\|_{L^2}^2\\\nonumber
		&\qquad+\int_{t_m}^{t_{m+1}}\|g(u(s))-g(u(t_{m+1}))\|_{L^2}^2\,ds+\left\|\int_{t_m}^{t_{m+1}}(B(u(s))-B(u(t_m)))\,dW(s)\right\|_{L^2}^2\Biggr]\\\nonumber
		&\quad+k\,\mE\left[\sup_{0\leq\ell<M}\left(\frac{17}{16}+4C_L^4\chi^2\|u_h^\ell\|_{L^{2}}^{2}\right)\|\varepsilon_u^{\ell+1}\|_{L^2}^2\right]\\\nonumber
		&:=Q_1+Q_2+Q_3.
	\end{align*}
	
Now, we estimate $Q_1, Q_2$ and $Q_3$. First using the Cauchy-Schwarz inequality, Lemma \ref{Stability_PDE}, and Lemma \ref{lemma_fem_stability}, we obtain
	\begin{align*}
		Q_1&\leq CT\Bigl(C_{6,1}h^2+C_{3,1}C_{6,2}h^4+ C_{4,q-1}C_{6,2}h^2+C_{3,1}C_{4,2}h^3+c_B^2C_{6,1}h^4 \Bigr)\\\nonumber
		&\qquad+ Ch^3k\sum_{m=0}^{M-1}\mE\left[\|u_h^m\|_{L^2}^2\|u(t_{m+1})\|_{H^2}^2\right]\\\nonumber
		&\leq CT\Bigl(C_{6,1}h^2+C_{3,1}C_{6,2}h^4+ C_{4,q-1}C_{6,2}h^2+C_{3,1}C_{4,2}h^3+c_B^2C_{6,1}h^4 \Bigr)\\\nonumber
		&\qquad+ CT\bar{C}_{1,2}C_{6,2}h^3\\\nonumber
		&:= C_{Q_1}h^2,
	\end{align*}
	where, in particular, the nonlinear projection term is estimated by
	\begin{align*}
			\|g(u)-g(\cP_u u)\|_{L^2}^2&\leq C\left(1+\|u\|_{L^{2q}}^{2(q-1)}+\|\cP_u u\|_{L^{2q}}^{2(q-1)}\right)\|u-\cP_u u\|_{L^{2q}}^2\\\nonumber
		&\leq C\left(1+\|u\|_{L^{2q}}^{2(q-1)}\right)h^2 \|u\|^2_{H^2}.
	\end{align*}

	Next, using Lemma~\ref{Lemma_Holder}, Lemma \ref{Stability_PDE}, the Cauchy-Schwarz inequality, the It\^o isometry, and Assumption~\eqref{Assump_Lipschitz}, we obtain
	\begin{align*}
		Q_2
		&\leq CT\Bigl[\tilde{C}_{2,1}k + C_{3,1}\tilde{C}_{1,2}k + C_{3,1}\tilde{C}_{1,2} k + C_{1,2}\tilde{C}_{2,2}k \\\nonumber
		&\qquad+ (C_{1,2} + \tilde{C})\tilde{C}_{2,2}k + (1+ C_{4,q-1})\tilde{C}_{2,2}k + \tilde{C}_{1,1}k\Bigr]\\\nonumber
		&:= C_{Q_2}k,
	\end{align*}
	where the nonlinear reaction term is controlled by
\begin{align*}
		\|g(u)-g(v)\|_{L^2}^2&\leq C\left(1+\|u\|_{L^{2q}}^{2(q-1)}+\|v\|_{L^{2q}}^{2(q-1)}\right)\|u-v\|_{L^{2q}}^2\\\nonumber
		&\leq C\left(1+\|u\|_{H^1}^{2(q-1)}+\|v\|_{H^1}^{2(q-1)}\right)\|u-v\|_{H^1}^2.
\end{align*}

	Finally, to estimate $Q_3$, we use the Cauchy--Schwarz inequality together with Lemma \ref{Stability_PDE} and Lemma \ref{lemma_fem_stability}. We obtain
	\begin{align*}
		Q_3&\leq k\frac{17}{16}\mE\left[\sup_{0\leq\ell<M}\|\varepsilon_u^{\ell+1}\|_{L^2}^2\right] +4C_L^4\chi^2k\left\{\mE\left[\sup_{0\leq\ell<M}\|u_h^\ell\|_{L^{2}}^{4}\right]\right\}^{\frac12}\left\{\mE\left[\sup_{0\leq\ell<M}\|\varepsilon_u^{\ell+1}\|_{L^2}^4\right]\right\}^{\frac12}\\\nonumber
		&\leq (C_{1,1} + \bar{C}_{1,1})k + 4C_L^^4\chi^2k \tilde{C}_{1,2}(C_{1,2} + \bar{C}_{1,2}):= C_3k.
	\end{align*}
	
	Consequently,
\begin{align}\label{eq3.52}
	\mE\left[\sup_{0\leq\ell<M}F_\ell\right]\leq C_{Q_1}h^2+(C_{Q_2}+C_{Q_3})k.
\end{align}

	Substituting this estimate together with \eqref{eq:exp_G_bound} into \eqref{eq3.13}, we arrive at
	\[
	\left(\mE\left[\sup_{1\leq\ell\leq M}\|\varepsilon_u^\ell\|_{L^2}^{2r}\right]\right)^{\frac{1}{2r}}+\left(\mE\left[\left(k\sum_{m=0}^{M-1}\|\nab\varepsilon_u^{m+1}\|_{L^2}^2\right)^r\right]\right)^{\frac{1}{2r}}\leq \widehat C\left(h+k^{\frac12}\right),
	\]
	where $\widehat C>0$ is independent of $h$ and $k$.
		
		We now establish that the sequence $\{M_\ell\}$ defines a martingale. First, we show that $M_\ell$ is integrable. Using the triangle inequality, It\^o's isometry, Assumption~\eqref{Assump_Lipschitz}, and the Burkholder--Davis--Gundy inequality, we obtain
		\begin{align}\label{eq_3.16}
			\mE[|M_\ell|]
			&\leq 4\sum_{m=0}^{\ell}\mE\left[\|B(u(t_m))-B(u_h^m)\|_{L^2}^2|\Delta W_m|^2\right]+4k\sum_{m=0}^{\ell}\mE\left[\|B(u(t_m))-B(u_h^m)\|_{L^2}^2\right]\\\nonumber
			&\qquad+4\mE\left[\left|\sum_{m=0}^{\ell}\left(\int_{t_m}^{t_{m+1}}(B(u(s))-B(u_h^m))\,dW(s),\varepsilon_u^m\right)\right|\right]\\\nonumber
			&\leq 8c_B^2k\sum_{m=0}^{\ell}\mE\left[\|u(t_m)-u_h^m\|_{L^2}^2\right]+4\mE\left[\left(\sum_{m=0}^{\ell}\int_{t_m}^{t_{m+1}}\|B(u(s))-B(u_h^m)\|_{L^2}^2\|\varepsilon_u^m\|_{L^2}^2\,ds\right)^{\frac12}\right]\\\nonumber
			&<\infty.
		\end{align}
		
		Moreover, for every $0\leq m\leq M-1$, since $u(t_m)$, $u_h^m$, and $\varepsilon_u^m$ are $\mathcal F_{t_m}$-measurable and $\Delta W_m$ is independent of $\mathcal F_{t_m}$ with $\mE[|\Delta W_m|^2\mid\mathcal F_{t_m}]=k$, we have
		\begin{align*}
			\mE[Z_m\mid\mathcal F_{t_m}]
			&=4\|B(u(t_m))-B(u_h^m)\|_{L^2}^2\mE[|\Delta W_m|^2\mid\mathcal F_{t_m}]-4k\|B(u(t_m))-B(u_h^m)\|_{L^2}^2\\\nonumber
			&\qquad+4\mE\left[\left(\int_{t_m}^{t_{m+1}}(B(u(s))-B(u_h^m))\,dW(s),\varepsilon_u^m\right)\Bigm|\mathcal F_{t_m}\right]\\\nonumber
			&=4k\|B(u(t_m))-B(u_h^m)\|_{L^2}^2-4k\|B(u(t_m))-B(u_h^m)\|_{L^2}^2+0=0.
		\end{align*}
		
		Therefore, for any $0\leq n<\ell\leq M-1$, using $M_\ell=M_n+\sum_{m=n+1}^{\ell}Z_m$ and the tower property of conditional expectation, we obtain
		\begin{align}\label{eq3.17}
			\mE[M_\ell\mid\mathcal F_{t_{n+1}}]
			&=M_n+\sum_{m=n+1}^{\ell}\mE[Z_m\mid\mathcal F_{t_{n+1}}]\\\nonumber
			&=M_n+\sum_{m=n+1}^{\ell}\mE\left[\mE[Z_m\mid\mathcal F_{t_m}]\mid\mathcal F_{t_{n+1}}\right]=M_n.
		\end{align}
		Thus, $\{M_\ell\}_{\ell=0}^{M-1}$ is a martingale with respect to the discrete filtration $\{\mathcal F_{t_{\ell+1}}\}_{\ell=0}^{M-1}$.
		
		Finally, the estimates \eqref{estimate_sigma_fem} and \eqref{estimate_c_fem} follow directly from \eqref{eq3.32}, the triangle inequality, the interpolation estimates \eqref{ineq_interpolation}, and \eqref{estimate_u_fem}. 
		
		The proof is complete.

	\end{proof}
	
	\subsection{High moment error estimates}
	The error estimate established in Theorem \ref{Thm_global_error} provides  a strong convergence in the $L_{\omega}^{r}L_t^{\infty}L_x^2$- and $L_{\omega}^{r}L_t^{2}H_x^1$-norm for $0<r<2$.  We note that these sub-quadratic moment estimates are consequences of using the stochastic Gronwall's inequality \eqref{Stochastic_Gronwall}. However, we show below, which is the goal of this subsection,  that such a limitation can be overcome by a bootstrap argument to obtain arbitrarily high-order moment estimates, which in turn implies strong convergence in the $L_{\omega}^{r}L_t^{\infty}L_x^2$-norm for all $0<r<\infty$.   
\begin{theorem}\label{Thm_higher_moment_error}
	Let $(u,\vsigma,v)$ and $(u_h^m,\vsigma_h^m,v_h^m)$ be the variational solution of \eqref{Weak_formulation} and the approximate solution generated by the Main Algorithm, respectively. Assume that $B$ satisfies \eqref{Assump_Lipschitz},
	\eqref{Assump_LinearGrowth}, and \eqref{Assump_bounded}. 
	Suppose that $u_0\in L^{2^{p+1}}(\Omega;H^2(D))$ for all $p \in \mathbb{N}$ and that there exists $\gamma>0$ such that
	\[
	\mE\left[\exp\left(\gamma\|u_0\|_{L^2}^2
	\right)\right]<\infty.
	\]
	Let $0<r<1$. Then, there exists a constants $k_0>0$ and $\widehat{C}_2>0$, independent of $h$ and $k$, such that for every $0<k \leq k_0$,
	\begin{align}\label{estimate_u_fem_higher}
		\left(
		\mE\left[
		\max_{1\leq m\leq M}
		\|u(t_m)-u_h^m\|_{L^2}^{2^pr}
		\right]
		\right)^{\frac{1}{2^pr}}
		\leq \widehat{C}_2\left(k^{\frac12}+h\right).
	\end{align}
\end{theorem}
\begin{proof} To prove this higher moment error estimate, we apply the boostrap technique from \cite{nguyen2026fully,nguyen2026finite,feng2024high,feng2025full}. Notice that the case $p=1$ was proved in Theorem \ref{Thm_global_error}.   We begin with proving  \eqref{estimate_u_fem_higher} for $p=2$ and then complete the proof by the induction argument. 
	Multiplying \eqref{eq_3.14} by $\|\varepsilon_{u}^{m+1}\|^2_{L^2}$, we obtain
\begin{align}\label{eq_3.15}
	&\frac14\left[\|\varepsilon_u^{m+1}\|_{L^2}^4-\|\varepsilon_u^m\|_{L^2}^4\right]
	+\frac14\left[\|\varepsilon_u^{m+1}\|_{L^2}^2-\|\varepsilon_u^m\|_{L^2}^2\right]^2\\\nonumber
	&\qquad+\frac14\|\varepsilon_u^{m+1}-\varepsilon_u^m\|_{L^2}^2
	\|\varepsilon_u^{m+1}\|_{L^2}^2
	+\frac{10k}{16}\|\nab\varepsilon_u^{m+1}\|_{L^2}^2
	\|\varepsilon_u^{m+1}\|_{L^2}^2\\\nonumber
	&\leq \Biggl\{Ckh^2\|u(t_{m+1})\|_{H^2}^2
	\|\varepsilon_u^{m+1}\|_{L^2}^2
	+Ckh^4\|u(t_{m+1})\|_{L^4}^2
	\|u(t_m)\|_{H^2}^2
	\|\varepsilon_u^{m+1}\|_{L^2}^2\\\nonumber
	&\qquad+Ck\|g(u(t_{m+1}))-g(\cP_u u(t_{m+1}))\|_{L^2}^2
	\|\varepsilon_u^{m+1}\|_{L^2}^2\\\nonumber
	&\qquad+Ckh^3\|u(t_{m+1})\|_{L^4}^2
	\|\vsigma(t_{m+1})\|_{H^2}^2
	\|\varepsilon_u^{m+1}\|_{L^2}^2\\\nonumber
	&\qquad+Ckh^3\|u_h^m\|_{L^2}^2
	\|u(t_{m+1})\|_{H^2}^2
	\|\varepsilon_u^{m+1}\|_{L^2}^2
	+2kCc_B^2h^4\|u(t_m)\|_{H^2}^2
	\|\varepsilon_u^{m+1}\|_{L^2}^2\\\nonumber
	&\qquad+C\left(\int_{t_m}^{t_{m+1}}
	\|\nab(u(t_{m+1})-u(s))\|_{L^2}^2\,ds\right)
	\|\varepsilon_u^{m+1}\|_{L^2}^2\\\nonumber
	&\qquad+C\left(\int_{t_m}^{t_{m+1}}
	\left[
	\|u(s)\|_{L^4}^2\|\vsigma(s)-\vsigma(t_{m+1})\|_{L^4}^2
	+\|u(s)-u(t_{m+1})\|_{L^4}^2
	\|\vsigma(t_{m+1})\|_{L^4}^2
	\right]ds\right)
	\|\varepsilon_u^{m+1}\|_{L^2}^2\\\nonumber
	&\qquad+Ck\|u(t_{m+1})\|_{L^4}^2
	\|u(t_{m+1})-u(t_m)\|_{L^2}^2
	\|\varepsilon_u^{m+1}\|_{L^2}^2\\\nonumber
	&\qquad+8\chi^2k^{\frac12}
	\left(
	\int_{t_m}^{t_{m+1}}
	\|u(t_{m+1})-u(s)\|_{L^4}^4\,ds
	\right)^{\frac12}
	\|\varepsilon_u^m\|_{L^2}^2
	\|\varepsilon_u^{m+1}\|_{L^2}^2\\\nonumber
	&\qquad+C\left(\int_{t_m}^{t_{m+1}}
	\|g(u(s))-g(u(t_{m+1}))\|_{L^2}^2\,ds\right)
	\|\varepsilon_u^{m+1}\|_{L^2}^2\\\nonumber
	&\qquad+2\left\|
	\int_{t_m}^{t_{m+1}}
	(B(u(s))-B(u(t_m)))\,dW(s)
	\right\|_{L^2}^2
	\|\varepsilon_u^{m+1}\|_{L^2}^2\Biggr\}\\\nonumber
	&\qquad+\Biggl\{4\chi^2
	\left(
	\int_{t_m}^{t_{m+1}}\|u(s)\|_{L^{q+1}}^{q+1}\,ds
	+(1+|D|)k
	\right)
	\|\varepsilon_u^m\|_{L^2}^2
	\|\varepsilon_u^{m+1}\|_{L^2}^2\\\nonumber
	&\qquad+\frac{k}{16}\|\varepsilon_u^{m+1}\|_{L^2}^4
	+4C_L^4\chi^2k
	\|u_h^m\|_{L^{2}}^{2}
	\|\varepsilon_u^{m+1}\|_{L^2}^4\\\nonumber
	&\qquad+k\|\varepsilon_u^{m+1}\|_{L^2}^4
	+2kc_B^2\|\varepsilon_u^{m+1}\|_{L^2}^4 -2k\|B(u(t_m))-B(u_h^m)\|_{L^2}^2
	\|\varepsilon_u^{m+1}\|_{L^2}^2\Biggr\}\\\nonumber
	&\qquad+\Biggl\{2\left\|
	(B(u(t_m))-B(u_h^m))\Delta W_m
	\right\|_{L^2}^2
	\|\varepsilon_u^{m+1}\|_{L^2}^2\\\nonumber
	&\qquad+\left(
	\int_{t_m}^{t_{m+1}}
	(B(u(s))-B(u_h^m))\,dW(s),
	\varepsilon_u^m
	\right)
	\|\varepsilon_u^{m+1}\|_{L^2}^2\Biggr\}\\\nonumber
	&:= I + II + III,
\end{align}
where $I, II$ and $III$ on the right-hand side of \eqref{eq_3.15} are seperated by using the curly brackets.
	Next, using the Young inequality to many terms of $I$ on the right-hand side of \eqref{eq_3.15},  we obtain
	\begin{align}
		I &\leq k\|\varepsilon_{u}^{m+1}\|^4_{L^2} + Ckh^4\|u(t_{m+1})\|_{H^2}^4
		+Ckh^4\|u(t_{m+1})\|_{L^4}^2
		\|u(t_m)\|_{H^2}^2
		\|\varepsilon_u^{m+1}\|_{L^2}^2\\\nonumber
		&\qquad+Ck\|g(u(t_{m+1}))-g(\cP_u u(t_{m+1}))\|_{L^2}^4
		\\\nonumber
		&\qquad+Ckh^6\|u(t_{m+1})\|_{L^4}^4
		\|\vsigma(t_{m+1})\|_{H^2}^4
		\\\nonumber
		&\qquad+Ckh^6\|u_h^m\|_{L^2}^4
		\|u(t_{m+1})\|_{H^2}^4
		+2kCc_B^2h^4\|u(t_m)\|_{H^2}^2
		\|\varepsilon_u^{m+1}\|_{L^2}^2\\\nonumber
		&\qquad+C\int_{t_m}^{t_{m+1}}
		\|\nab(u(t_{m+1})-u(s))\|_{L^2}^4\,ds
	\\\nonumber
		&\qquad+C\left(\int_{t_m}^{t_{m+1}}
		\left[
		\|u(s)\|_{L^4}^4\|\vsigma(s)-\vsigma(t_{m+1})\|_{L^4}^4
		+\|u(s)-u(t_{m+1})\|_{L^4}^4
		\|\vsigma(t_{m+1})\|_{L^4}^4
		\right]ds\right)
		\\\nonumber
		&\qquad+Ck\|u(t_{m+1})\|_{L^4}^4
		\|u(t_{m+1})-u(t_m)\|_{L^2}^4
	\\\nonumber
		&\qquad+C\chi^4
		\int_{t_m}^{t_{m+1}}
		\|u(t_{m+1})-u(s)\|_{L^4}^4\,ds
		\|\varepsilon_u^m\|_{L^2}^4
\\\nonumber
		&\qquad+C\int_{t_m}^{t_{m+1}}
		\|g(u(s))-g(u(t_{m+1}))\|_{L^2}^4\,ds
		\\\nonumber
		&\qquad+\frac{C}{k}\left\|
		\int_{t_m}^{t_{m+1}}
		(B(u(s))-B(u(t_m)))\,dW(s)
		\right\|_{L^2}^4.
	\end{align}
	
	Similarly, using the Young inequality, we also obtain
	\begin{align}
		II &\leq 2\chi^4
		\left(
		\int_{t_m}^{t_{m+1}}\|u(s)\|_{L^{q+1}}^{q+1}\,ds
		+(1+|D|)k
		\right)
		\|\varepsilon_u^m\|_{L^2}^4 + 
		2k\|\varepsilon_u^{m+1}\|_{L^2}^4\\\nonumber
		&\qquad
		+4C_L^4\chi^2k
		\|u_h^m\|_{L^{2}}^{2}
		\|\varepsilon_u^{m+1}\|_{L^2}^4 +\frac{17k}{16}\|\varepsilon_u^{m+1}\|_{L^2}^4
		+2kc_B^2\|\varepsilon_u^{m+1}\|_{L^2}^4 \\\nonumber
		&= 2\chi^4
		\left(
		\int_{t_m}^{t_{m+1}}\|u(s)\|_{L^{q+1}}^{q+1}\,ds
		+(1+|D|)k
		\right)
		\|\varepsilon_u^m\|_{L^2}^4 \\\nonumber
		&\qquad
		+4C_L^4\chi^2k
		\|u_h^m\|_{L^{2}}^{2}
		\|\varepsilon_u^{m+1}\|_{L^2}^4 +k \left(\frac{49}{16} + 2c_B^2\right)\|\varepsilon_u^{m+1}\|_{L^2}^4.
	\end{align}
	
	To control $III$, we add and subtract the terms $\|\varepsilon_{u}^m\|^2_{L^2}$ and then using the Young inequality, we obtain
	\begin{align}\label{eq3.56}
		III &= 2\left\|
		(B(u(t_m))-B(u_h^m))\Delta W_m
		\right\|_{L^2}^2
		\left(\|\varepsilon_u^{m+1}\|_{L^2}^2 - \|\varepsilon_u^{m}\|_{L^2}^2\right)\\\nonumber
		&\qquad +2\left\|
		(B(u(t_m))-B(u_h^m))\Delta W_m
		\right\|_{L^2}^2
		\|\varepsilon_u^{m}\|_{L^2}^2\\\nonumber
		&\qquad+\left(
		\int_{t_m}^{t_{m+1}}
		(B(u(s))-B(u_h^m))\,dW(s),
		\varepsilon_u^m
		\right)
		\left(\|\varepsilon_u^{m+1}\|_{L^2}^2 - \|\varepsilon_u^{m}\|_{L^2}^2\right)\\\nonumber
		&\qquad+\left(
		\int_{t_m}^{t_{m+1}}
		(B(u(s))-B(u_h^m))\,dW(s),
		\varepsilon_u^m
		\right)
		\|\varepsilon_u^{m}\|_{L^2}^2\\\nonumber
		&\leq 16\left\|
		(B(u(t_m))-B(u_h^m))\Delta W_m
		\right\|_{L^2}^4
	\\\nonumber
		&\qquad+4\left\|
		\int_{t_m}^{t_{m+1}}
		(B(u(s))-B(u_h^m))\,dW(s)\right\|^2_{L^2}
		\|\varepsilon_u^m\|^2_{L^2} + 
		\frac18\left(\|\varepsilon_u^{m+1}\|_{L^2}^2 - \|\varepsilon_u^{m}\|_{L^2}^2\right)^2\\\nonumber
		&\qquad +2\left\|
		(B(u(t_m))-B(u_h^m))\Delta W_m
		\right\|_{L^2}^2
		\|\varepsilon_u^{m}\|_{L^2}^2\\\nonumber
		&\qquad+\left(
		\int_{t_m}^{t_{m+1}}
		(B(u(s))-B(u_h^m))\,dW(s),
		\varepsilon_u^m
		\right)
		\|\varepsilon_u^{m}\|_{L^2}^2 \\\nonumber
			&\leq 16\left\|
		(B(u(t_m))-B(u_h^m))\Delta W_m
		\right\|_{L^2}^4
		\\\nonumber
		&\qquad+8\left\|
		\int_{t_m}^{t_{m+1}}
		(B(u(s))-B(u(t_m)))\,dW(s)\right\|^2_{L^2}
		\|\varepsilon_u^m\|^2_{L^2}  \\\nonumber
		&\qquad+8\left\|
		\int_{t_m}^{t_{m+1}}
		(B(u(t_m))-B(u_h^m))\,dW(s)\right\|^2_{L^2}
		\|\varepsilon_u^m\|^2_{L^2} + 
		\frac18\left(\|\varepsilon_u^{m+1}\|_{L^2}^2 - \|\varepsilon_u^{m}\|_{L^2}^2\right)^2\\\nonumber
		&\qquad +2\left\|
		(B(u(t_m))-B(u_h^m))\Delta W_m
		\right\|_{L^2}^2
		\|\varepsilon_u^{m}\|_{L^2}^2\\\nonumber
		&\qquad+\left(
		\int_{t_m}^{t_{m+1}}
		(B(u(s))-B(u_h^m))\,dW(s),
		\varepsilon_u^m
		\right)
		\|\varepsilon_u^{m}\|_{L^2}^2 \\\nonumber
		&= 16\left\|
		(B(u(t_m))-B(u_h^m))\Delta W_m
		\right\|_{L^2}^4
		 \\\nonumber
		&\qquad +10\left\|
		(B(u(t_m))-B(u_h^m))\Delta W_m
		\right\|_{L^2}^2
		\|\varepsilon_u^{m}\|_{L^2}^2\\\nonumber
		&\qquad+\left(
		\int_{t_m}^{t_{m+1}}
		(B(u(s))-B(u_h^m))\,dW(s),
		\varepsilon_u^m
		\right)
		\|\varepsilon_u^{m}\|_{L^2}^2 \\\nonumber
		&\qquad+8\left\|
		\int_{t_m}^{t_{m+1}}
		(B(u(s))-B(u(t_m)))\,dW(s)\right\|^2_{L^2}
		\|\varepsilon_u^m\|^2_{L^2}  + 
		\frac18\left(\|\varepsilon_u^{m+1}\|_{L^2}^2 - \|\varepsilon_u^{m}\|_{L^2}^2\right)^2.
	\end{align}
	
	One of the last two terms on the right-hand side of \eqref{eq3.56}, $8\left\|
	\int_{t_m}^{t_{m+1}}
	(B(u(s))-B(u(t_m)))\,dW(s)\right\|^2_{L^2}
	\|\varepsilon_u^m\|^2_{L^2} $, will be controlled by using the H\"older inequality from Lemma \ref{Lemma_Holder}, while $\frac18\left(\|\varepsilon_u^{m+1}\|_{L^2}^2 - \|\varepsilon_u^{m}\|_{L^2}^2\right)^2$ is absorbed to the left-hand side. In addition, the first three terms on the right-hand of \eqref{eq3.56} will be modified to generate the martingale processes in the using of the stochastic Gronwall inequality later. To the end, we can continue to analyze $III$ by adding and subtracting appropriate terms as follows:
	\begin{align}\label{eq3.57}
		III &\leq 16\left\|
		(B(u(t_m))-B(u_h^m))\Delta W_m
		\right\|_{L^2}^4 - 48 k^2\|B(u(t_{m})) - B(u_h^m)\|^4_{L^2}
		\\\nonumber
		&\qquad +10\left\|
		(B(u(t_m))-B(u_h^m))\Delta W_m
		\right\|_{L^2}^2
		\|\varepsilon_u^{m}\|_{L^2}^2 - 10k\|B(u(t_m)) -B(u_h^m)\|^2_{L^2}\|\varepsilon_{u}^m\|^2_{L^2}\\\nonumber
		&\qquad+\left(
		\int_{t_m}^{t_{m+1}}
		(B(u(s))-B(u_h^m))\,dW(s),
		\varepsilon_u^m
		\right)
		\|\varepsilon_u^{m}\|_{L^2}^2 \\\nonumber
		&\qquad + 48 k^2\|B(u(t_{m})) - B(u_h^m)\|^4_{L^2} + 10k\|B(u(t_m)) -B(u_h^m)\|^2_{L^2}\|\varepsilon_{u}^m\|^2_{L^2}\\\nonumber
		&\qquad+8\left\|
		\int_{t_m}^{t_{m+1}}
		(B(u(s))-B(u(t_m)))\,dW(s)\right\|^2_{L^2}
		\|\varepsilon_u^m\|^2_{L^2}  + 
		\frac18\left(\|\varepsilon_u^{m+1}\|_{L^2}^2 - \|\varepsilon_u^{m}\|_{L^2}^2\right)^2\\\nonumber
		&\leq 16\left\|
		(B(u(t_m))-B(u_h^m))\Delta W_m
		\right\|_{L^2}^4 - 48 k^2\|B(u(t_{m})) - B(u_h^m)\|^4_{L^2}
		\\\nonumber
		&\qquad +10\left\|
		(B(u(t_m))-B(u_h^m))\Delta W_m
		\right\|_{L^2}^2
		\|\varepsilon_u^{m}\|_{L^2}^2 - 10k\|B(u(t_m)) -B(u_h^m)\|^2_{L^2}\|\varepsilon_{u}^m\|^2_{L^2}\\\nonumber
		&\qquad+\left(
		\int_{t_m}^{t_{m+1}}
		(B(u(s))-B(u_h^m))\,dW(s),
		\varepsilon_u^m
		\right)
		\|\varepsilon_u^{m}\|_{L^2}^2 \\\nonumber
		&\qquad + 48 k^2C_B^4\|\varepsilon_{u}^m\|^4_{L^2} + Ck^2h^8\|u(t_m)\|^4_{H^2}+ 10kC_B^2\|\varepsilon_{u}^m\|^4_{L^2} + Ckh^4\|u(t_m)\|^2_{H^2}\|\varepsilon_{u}^m\|^2_{L^2}\\\nonumber
		&\qquad+8\left\|
		\int_{t_m}^{t_{m+1}}
		(B(u(s))-B(u(t_m)))\,dW(s)\right\|^2_{L^2}
		\|\varepsilon_u^m\|^2_{L^2}  + 
		\frac18\left(\|\varepsilon_u^{m+1}\|_{L^2}^2 - \|\varepsilon_u^{m}\|_{L^2}^2\right)^2.
	\end{align}
	
Substituting the above estimates of $I$, $II$, and $III$ into \eqref{eq_3.15}, absorbing the like terms on the right-hand side into the left-hand side, summing over $m=0,\ldots,\ell$ for any $0\leq\ell<M$, and multiplying the resulting inequality by $4$, we obtain
\begin{align}\label{eq3.59}
	&\|\varepsilon_u^{\ell+1}\|_{L^2}^4
	+\frac12\sum_{m=0}^{\ell}
	\left(\|\varepsilon_u^{m+1}\|_{L^2}^2-\|\varepsilon_u^m\|_{L^2}^2\right)^2\\\nonumber
	&\qquad+\sum_{m=0}^{\ell}
	\|\varepsilon_u^{m+1}-\varepsilon_u^m\|_{L^2}^2
	\|\varepsilon_u^{m+1}\|_{L^2}^2
	+\frac{5k}{2}\sum_{m=0}^{\ell}
	\|\nab\varepsilon_u^{m+1}\|_{L^2}^2
	\|\varepsilon_u^{m+1}\|_{L^2}^2\\\nonumber
	&\leq \mathcal R_\ell+\mathcal M_\ell
	+\sum_{m=0}^{\ell}\mathcal G_m\|\varepsilon_u^m\|_{L^2}^4
	+\sum_{m=0}^{\ell}\widetilde{\mathcal G}_m
	\|\varepsilon_u^{m+1}\|_{L^2}^4,
\end{align}
where
\begin{align*}
	\mathcal R_\ell
	&:=4\sum_{m=0}^{\ell}\Biggl[
	Ckh^4\|u(t_{m+1})\|_{H^2}^4
	+Ckh^4\|u(t_{m+1})\|_{L^4}^2
	\|u(t_m)\|_{H^2}^2
	\|\varepsilon_u^{m+1}\|_{L^2}^2\\
	&\qquad+Ck\|g(u(t_{m+1}))-g(\cP_u u(t_{m+1}))\|_{L^2}^4
	+Ckh^6\|u(t_{m+1})\|_{L^4}^4
	\|\vsigma(t_{m+1})\|_{H^2}^4\\
	&\qquad+Ckh^6\|u_h^m\|_{L^2}^4
	\|u(t_{m+1})\|_{H^2}^4
	+Ckh^4\|u(t_m)\|_{H^2}^2
	\|\varepsilon_u^{m+1}\|_{L^2}^2\\
	&\qquad+C\int_{t_m}^{t_{m+1}}
	\|\nab(u(t_{m+1})-u(s))\|_{L^2}^4\,ds+C\int_{t_m}^{t_{m+1}}
	\|u(s)\|_{L^4}^4
	\|\vsigma(s)-\vsigma(t_{m+1})\|_{L^4}^4\,ds\\
	&\qquad+C\int_{t_m}^{t_{m+1}}
	\|u(s)-u(t_{m+1})\|_{L^4}^4
	\|\vsigma(t_{m+1})\|_{L^4}^4\,ds\\
	&\qquad+Ck\|u(t_{m+1})\|_{L^4}^4
	\|u(t_{m+1})-u(t_m)\|_{L^2}^4\\
	&\qquad+4C\chi^4\int_{t_m}^{t_{m+1}}
	\|u(t_{m+1})-u(s)\|_{L^4}^4\,ds\|\varepsilon_{u}^m\|^4_{L^2}+C\int_{t_m}^{t_{m+1}}
	\|g(u(s))-g(u(t_{m+1}))\|_{L^2}^4\,ds\\
	&\qquad+\frac{C}{k}\left\|
	\int_{t_m}^{t_{m+1}}
	(B(u(s))-B(u(t_m)))\,dW(s)
	\right\|_{L^2}^4+Ck^2h^8\|u(t_m)\|_{H^2}^4
	+Ckh^4\|u(t_m)\|_{H^2}^2
	\|\varepsilon_u^m\|_{L^2}^2\\
	&\qquad+8\left\|
	\int_{t_m}^{t_{m+1}}
	(B(u(s))-B(u(t_m)))\,dW(s)
	\right\|_{L^2}^2
	\|\varepsilon_u^m\|_{L^2}^2
	\Biggr],
\end{align*}
\begin{align*}
	\mathcal M_\ell
	&:=4\sum_{m=0}^{\ell}
	\left(\Delta M_{1,m}+\Delta M_{2,m}+\Delta M_{3,m}\right),
\end{align*}
where
\begin{align*}
	\Delta M_{1,m}
	&:=16\left\|
	(B(u(t_m))-B(u_h^m))\Delta W_m
	\right\|_{L^2}^4
	-48k^2\|B(u(t_m))-B(u_h^m)\|_{L^2}^4,\\
	\Delta M_{2,m}
	&:=10\left[
	\left\|(B(u(t_m))-B(u_h^m))\Delta W_m\right\|_{L^2}^2
	-k\|B(u(t_m))-B(u_h^m)\|_{L^2}^2
	\right]\|\varepsilon_u^m\|_{L^2}^2,\\
	\Delta M_{3,m}
	&:=\left(
	\int_{t_m}^{t_{m+1}}
	(B(u(s))-B(u_h^m))\,dW(s),
	\varepsilon_u^m
	\right)\|\varepsilon_u^m\|_{L^2}^2,
\end{align*}
and
\begin{align*}
	\mathcal G_m
	&:=8\chi^4\left(
	\int_{t_m}^{t_{m+1}}\|u(s)\|_{L^{q+1}}^{q+1}\,ds
	+(1+|D|)k\right)
	+4Ck^2C_B^4+40kC_B^2,
\end{align*}
while
\begin{align*}
	\widetilde{\mathcal G}_m
	&:=4k\left(
	\frac{65}{16}+2c_B^2
	+4C_L^4\chi^2
	\|u_h^m\|_{L^{2}}^{2}
	\right).
\end{align*}

To apply the stochastic Gronwall inequality, it remains to rewrite the last term on the right-hand side of \eqref{eq3.59} in terms of $\|\varepsilon_u^m\|_{L^2}^4$. To this end, we first separate the last term in the summation and add and subtract $\widetilde{\mathcal G}_\ell\|\varepsilon_u^\ell\|_{L^2}^4$. We obtain
\begin{align}\label{eq3.60}
	\sum_{m=0}^{\ell}\widetilde{\mathcal G}_m
	\|\varepsilon_u^{m+1}\|_{L^2}^4
	&=\sum_{m=0}^{\ell-1}\widetilde{\mathcal G}_m
	\|\varepsilon_u^{m+1}\|_{L^2}^4
	+\widetilde{\mathcal G}_\ell
	\|\varepsilon_u^{\ell+1}\|_{L^2}^4\\\nonumber
	&=\sum_{m=0}^{\ell-1}\widetilde{\mathcal G}_m
	\|\varepsilon_u^{m+1}\|_{L^2}^4
	+\widetilde{\mathcal G}_\ell
	\|\varepsilon_u^\ell\|_{L^2}^4\\\nonumber
	&\qquad+\widetilde{\mathcal G}_\ell
	\left(
	\|\varepsilon_u^{\ell+1}\|_{L^2}^2
	+\|\varepsilon_u^\ell\|_{L^2}^2
	\right)
	\left(
	\|\varepsilon_u^{\ell+1}\|_{L^2}^2
	-\|\varepsilon_u^\ell\|_{L^2}^2
	\right).
\end{align}
Using Young's inequality, we further obtain
\begin{align}\label{eq3.61}
	\sum_{m=0}^{\ell}\widetilde{\mathcal G}_m
	\|\varepsilon_u^{m+1}\|_{L^2}^4
	&\leq
	\sum_{m=0}^{\ell}
	\left(
	\widetilde{\mathcal G}_{m-1}
	+\widetilde{\mathcal G}_m
	\right)
	\|\varepsilon_u^m\|_{L^2}^4\\\nonumber
	&\qquad+\widetilde{\mathcal G}_\ell^2
	\left(
	\|\varepsilon_u^{\ell+1}\|_{L^2}^2
	+\|\varepsilon_u^\ell\|_{L^2}^2
	\right)^2\\\nonumber
	&\qquad+\frac14
	\left(
	\|\varepsilon_u^{\ell+1}\|_{L^2}^2
	-\|\varepsilon_u^\ell\|_{L^2}^2
	\right)^2,
\end{align}
where we set $\widetilde{\mathcal G}_{-1}:=0$. The last term on the right-hand side of \eqref{eq3.61} can be absorbed into the second term on the left-hand side of \eqref{eq3.59}. Therefore, substituting \eqref{eq3.61} into \eqref{eq3.59}, we obtain
\begin{align}\label{eq3.62}
	&\|\varepsilon_u^{\ell+1}\|_{L^2}^4
	+\frac14\sum_{m=0}^{\ell}
	\left(
	\|\varepsilon_u^{m+1}\|_{L^2}^2
	-\|\varepsilon_u^m\|_{L^2}^2
	\right)^2\\\nonumber
	&\qquad+\sum_{m=0}^{\ell}
	\|\varepsilon_u^{m+1}-\varepsilon_u^m\|_{L^2}^2
	\|\varepsilon_u^{m+1}\|_{L^2}^2
	+\frac{5k}{2}\sum_{m=0}^{\ell}
	\|\nab\varepsilon_u^{m+1}\|_{L^2}^2
	\|\varepsilon_u^{m+1}\|_{L^2}^2\\\nonumber
	&\leq
	\widetilde{\mathcal R}_\ell
	+\mathcal M_\ell
	+\sum_{m=0}^{\ell}
	\mathcal X_m\|\varepsilon_u^m\|_{L^2}^4,
\end{align}
where
\begin{align}
	\widetilde{\mathcal R}_\ell
	&:=\mathcal R_\ell
	+\widetilde{\mathcal G}_\ell^2
	\left(
	\|\varepsilon_u^{\ell+1}\|_{L^2}^2
	+\|\varepsilon_u^\ell\|_{L^2}^2
	\right)^2, \label{eq3.63}\\
	\mathcal X_m
	&:=\mathcal G_m
	+\widetilde{\mathcal G}_{m-1}
	+\widetilde{\mathcal G}_m. \label{eq3.64}
\end{align}
More explicitly, 
\begin{align}\label{eq3.65}
	\mathcal X_m
	&=8\chi^4\left(
	\int_{t_m}^{t_{m+1}}
	\|u(s)\|_{L^{q+1}}^{q+1}\,ds
	+(1+|D|)k
	\right)
	+4Ck^2C_B^4+40kC_B^2\\\nonumber
	&\qquad+4k\left(
	\frac{65}{16}+2c_B^2
	+4C_L^4\chi^2
	\|u_h^{m-1}\|_{L^{2}}^{2}
	\right)+4k\left(
	\frac{65}{16}+2c_B^2
	+4C_L^4\chi^2
	\|u_h^m\|_{L^{2}}^{2}
	\right).
\end{align}

Hence, \eqref{eq3.62} is in the form required for the stochastic Gronwall inequality, with $Y_m:=\|\varepsilon_u^m\|_{L^2}^4.$
Moreover, by the Gaussian moment identities
\[
\mE\left[(\Delta W_m)^2\mid\mathcal F_{t_m}\right]=k,
\qquad
\mE\left[(\Delta W_m)^4\mid\mathcal F_{t_m}\right]=3k^2,
\]
and the martingale property of the It\^o integral, we have
\[
\mE\left[\Delta M_{i,m}\mid\mathcal F_{t_m}\right]=0,
\qquad i=1,2,3.
\]
Therefore, $\{\mathcal M_\ell\}_{\ell=0}^{M-1}$ is a martingale with respect to the discrete filtration $\{\mathcal F_{t_{\ell+1}}\}_{\ell=0}^{M-1}$.
	
	Now, we are ready to applying the stochastic Gronwall inequality \eqref{Stochastic_Gronwall} with
	\[
	\alpha=\frac{2}{1+r},\qquad
	\beta=\frac{2}{1-r},\qquad
	0<r<1,
	\] to \eqref{eq3.62}, we obtain
	\begin{align}\label{eq3.66}
		&\left(
		\mE\left[
		\sup_{0\leq\ell<M}
		\|\varepsilon_u^{\ell+1}\|_{L^2}^{4r}
		\right]
		\right)^{\frac{1}{4r}}\\\nonumber
		&\quad\leq
		\left(
		1+\frac{1}{1-\alpha r}
		\right)^{\frac{1}{4\alpha r}}
		\left[
		\mE\left[
		\exp\left(
		\beta r\sum_{m=0}^{M-1}\mathcal X_m
		\right)
		\right]
		\right]^{\frac{1}{4\beta r}}
		\left(
		\mE\left[
		\sup_{0\leq\ell<M}\widetilde{\mathcal R}_\ell
		\right]
		\right)^{\frac14}.
	\end{align}
	
Next, using Lemma~\ref{lemma_expo_moment_H} and Lemma~\ref{lemma_exp_fem_stability}, we estimate the exponential factor on the right-hand side of \eqref{eq3.66}. From \eqref{eq3.65}, we have
\begin{align}
	\sum_{m=0}^{M-1}\mathcal X_m
	&\leq 8\chi^4\int_0^T\|u(s)\|_{L^{q+1}}^{q+1}\,ds
	+32C_L^4\chi^2k\sum_{m=0}^{M-1}\|u_h^m\|_{L^2}^2
	+C_XT,
\end{align}
where $C_X>0$ is independent of $h$ and $k$. Hence, recalling that $\beta=2/(1-r)$, we obtain
\begin{align}\label{eq3.67}
	&\mE\left[
	\exp\left(
	\beta r\sum_{m=0}^{M-1}\mathcal X_m
	\right)
	\right]\leq
	\exp\left(\frac{2r}{1-r}C_XT\right)
	\mE\left[
	\exp\left(
	\frac{16r}{1-r}\chi^4
	\int_0^T\|u(s)\|_{L^{q+1}}^{q+1}\,ds
	\right)\right.\\\nonumber
	&\hspace{5cm}\left.
	\times
	\exp\left(
	\frac{64r}{1-r}C_L^4\chi^2k
	\sum_{m=0}^{M-1}\|u_h^m\|_{L^2}^2
	\right)
	\right].
\end{align}
By the Cauchy--Schwarz inequality,
\begin{align}\label{eq3.68}
	\mE\left[
	\exp\left(
	\beta r\sum_{m=0}^{M-1}\mathcal X_m
	\right)
	\right] &\leq
	\exp\left(\frac{2r}{1-r}C_XT\right)
	\left\{
	\mE\left[
	\exp\left(
	\frac{32r}{1-r}\chi^4
	\int_0^T\|u(s)\|_{L^{q+1}}^{q+1}\,ds
	\right)
	\right]
	\right\}^{\frac12}\\\nonumber
	&\qquad\times
	\left\{
	\mE\left[
	\exp\left(
	\frac{128r}{1-r}C_L^4\chi^2k
	\sum_{m=0}^{M-1}\|u_h^m\|_{L^2}^2
	\right)
	\right]
	\right\}^{\frac12}\\\nonumber
	&\leq \exp\left(\frac{2r}{1-r}C_XT\right) C_{exp,1}^{\frac12}C_{exp,2}^{\frac12}.
\end{align}

Next, we estimate $\mE\left[
\sup_{0\leq\ell<M}\widetilde{\mathcal R}_\ell
\right]$.
To this end, we divide it into two groups according to the spatial and temporal discretization errors:
\begin{align}\label{eq3.74}
\mE\left[
\sup_{0\leq\ell<M}\widetilde{\mathcal R}_\ell
\right]	\leq \mathcal R_{M,h}+\mathcal R_{M,k},
\end{align}
where each tem is defined and estimated as follows:
\begin{align}
	\mathcal R_{M,h}
	&:=4\sum_{m=0}^{M-1}\Biggl[
	Ckh^4\|u(t_{m+1})\|_{H^2}^4+Ckh^4\|u(t_{m+1})\|_{L^4}^2
	\|u(t_m)\|_{H^2}^2
	\|\varepsilon_u^{m+1}\|_{L^2}^2\\\nonumber
	&\qquad+Ck\|g(u(t_{m+1}))
	-g(\cP_u u(t_{m+1}))\|_{L^2}^4+Ckh^6\|u(t_{m+1})\|_{L^4}^4
	\|\vsigma(t_{m+1})\|_{H^2}^4\\\nonumber
	&\qquad+Ckh^6\|u_h^m\|_{L^2}^4
	\|u(t_{m+1})\|_{H^2}^4+Ckh^4\|u(t_m)\|_{H^2}^2
	\|\varepsilon_u^{m+1}\|_{L^2}^2\\\nonumber
	&\qquad+Ck^2h^8\|u(t_m)\|_{H^2}^4
	+Ckh^4\|u(t_m)\|_{H^2}^2
	\|\varepsilon_u^m\|_{L^2}^2
	\Biggr]\\\nonumber
	&\leq CTh^4\Bigl[C_{6,2} + C_{3,2}C_{6,4}(C_{1,2} + \bar{C}_{1,2}) + h^4 C_{4,4(q-1)}C_{6,8} \\\nonumber
	&\qquad+ h^2 C_{3,2}C_{4,4} + h^2 \bar{C}_{1,3}C_{6,4} + C_{6,4}(C_{1,2} + \bar{C}_{1,2}) + kh^4 C_{6,2} + C_{6,2}(C_{1,2} + \bar{C}_{1,2})\Bigr]\\\nonumber
	&:= C_{\mathcal R_{M,h}}h^4,
\end{align}
where we have used Lemma \ref{lemma_fem_highmoment_stability} and Lemma \ref{Stability_PDE}. Similarly, using Lemma \ref{Lemma_Holder}, Lemma \ref{Stability_PDE}, Lemma \ref{lemma_fem_highmoment_stability}, the Buckholder-Davis-Gundy inequality, the It\^o isometry and the Lipschitz assumption \eqref{Assump_Lipschitz}, we also have
\begin{align*}
	\mathcal R_{M,k}
	&:=4\sum_{m=0}^{M-1}\Biggl[
	C\int_{t_m}^{t_{m+1}}
	\|\nab(u(t_{m+1})-u(s))\|_{L^2}^4\,ds+C\int_{t_m}^{t_{m+1}}
	\|u(s)\|_{L^4}^4
	\|\vsigma(s)-\vsigma(t_{m+1})\|_{L^4}^4\,ds\\\nonumber
	&\qquad+C\int_{t_m}^{t_{m+1}}
	\|u(s)-u(t_{m+1})\|_{L^4}^4
	\|\vsigma(t_{m+1})\|_{L^4}^4\,ds+Ck\|u(t_{m+1})\|_{L^4}^4
	\|u(t_{m+1})-u(t_m)\|_{L^2}^4\\\nonumber
	&\qquad+4C\chi^4\int_{t_m}^{t_{m+1}}
	\|u(t_{m+1})-u(s)\|_{L^4}^4\,ds
	\|\varepsilon_u^m\|_{L^2}^4+C\int_{t_m}^{t_{m+1}}
	\|g(u(s))-g(u(t_{m+1}))\|_{L^2}^4\,ds\\\nonumber
	&\qquad+\frac{C}{k}\left\|
	\int_{t_m}^{t_{m+1}}
	(B(u(s))-B(u(t_m)))\,dW(s)
	\right\|_{L^2}^4
	\Biggr]+\widetilde{\mathcal G}_\ell^2
	\left(
	\|\varepsilon_u^{\ell+1}\|_{L^2}^2
	+\|\varepsilon_u^\ell\|_{L^2}^2
	\right)^2\\\nonumber
	&\leq Ck^2\Bigl[\tilde{C}_{1,2} + C_{3,2}\tilde{C}_{2,4} + C_{1,4}\tilde{C}_{2,4} + C_{3,2}\tilde{C}_{1,4} + (C_{1,4} + \bar{C}_{1,4})\tilde{C}_{2,4} \\\nonumber
	&\qquad+ C_{4,4(q-1)}\tilde{C}_{2,4}+ C_B^4\tilde{C}_{1,2} + C(1+\bar{C}_{1,4})(C_{1,4} + \bar{C}_{1,3}) \Bigr]\\\nonumber
	&:=C_{	\mathcal R_{M,k}}k^2.
\end{align*}
Consequently,
\begin{align}\label{eq3.75}
	\mE\left[
	\sup_{0\leq\ell<M}\widetilde{\mathcal R}_\ell
	\right]
	&\leq C_{	\mathcal R_{M,k}}k^2 + C_{	\mathcal R_{M,h}}h^4.
\end{align}

With this, we completes the proof of \eqref{estimate_u_fem_higher} for the case $p=2$.

Lastly, to prove \eqref{estimate_u_fem_higher} for $p>2$, we use an inductive argument as it was done in the proof of \cite[Theorem 3.3 ]{vo2025high}, starting with multiplying the error inequality \eqref{eq_3.14} by $\|\varepsilon_{u}^{m+1}\|^{2^j}_{L^2}$ (assuming \eqref{estimate_u_fem_higher} holds for $p=j$),
and utilize the technique of leveraging the stochastic Gronwall's inequality as shown for the case $p=2$. We skip the derivation to save space and leave it to the interested reader to complete. 

\end{proof}

\medskip

Now, we state the second-moment error estimates for $u_h^m, \vsigma_h^m, v_h^m$.
\begin{corollary}
	Let $(u,\vsigma,v)$ and $(u_h^m,\vsigma_h^m,v_h^m)$ be the variational solution of \eqref{Weak_formulation} and the approximate solution generated by the Main Algorithm, respectively. Assume that $B$ satisfies \eqref{Assump_Lipschitz}, \eqref{Assump_LinearGrowth}, and \eqref{Assump_bounded}. Assume that $u_0\in L^{16}(\Omega;H^2(D))$ and that
	\[
	\mE\left[
	\exp\left(
	\gamma\|u_0\|_{L^2}^2
	\right)
	\right]<\infty
	\qquad\text{for some }\gamma>0.
	\]
	Then there exist constants $k_0>0$ and $\widehat{C}_3>0$, such that, for every $0<k\leq k_0$,
	\begin{align}
		\label{estimate_second_u_fem}
		\left(
		\mE\left[
		\max_{1\leq m\leq M}
		\|u(t_m)-u_h^m\|_{L^2}^{2}
		\right]
		\right)^{\frac{1}{2}}
		+
		\left(
		\mE\left[
		k\sum_{m=1}^{M}
		\|\nab(u(t_m)-u_h^m)\|_{L^2}^2
		\right]
		\right)^{\frac{1}{2}}
		&\leq \widehat{C}_3\left(k^{\frac12}+h\right),\\
		\label{estimate_second_sigma_fem}
		\left(
		\mE\left[
		\max_{1\leq m\leq M}
		\|\vsigma(t_m)-\vsigma_h^m\|_{H^1}^{2}
		\right]
		\right)^{\frac{1}{2}}
		&\leq \widehat{C}_3\left(k^{\frac12}+h\right),\\
		\label{estimate_second_c_fem}
		\left(
		\mE\left[
		\max_{1\leq m\leq M}
		\|v(t_m)-v_h^m\|_{L^2}^{2}
		\right]
		\right)^{\frac{1}{2}}
		&\leq \widehat{C}_3\left(k^{\frac12}+h\right).
	\end{align}
\end{corollary}
\begin{proof}
	First, using Theorem \ref{Thm_higher_moment_error} with $r =\frac{1}{2}$ and $p=2$, we obtain $\eqref{estimate_second_u_fem}_1$. Then, using this $\eqref{estimate_second_u_fem}_1$ and Lemma \ref{Lemma_Holder}, we can derive the error estimates \eqref{estimate_second_sigma_fem} and \eqref{estimate_second_c_fem} easily by using the error equations of the scheme \eqref{eq_discrete_sigma} and \eqref{eq_discrete_c} and \eqref{eq_sigma} and \eqref{eq_c}. 
	
	Next, to prove the energy error $\eqref{estimate_second_u_fem}_2$, we use $\eqref{estimate_second_u_fem}_1$ and the inequality \eqref{eq3.15} as follows: Taking the expectation to \eqref{eq3.15}, we obtain
	\begin{align}
		\mE\left[k\sum_{m=0}^{M-1}\|\nab\varepsilon_u^{m+1}\|^2_{L^2}\right] \leq \mE\left[F_M\right] + \mE\left[\sum_{m=0}^{M-1} G_m\|\varepsilon_{u}^m\|^2_{L^2}\right].
	\end{align}
	
	Using the estimate \eqref{eq3.52}, we have 
	\begin{align}
		\mE\left[F_M\right] &\leq C_{Q_1}h^2 + (C_{Q_2} + C_{Q_3})k.
	\end{align}
	
	next, using the Cauchy-Schwarz inequality, Lemma \ref{Stability_PDE}, Lemma \ref{lemma_fem_highmoment_stability} and Theorem \ref{Thm_higher_moment_error} with $r=\frac12$ and $p =4$, we get
	\begin{align}
		\mE\left[\sum_{m=0}^{M-1} G_m\|\varepsilon_{u}^m\|^2_{L^2}\right] &\leq \sum_{m=0}^{M-1} \left(\mE\left[G_m^2\right]\right)^{\frac12}\left(\mE\left[\|\varepsilon_{u}^m\|^4_{L^2}\right]\right)^{\frac12} \leq C(1 + C_{4,2} + \bar{C}_{1,2}) \widehat{C}_2(k + h^2).
	\end{align}
	
	This completes the proof.
	
\end{proof}

\section{Conclusion}\label{sec5}

In this paper, we have established a rigorous strong convergence analysis for a fully discrete splitting mixed finite element method applied to a stochastic Keller--Segel system with polynomial logistic growth and multiplicative noise. We first proved the well-posedness of the variational solution and derived high-moment regularity, and exponential stability estimates for the exact and numerical solutions. By combining these estimates with a stochastic Gronwall framework, we successfully controlled the nonlinear chemotactic coupling and multiplicative stochastic forcing directly in expectation, without relying on localization arguments. Consequently, we established strong convergence in high-moment norms with the optimal rates $O(k^{1/2} + h)$. These theoretical rates are consistent with the numerical experiments for the same splitting mixed finite element scheme reported in \cite[Section 4]{thieu2026splitting}, which provide numerical evidence of their optimality. The results provide a rigorous foundation for the strong convergence of splitting mixed finite element approximations of stochastic chemotaxis systems.

	\section*{Declarations of Funding}\,
	
	The author Liet Vo was supported by the National Science Foundation (NSF) under Grant No. DMS-2530211.
	
	\section*{Data Availability} \,
	
	Data sharing is not applicable to this article as no datasets were generated or analysed during the current study.
	
	\section*{Declaration of Conflict of Interest}\,
	
	The authors have no conflict of interest.

	\appendix
	
	\section{Useful results}\label{appendxiA}
	
	First, we state the following discrete stochastic Gronwall inequality from \cite[Theorem 1]{kruse2018discrete}, which plays a vital role in our error analysis in Section \ref{sec4}.
	
	\begin{lemma}{\cite[Theorem 1]{kruse2018discrete}}\label{Stochastic_Gronwall}
		Let $\left\{M_n\right\}_{n \in \mathbb{N}}$ be an $\left\{\mathcal{F}_n\right\}_{n \in \mathbb{N}}$-martingale satisfying $M_0=0$ on a filtered probability space $\left(\Omega, \mathcal{F},\left\{\mathcal{F}_n\right\}_{n \in \mathbb{N}}, \mathbb{P}\right)$. Let $\left\{X_n\right\}_{n \in \mathbb{N}},\left\{F_n\right\}_{n \in \mathbb{N}}$, and $\left\{G_n\right\}_{n \in \mathbb{N}}$ be sequences of nonnegative and adapted random variables with $\mathbb{E}\left[X_0\right]<\infty$ such that
		\begin{align}\label{ineq2.3}
			X_n \leq F_n+M_n+\sum_{k=0}^{n-1} G_{\ell} X_{\ell} \quad \text { for all } n \in \mathbb{N}.
		\end{align}
		Then, for any $q \in(0,1)$ and a pair of conjugate numbers $\alpha, \beta \in[1, \infty]$, i.e.,  $\frac{1}{\alpha}+\frac{1}{\beta}=1$, satisfying $q \alpha<1$,  there holds 
		\begin{align}\label{ineq2.4}
			\mathbb{E}\left[\sup _{0 \leq \ell \leq n} X_{\ell}^q\right] \leq\left(1+\frac{1}{1-\alpha q}\right)^{\frac{1}{\alpha}}\left\|\prod_{\ell=0}^{n-1}\left(1+G_{\ell}\right)^q\right\|_{L^{\beta}(\Omega)}\left(\mathbb{E}\left[\sup _{0 \leq \ell \leq n} F_{\ell}\right]\right)^q.
		\end{align}
	\end{lemma}
	
	\section{Proof of Theorem \eqref{Thm_wellposedness}}\label{appendixB}
	
	We establish the existence and uniqueness of the variational solution by using a standard compactness approach for stochastic partial differential equations. More precisely, following the framework of \cite{li2026martingale,nguyen2026finite}, we first construct a martingale solution to \eqref{eq_2.1} by means of a Galerkin approximation, uniform a priori estimates, and a compactness argument. We then establish the pathwise uniqueness of the martingale solution. Combining the existence of a martingale solution with pathwise uniqueness, we obtain the existence and uniqueness of the variational solution stated in Theorem~\ref{Thm_wellposedness}.
	
	The two main ingredients in the proof of Theorem~\ref{Thm_wellposedness} are summarized in the following lemmas.
	
\begin{lemma}\label{Lemma_martingale_solution}
Assume that $B$ satisfies \eqref{Assump_Lipschitz} and \eqref{Assump_LinearGrowth} with $m=0$, and let $u_0\in L^2(\Omega;L^2(D))$. Then there exists a martingale solution to \eqref{eq_2.1}. More precisely, there exist a stochastic basis $	(\widetilde{\Omega},\widetilde{\mathcal F},
\{\widetilde{\mathcal F}_t\}_{t\geq0},
\widetilde{\mathbb P}),$
	a real-valued $\{\widetilde{\mathcal F}_t\}_{t\geq0}$-Wiener process $\widetilde W$, and progressively measurable processes $(\widetilde u,\widetilde v)$ such that
	\begin{align}
		\widetilde u&\in L^2\left(\widetilde{\Omega};C([0,T];L^2(D))\right)
		\cap L^2\left(\widetilde{\Omega};L^2(0,T;H^1(D))\right) \cap L^{q+1}\left(\widetilde{\Omega};L^{q+1}(0,T;L^{q+1}(D))\right),\\
		\widetilde v&\in L^2\left(\widetilde{\Omega};C([0,T];H^2(D))\right),
	\end{align}
	and, for every $t\in[0,T]$, $\widetilde{\mathbb P}$-a.s.,
	\begin{subequations}\label{eq_martingale_solution}
		\begin{align}
			(\widetilde u(t),\phi)
			&=(u_0,\phi)
			-\int_0^t(\nab\widetilde u(s),\nab\phi)\,ds
			+\chi\int_0^t(\widetilde u(s)\nab\widetilde v(s),\nab\phi)\,ds\\\nonumber
			&\qquad+\int_0^t(g(\widetilde u(s)),\phi)\,ds
			+\int_0^t(B(\widetilde u(s)),\phi)\,d\widetilde W(s),\\
			(\nab\widetilde v(t),\nab\psi)+(\widetilde v(t),\psi)
			&=(\widetilde u(t),\psi),
		\end{align}
	\end{subequations}
	for all $\phi,\psi\in H^1(D)$.
\end{lemma}

\begin{proof}
	The proof follows by adapting the martingale solution construction in
	\cite[Theorem~2.1]{li2026martingale} to the present stochastic
	Keller--Segel system. The principal diffusion and chemotactic
	coupling have the same structure, while the nonlinear reaction
	term is replaced by the polynomial logistic nonlinearity
	$g(u)=u-u^q$. In particular, the dissipative property
	\[
	(g(u),u)
	=
	\|u\|_{L^2}^2-\|u\|_{L^{q+1}}^{q+1}
	\]
	provides the necessary control of the reaction term in the
	uniform energy estimates. The existence of a martingale
	solution then follows through the Galerkin approximation,
	stochastic compactness, and martingale identification
	arguments developed in \cite[Theorem~2.1]{li2026martingale},
	with the corresponding modifications for the logistic
	nonlinearity and periodic boundary conditions.
	We omit the remaining details.
\end{proof}
	
\begin{lemma}\label{Lemma_pathwise_uniqueness}
	Assume that $B$ satisfies \eqref{Assump_Lipschitz} with $m=0$. Let $(u_1,v_1)$ and $(u_2,v_2)$ be two martingale solutions of \eqref{eq_2.1} defined on the same stochastic basis, adapted to the same filtration, driven by the same Wiener process $W$, and satisfying
	\[
	u_1(0)=u_2(0)=u_0,\qquad \mP\text{-a.s.}
	\]
	Then $(u_1,v_1)$ and $(u_2,v_2)$ are indistinguishable, that is,
	\begin{align}
		\mP\left(
		u_1(t)=u_2(t)\ \text{and}\ v_1(t)=v_2(t)
		\text{ for all }t\in[0,T]
		\right)=1.
	\end{align}
	In particular, pathwise uniqueness holds for \eqref{eq_2.1}.
\end{lemma}
\begin{proof}
	The proof follows the pathwise uniqueness proof of the stochastic BBM equation \cite[Lemma B.6]{nguyen2026finite}. Let
	\[
	U:=u_1-u_2,\qquad V:=v_1-v_2.
	\]
	Subtracting the two equations in \eqref{eq_2.1}, we obtain
	\begin{align}\label{eq_pathuniq_difference}
		(dU,\phi)
		&=-(\nab U,\nab\phi)\,dt
		+\chi\bigl(u_1\nab v_1-u_2\nab v_2,\nab\phi\bigr)\,dt\\\nonumber
		&\qquad+\bigl(g(u_1)-g(u_2),\phi\bigr)\,dt
		+\bigl(B(u_1)-B(u_2),\phi\bigr)\,dW(t),
	\end{align}
	and
	\begin{align}\label{eq_pathuniq_elliptic}
		(\nab V,\nab\psi)+(V,\psi)=(U,\psi),
		\qquad \forall\,\psi\in H^1(D).
	\end{align}
	
	We first estimate the nonlinear terms in \eqref{eq_pathuniq_difference}. Since
	\[
	u_1\nab v_1-u_2\nab v_2
	=U\nab v_1+u_2\nab V,
	\]
	we have
	\begin{align}\label{eq_pathuniq_chem1}
		\left|\bigl(u_1\nab v_1-u_2\nab v_2,\nab U\bigr)\right|
		&\leq \|U\|_{L^4}\|\nab v_1\|_{L^4}\|\nab U\|_{L^2}
		+\|u_2\|_{L^4}\|\nab V\|_{L^4}\|\nab U\|_{L^2}.
	\end{align}
	By the elliptic regularity estimate \eqref{elliptic_estimate}, we have
	\begin{align}
		\|\nab v_1\|_{L^4}&\leq C\|v_1\|_{H^2} \leq C\|u_1\|_{L^2},&
		\|\nab V\|_{L^4}&\leq C\|V\|_{H^2} \leq C\|U\|_{L^2}.
	\end{align}
	Moreover, by the Ladyzhenskaya inequality,
	\[
	\|U\|_{L^4}
	\leq C\|U\|_{L^2}^{1/2}\|U\|_{H^1}^{1/2}.
	\]
	Consequently, using Young's inequality in \eqref{eq_pathuniq_chem1}, we obtain
	\begin{align}\label{eq_pathuniq_chem2}
		\chi\left|\bigl(u_1\nab v_1-u_2\nab v_2,\nab U\bigr)\right|
		&\leq \frac12\|\nab U\|_{L^2}^2
		+C^*\left(1+\|u_1\|_{L^2}^4+\|u_2\|_{L^4}^2\right)\|U\|_{L^2}^2,
	\end{align}
	for some constant $C^*>0$.
	
	Next, since $q\geq3$ is odd, the function $r\mapsto r^q$ is monotone increasing. Therefore,
	\begin{align}\label{eq_pathuniq_g}
		\bigl(g(u_1)-g(u_2),U\bigr)
		&=\|U\|_{L^2}^2-\bigl(u_1^q-u_2^q,u_1-u_2\bigr)
		\leq\|U\|_{L^2}^2.
	\end{align}
	Furthermore, by \eqref{Assump_Lipschitz} with $m=0$,
	\begin{align}\label{eq_pathuniq_B}
		\|B(u_1)-B(u_2)\|_{L^2}^2
		\leq C_B^2\|U\|_{L^2}^2.
	\end{align}
	
	We now introduce the weight
	\begin{align}\label{eq_pathuniq_weight}
		\rho(t)
		:=\exp\left\{-C^*\int_0^t
		\left(1+\|u_1(s)\|_{L^2}^4+\|u_2(s)\|_{L^4}^2\right)\,ds\right\},
	\end{align}
	where $C>0$ will be chosen sufficiently large. Notice that $\rho(t)>0$ for all $t\in[0,T]$, $\mP$-a.s., since
	\[
	u_1\in C([0,T];L^2(D)),\qquad
	u_2\in L^2(0,T;H^1(D))\hookrightarrow L^2(0,T;L^4(D)),
	\]
	$\mP$-a.s.
	
	Applying It\^o's formula to $\rho(t)\|U(t)\|_{L^2}^2$, we obtain
	\begin{align}\label{eq_pathuniq_Ito}
		d\left(\rho\|U\|_{L^2}^2\right)
		&=-C^*\rho\left(1+\|u_1\|_{L^2}^4+\|u_2\|_{L^4}^2\right)\|U\|_{L^2}^2\,dt
		-2\rho\|\nab U\|_{L^2}^2\,dt\\\nonumber
		&\qquad+2\chi\rho\bigl(u_1\nab v_1-u_2\nab v_2,\nab U\bigr)\,dt\\\nonumber
		&\qquad+2\rho\bigl(g(u_1)-g(u_2),U\bigr)\,dt
		+\rho\|B(u_1)-B(u_2)\|_{L^2}^2\,dt\\\nonumber
		&\qquad+2\rho\bigl(B(u_1)-B(u_2),U\bigr)\,dW(t).
	\end{align}
	Using \eqref{eq_pathuniq_chem2}--\eqref{eq_pathuniq_B} in \eqref{eq_pathuniq_Ito}, we obtain
	\begin{align}\label{eq_pathuniq_final}
		d\left(\rho(t)\|U(t)\|_{L^2}^2\right)
		+\rho(t)\|\nab U(t)\|_{L^2}^2\,dt
		&\leq (2+C_B^2)\rho(t) \|U(t)\|^2_{L^2}\, dt\\\nonumber
		&\qquad+
		2\rho(t)\bigl(B(u_1(t))-B(u_2(t)),U(t)\bigr)\,dW(t).
	\end{align}
	
	 Taking expectations in \eqref{eq_pathuniq_final}, the stochastic integral has zero expectation, and hence
	\begin{align}
		\mE\left[\rho(t)\|U(t)\|_{L^2}^2\right]
		+\mE\left[\int_0^t\rho(s)\|\nab U(s)\|_{L^2}^2\,ds\right]
		&\leq \mE\left[\|U(0)\|_{L^2}^2\right] \\\nonumber
		&\qquad+ (2+C_B^2)\int_{0}^t\mE\left[\rho(s) \|U(s)\|^2_{L^2}\right]\, ds.
	\end{align}
	
	Using the Gronwall inequality, we obtain
	\begin{align}
	\mE\left[\rho(t)\|U(t)\|_{L^2}^2\right]
	+\mE\left[\int_0^t\rho(s)\|\nab U(s)\|_{L^2}^2\,ds\right]
	&\leq \mE\left[\|U(0)\|_{L^2}^2\right]e^{(2+C_B^2)T}.
	\end{align}
	
	Since $u_1(0)=u_2(0)=u_0$, we have $U(0)=0$, $\mP$-a.s., and therefore
	\begin{align}
		\mE\left[\rho(t)\|U(t)\|_{L^2}^2\right]=0,
		\qquad t\in[0,T].
	\end{align}
	Since $\rho(t)>0$, it follows that
	\[
	U(t)=0,\qquad \mP\text{-a.s.},
	\]
	for every $t\in[0,T]$, which implies that
	\begin{align}
		\mP\left(u_1(t)=u_2(t)\text{ for all }t\in[0,T]\right)=1.
	\end{align}
	
	Finally, since $U=0$, equation \eqref{eq_pathuniq_elliptic} gives
	\[
	(\nab V,\nab\psi)+(V,\psi)=0,
	\qquad\forall\,\psi\in H^1(D).
	\]
	Taking $\psi=V$, we obtain
	\[
	\|\nab V\|_{L^2}^2+\|V\|_{L^2}^2=0,
	\]
	and hence $V=0$. Therefore,
	\[
	v_1(t)=v_2(t),\qquad t\in[0,T],\quad\mP\text{-a.s.}
	\]
	This proves the pathwise uniqueness of \eqref{eq_2.1}.
\end{proof}
	
	Combining Lemmas~\ref{Lemma_martingale_solution} and~\ref{Lemma_pathwise_uniqueness}, we obtain the existence of a unique probabilistically strong variational solution $(u,v)$ to \eqref{eq_2.1}, which proves Theorem~\ref{Thm_wellposedness}.
	
	
	\bibliographystyle{abbrv}
	\bibliography{references}

@article{chen2025well,
	title={Well-posedness of stochastic chemotaxis system},
	author={Chen, Yunfeng and Zhai, Jianliang and Zhang, Tusheng},
	journal={Journal of Differential Equations},
	volume={442},
	pages={113531},
	year={2025},
	publisher={Elsevier}
}

@article{zhang2016characteristic,
	title={{Characteristic splitting mixed finite element analysis of Keller--Segel chemotaxis models}},
	author={Zhang, Jiansong and Zhu, Jiang and Zhang, Rongpei},
	journal={Applied Mathematics and Computation},
	volume={278},
	pages={33--44},
	year={2016},
	publisher={Elsevier}
}

@article{misiats2022global,
	title={{On global existence and blowup of solutions of stochastic Keller--Segel type equation}},
	author={Misiats, Oleksandr and Stanzhytskyi, Oleksandr and Topaloglu, Ihsan},
	journal={Nonlinear Differential Equations and Applications NoDEA},
	volume={29},
	number={1},
	pages={3},
	year={2022},
	publisher={Springer}
}

@article{vo2025analysis,
	title={{Analysis of fully discrete Crank-Nicolson finite element methods for a stochastic Keller-Segel chemotaxis system with gradient-type multiplicative noise}},
	author={Vo, Liet},
	journal={arXiv preprint arXiv:2507.15103},
	year={2025}
}

@article{feng2024high,
	title={{High moment and pathwise error estimates for fully discrete mixed finite element approximattions of stochastic Navier-Stokes equations with additive noise}},
	author={Feng, Xiaobing and Vo, Liet},
	journal={Communications in Computational Physics},
	volume={36},
	number={3},
	pages={821-849},
	year={2024},
	publisher={Global Science Press}
}

@article{nguyen2026finite,
	title={{Finite element approximations of the stochastic Benjamin-Bona-Mahony equation with multiplicative noise}},
	author={Nguyen, Hung D and Thieu, Thoa and Vo, Liet},
	journal={arXiv preprint arXiv:2603.08496},
	year={2026}
}

@article{nguyen2026fully,
	title={{Fully discrete finite element methods for the stochastic Kuramoto--Sivashinsky equation with multiplicative noise}},
	author={Nguyen, Hung D and Vo, Liet},
	journal={Stochastics and Partial Differential Equations: Analysis and Computations},
	pages={1--63},
	year={2026},
	publisher={Springer}
}

@article{keller1970initiation,
	title={Initiation of slime mold aggregation viewed as an instability},
	author={Keller, Evelyn F. and Segel, Lee A.},
	journal={Journal of Theoretical Biology},
	volume={26},
	number={3},
	pages={399--415},
	year={1970}
}

@article{keller1971model,
	title={Model for chemotaxis},
	author={Keller, Evelyn F and Segel, Lee A},
	journal={Journal of theoretical biology},
	volume={30},
	number={2},
	pages={225--234},
	year={1971},
	publisher={Elsevier}
}

@article{mimura1996aggregating,
	title={Aggregating pattern dynamics in a chemotaxis model including growth},
	author={Mimura, Masayasu and Tsujikawa, Tohru},
	journal={Physica A: Statistical Mechanics and its Applications},
	volume={230},
	number={3-4},
	pages={499--543},
	year={1996},
	publisher={Elsevier}
}

@article{tello2007chemotaxis,
	title={A chemotaxis system with logistic source},
	author={Tello, J Ignacio and Winkler, Michael},
	journal={Communications in Partial Differential Equations},
	volume={32},
	number={6},
	pages={849--877},
	year={2007},
	publisher={Taylor \& Francis}
}

@article{lankeit2017generalized,
	title={{A generalized solution concept for the Keller--Segel system with logarithmic sensitivity: global solvability for large nonradial data}},
	author={Lankeit, Johannes and Winkler, Michael},
	journal={Nonlinear Differential Equations and Applications NoDEA},
	volume={24},
	number={4},
	pages={49},
	year={2017},
	publisher={Springer}
}

@article{cao2014boundedness,
	title={{Boundedness in a quasilinear parabolic--parabolic Keller--Segel system with logistic source}},
	author={Cao, Xinru},
	journal={Journal of Mathematical Analysis and Applications},
	volume={412},
	number={1},
	pages={181--188},
	year={2014},
	publisher={Elsevier}
}

@article{huang2021microscopic,
	title={{The microscopic derivation and well-posedness of the stochastic Keller--Segel equation}},
	author={Huang, Hui and Qiu, Jinniao},
	journal={Journal of nonlinear science},
	volume={31},
	number={1},
	pages={6},
	year={2021},
	publisher={Springer}
}

@article{mukherjee2026uniqueness,
	title={{Uniqueness of the stochastic Keller--Segel model in one dimension}},
	author={Mukherjee, Debopriya and Hausenblas, Erika and Tran, Thanh},
	journal={Stochastic Processes and their Applications},
	volume={200},
	pages={104999},
	year={2026},
	publisher={Elsevier}
}

@article{carelli2012rates,
	title={{Rates of convergence for discretizations of the stochastic incompressible Navier--Stokes equations}},
	author={Carelli, Erich and Prohl, Andreas},
	journal={SIAM Journal on Numerical Analysis},
	volume={50},
	number={5},
	pages={2467--2496},
	year={2012},
	publisher={SIAM}
}

@article{lord2014introduction,
	title={{An introduction to computational stochastic PDEs}},
	author={Lord, Gabriel J and Powell, Catherine E and Shardlow, Tony},
	journal={Cambridge Texts Appl. Math},
	volume={50},
	year={2014}
}

@article{thieu2026splitting,
	title={A splitting mixed finite element method for a stochastic {Keller-Segel} system with multiplicative noise},
	author={Thieu, Thoa and Vo, Liet},
	journal={arXiv preprint arXiv:2608.07798},
	year={2026}
}

@article{kruse2018discrete,
	title={A discrete stochastic {Gronwall} lemma},
	author={Kruse, Raphael and Scheutzow, Michael},
	journal={Mathematics and Computers in Simulation},
	volume={143},
	pages={149--157},
	year={2018},
	publisher={Elsevier}
}

@article{li2026martingale,
	title={Martingale Solutions to a Stochastic {Keller-Segel} System with nonlocal Source and Super-linear Noise},
	author={Li, Qian and Chen, Li and Wang, Jinhuan},
	journal={arXiv preprint arXiv:2606.11774},
	year={2026}
}

@article{feng2025full,
	title={Full moment error estimates in strong norms for numerical approximations of stochastic {Navier-Stokes} equations with multiplicative noise, {Part I}: time discretization},
	author={Feng, Xiaobing and Vo, Liet},
	journal={arXiv preprint arXiv:2510.08291},
	year={2025}
}

@article{vo2025high,
	title={High moment and pathwise error estimates for fully discrete mixed finite element approximations of the stochastic {Stokes} equations with multiplicative noise},
	author={Vo, Liet},
	journal={ESAIM: Mathematical Modelling and Numerical Analysis},
	volume={59},
	number={3},
	pages={1301--1331},
	year={2025},
	publisher={EDP Sciences}
}

@article{brzezniak2013finite,
	title={Finite-element-based discretizations of the incompressible {Navier--Stokes} equations with multiplicative random forcing},
	author={Brze{\'z}niak, Zdzislaw and Carelli, Erich and Prohl, Andreas},
	journal={IMA Journal of Numerical Analysis},
	volume={33},
	number={3},
	pages={771--824},
	year={2013},
	publisher={Oxford University Press}
}

@article{sherratt1994chemotaxis,
	title={Chemotaxis and chemokinesis in eukaryotic cells: the {Keller-Segel} equations as an approximation to a detailed model},
	author={Sherratt, Jonathan A},
	journal={Bulletin of mathematical biology},
	volume={56},
	number={1},
	pages={129--146},
	year={1994},
	publisher={Springer}
}

@article{arumugam2021keller,
	title={Keller-Segel chemotaxis models: a review},
	author={Arumugam, Gurusamy and Tyagi, Jagmohan},
	journal={Acta Applicandae Mathematicae},
	volume={171},
	number={1},
	pages={6},
	year={2021},
	publisher={Springer}
}

@article{herrero1996chemotactic,
	title={Chemotactic collapse for the {Keller-Segel} model},
	author={Herrero, Miguel A and Vel{\'a}zquez, Juan JL},
	journal={Journal of Mathematical Biology},
	volume={35},
	number={2},
	pages={177--194},
	year={1996},
	publisher={Springer}
}

\end{document}